\documentclass[10pt,reqno]{amsart}
\usepackage{graphicx} 

\usepackage{cancel, esint}
\usepackage{enumitem}
\usepackage{amsmath,amssymb,amsthm,mathtools}
\usepackage{geometry}
\usepackage{xcolor}
\usepackage{hyperref}
\hypersetup{
    colorlinks,
    linkcolor={red!50!black},
    citecolor={blue!50!black},
    urlcolor={blue!80!black}
}
\usepackage{enumitem, float}
\usepackage{bm}
\usepackage{xfrac, nccmath}
\usepackage{caption}
\usepackage{parskip}
\newcommand{\R}{\mathbb{R}}

\newcommand{\cJ}{\mathcal{E}}
\newcommand{\dd}{\,\mathrm{d}}
\newcommand{\divv}{\operatorname{div}}

\newcommand{\G}{\Gamma}

\newcommand{\K}{K}

\newcommand{\J}{\mathcal{E}^\Phi}
\newcommand{\na}{\nabla}
\newcommand{\Reg}{\text{Reg}}
\newcommand{\Sing}{\text{Sing}}
\newcommand{\La}{\Lambda}
\newcommand{\Hm}{\mathcal{H}}
\newcommand{\e}{\varepsilon}

\newcommand{\Istarz}{I_{\La, \alpha, \C}^{*}}
\newcommand{\Jfm}{\mathcal{E}^\Phi_{f,\mu}}
\newcommand{\Div}{\text{div}}
\newcommand{\TT}{\tau}
\newcommand{\FL}{\Psi}
\newcommand{\E}{\mathcal E}
\newcommand{\Lo}{\mathcal L}
\newcommand{\db}{\bar{\delta}}
\newcommand{\C}{\Upsilon}
\newcommand{\An}{\mathcal{A}(\La)}
\newcommand{\cG}{\mathcal{G}}

\def\R {\mathbb{R}}

\def\eps{\varepsilon}

\newcommand{\II}{I_{\La, \alpha, \C}}

\def\PosS{\{u>0\}}

\def\GU{\Gamma(u)}
\def\Reg{\mathrm{Reg}}
\def\Sing{\mathrm{Sing}}
\def\SingU{\mathrm{Sing}(u)}

\newtheorem{theorem}{Theorem}[section]
\newtheorem{lemma}[theorem]{Lemma}
\newtheorem{proposition}[theorem]{Proposition}
\newtheorem{corollary}[theorem]{Corollary}

\theoremstyle{definition}
\newtheorem{definition}[theorem]{Definition}
\newtheorem{remark}[theorem]{Remark}

\def\hem{\hspace{0.5em}}
\def\vem{\vspace{0.6em}}

\numberwithin{equation}{section}

\title{Singular set estimates in quasilinear Alt-Caffarelli problems}
\author[A. De Rosa]{Antonio De Rosa}
\address{Department of Decision Sciences and BIDSA, Bocconi University, Milano, Italy}
\email{antonio.derosa@unibocconi.it}

	\author[R. Neumayer]{Robin Neumayer}
	\address{Department of Mathematical Sciences, Carnegie Mellon University, 5000 Forbes Avenue, Pittsburgh, PA 15213, United States of America }
	\email{neumayer@cmu.edu}

\author[H. Yu]{Hui Yu}
\address{Department of Mathematics, National University of Singapore, Singapore}
\email{huiyu@nus.edu.sg}

\begin{document}

\begin{abstract}
For a nonnegative, uniformly convex function $\Phi$, we study nonnegative minimizers of the quasilinear Alt-Caffarelli energy  $\int_{\Omega}\Phi(\nabla u)+\chi_{\PosS}$. Under mild regularity assumptions on $\Phi$, we show that the codimension of the singular set of the free boundary is at least $2+\eta$ for a positive $\eta$. 
With a singular minimizer in three dimensions in a forthcoming work, this result is sharp as long as integer dimensions are considered. Our proof method is flexible, and we use it to also prove a similar result for minimizers of the $p$-Laplacian counterpart of the Alt-Caffarelli energy.
\end{abstract}

\maketitle

\section{Introduction}

For a nonnegative and uniformly convex function $\Phi$, we consider the \textit{quasilinear Alt-Caffarelli energy} on a bounded domain $\Omega$ in $\R^n$, defined as 
\begin{equation}
\label{eqn: main functional}
\mathcal{E}^\Phi(u;\Omega):=\int_{\Omega}\Phi(\nabla u)+\chi_{\PosS}.
\end{equation} 
This functional appears in a variety of applications, for instance, fluid mechanics \cite{ACF1, GLS, KW1, KW2, LLWX, LSTX},   combustion theory \cite{DPS, MWo, T1}, electrostatics \cite{Ph, Tep} as well as material science \cite{F, FS,T2}. For these, it is natural to consider the \textit{one-phase} problem, that is, when the minimizer satisfies
$$
u\ge 0 \hem  \text{ in }\Omega.
$$ 
Such minimizers are obtained when we impose nonnegative boundary data.

When the {integrand} $\Phi$ is taken as $\Phi(p)=|p|^2$, our problem reduces to the {classical Alt-Caffarelli problem}, one of the most well-studied elliptic free boundary problems. 
For the classical theory, see the monographs \cite{CS, V} and 
Section~\ref{sec:aclit:classical} below. 
For some recent developments, see \cite{CFFS, DJ, DJS, EdSV, EDZ, EFeY, ESV, FeFS, FeS, FSV, FTW1, FTW2, HKM,KaWa}. 
Motivated by applications in inhomogeneous media, the quasilinear version in \eqref{eqn: main functional} has also received extensive attention, for instance, \cite{ACF2, CV,DP, LT, MW, MWo, PV05, DP, T1, T2, LN1, LN2,  DFFV, Kp, FL21}. 

\vem

Beyond its applications, the mathematical interest of \eqref{eqn: main functional} lies in the regularity properties of the minimizer $u$ as well as its {free boundary}
\begin{equation}
\label{eqn:free boundary}
\GU:=\partial\PosS\cap\Omega.
\end{equation} 
This is the interface between the {positive phase} $\PosS$ and the {zero phase} $\{u=0\}.$

Local Lipschitz regularity of minimizers and $\eps$-regularity/improvement-of-flatness 
of the free boundary $\GU$ are classical for the Alt-Caffarelli problem
\cite{AC,CS,V} and extend, under appropriate structural assumptions, to quasilinear problems \cite{ACF2,CV,DP,PV05, LN1, LN2, MW,T1, FL21}.
 Consequently, the free boundary decomposes into two pieces, the \textit{regular part}, 
i.e., the set of points $x$ where $\G(u)$ is $C^{1,\alpha}$ in a neighborhood of $x$, and the \textit{singular part}:

\begin{equation}
\label{eqn:decomposition of fb}
\GU=\Reg(u)\cup\SingU;
\end{equation} 
see Definition~\ref{def: sing point}. For a nice integrand $\Phi$, the regular part  is locally a smooth hypersurface \cite{KN}. 

\vem

The next stage is to estimate the size of the singular part $\SingU$, for instance, in terms of its Hausdorff dimension. Here the theory in the quasilinear setting \eqref{eqn: main functional} diverges substantially from its classical counterpart. 

In the classical setting (when $\Phi(p)=|p|^2$), with the monotonicity formula by Weiss \cite{W},  an estimate on the dimension of $\SingU$ is equivalent to the classification of homogeneous minimizers in each dimension. Thanks to \cite{CJK, JS}, and very recently, \cite{OpenAI} we know that  there are no singular homogeneous minimizers in $\R^n$ for $n\leq 6$. As a result, the singular set has dimension at most $n-7$
for a minimizer $u$ of the classical Alt-Caffarelli functional. This is sharp thanks to \cite{DJ}.

For the quasilinear problem \eqref{eqn: main functional}, however, almost nothing is known about the size of $\SingU$.\footnote{By \cite{MW}, we have $\Hm^{n-1}(\GU)<+\infty.$ The theory of sets of finite perimeter implies $\Hm^{n-1}(\SingU)=0$.} In Theorem~\ref{thm: weiss obstruction}, we show that a Weiss-type monotonicity formula does not hold in this general class. Even for a specialized class of integrands for which a monotonicity formula is known \cite{karakhanyan2025newmonotonicityformulaquasilinear}, it is not straightforward to reduce singularity analysis to the classification of homogeneous minimizers as in the classical setting. As a result, estimates on the size of $\SingU$ require new ideas and techniques. 

\vem 

In this article, we estimate the Hausdorff dimension of the singular set of a minimizer of the quasilinear functional \eqref{eqn: main functional}. 

To be precise, for given constants $\Lambda\ge1$, $\Upsilon>0$ and $\alpha\in(0,1)$, the class of integrands we consider satisfy the following assumptions:
\begin{enumerate}
\item\label{1}{$\Phi\ge 0$ in $\R^n$ and $\Phi(0)=0$;}
\item\label{2}{$\Phi\in C^{1,1}(\R^n)$ with 
 \begin{equation}\label{eqn: ellipt}
 \frac{1}{\Lambda} |\xi|^2 \leq  D^2 \Phi(p)[\xi,\xi] \leq  {\Lambda} |\xi|^2 \qquad \text{ for a.e. $p \in \R^n$ and for all } \xi \in\R^n\,.
 \end{equation}
}

\item\label{3}{For $R_\Lambda>1$ and $0<r_{\Lambda}<1$, depending only on $\Lambda$ and the dimension $n$, we have $\Phi\in C^{2,\alpha}(\An)$ with
\begin{equation}\label{eqn: holder intro}
[D^2\Phi]_{C^\alpha(\An)}\le\C \qquad \hem\text{ where }\qquad \An := B_{R_\Lambda}\setminus \overline{B_{r_\Lambda}}.
\end{equation}
}
\end{enumerate}
For an integrand $\Phi$ satisfying all these assumptions, we write
\begin{equation}
\label{eqn:class of integrands}
\Phi\in \II.
\end{equation} 
If a quantity depends only on the constants $\Lambda,\Upsilon, \alpha$ and the dimension $n$,  then it is referred to as a \textit{universal quantity}.

\begin{remark}
The two universal radii $R_\Lambda$ and $r_\Lambda$ arise from the Lipschitz estimate and the non-degeneracy of minimizers of \eqref{eqn: main functional} and are specified in Definition~\ref{def: annulus}.
\end{remark} 

\begin{remark}
The class $\II$ contains the following special cases:
\begin{enumerate}
\item{A uniformly convex integrand $\Phi$ in $C^{2,\alpha}(\R^n)$ that is nonnegative and vanishes at $0$;}
\item{A {positively} $2$-homogeneous integrand of the form 
$$
\Phi(p):=|p|^2\varphi(p/|p|),
$$
where $\varphi\in C^{2,\alpha}(\mathbb{S}^{n-1})$ is strictly positive and uniformly convex on the sphere.}
\end{enumerate}
In the second class, 
$D^2\Phi$ is $0$-homogeneous away from the origin. It extends continuously to the origin if and only if $\Phi$ is a quadratic form, equivalently $\varphi(\theta)=\theta\cdot A\theta$ for a symmetric positive definite matrix $A$. For this reason, we do not pose continuity  on $D^2\Phi$ near the origin in our definition of $\II$ from \eqref{eqn:class of integrands}. This leads to challenges regarding the Euler-Lagrange equation in a neighborhood of the \text{critical set} $\{\nabla u=0\}$ of a minimizer $u$. {See, for instance, \eqref{eqn: curvature estimate eta intro}, Lemma~\ref{lem: ball with big gradient}, Lemma~\ref{lem: Hessian eps reg}, and Corollary~\ref{lem: hessian bound implies small sing set}.}
\end{remark}

Our main result is the following:
\begin{theorem}\label{thm: main}
Let $n\geq2$, let $\Phi\in\II$, and let $u$ be a nonnegative
minimizer of $\E^\Phi(\cdot)$ in $\Omega$. If $n\geq3$, there exists a universal constant
$\eta\in(0,n-2)$ such that
\begin{equation}\label{eqn: intro measure est}
 \Hm^{n-2-\eta}(\Sing(u))=0.
\end{equation}
If $n=2$, then $\Sing(u)=\varnothing$.
\end{theorem}
Here and in the remainder of the paper, we denote the $d$-dimensional Hausdorff measure by $\Hm^d(\cdot)$. The energy $\mathcal{E}^\Phi(\cdot)= \J(\cdot; \Omega)$ is as defined in \eqref{eqn: main functional}, the class of integrands $I$ is as defined in \eqref{eqn:class of integrands}, and the singular part of the free boundary $\SingU$ is defined in \eqref{eqn:decomposition of fb}. 
 
\begin{remark}
This result is new even without $\eta>0$.  
Considering integer dimensions, this estimate is sharp since there is a singular minimizer in three dimensions \cite{DFKNTY}.
\end{remark}

\begin{remark}
For a similar estimate with $\eta=0$, the third condition in the definition of $I$ in \eqref{eqn:class of integrands} can be relaxed to merely $\Phi \in C^{2}({\overline{B_{R_\Lambda}}}\backslash B_{r_\Lambda})$ without any quantified upper bound; see  Remark~\ref{rmk: no holder}.
\end{remark}

The method is robust enough to cover a larger class of integrands that allow degeneracies in the convexity. To demonstrate the strategy, we estimate the singular set of minimizers of \eqref{eqn: main functional} when $\Phi(\cdot)=|\cdot|^p$, i.e. minimizers of 
\[
\E_p(u) = \int_\Omega |\na u|^p + \chi_{\{u>0\}}\,.
\]
\begin{theorem}\label{thm: p intro}
Let $I\subset (1,\infty)$ be a compact interval and fix $p \in I$.
    There exists $\eta\in (0,n-2)$ depending on $n$ and $I$ such that, for any nonnegative minimizer $u$ of $\E_p(\cdot)$, we have 
    \[
    \Hm^{n-2-\eta}(\Sing(u))=0.
    \] 
    For $n=2$, this should be interpreted as $\Sing(u)=\varnothing$.
\end{theorem}
We refer the reader to Subsection~\ref{rmk: meta thm app} for a discussion of open problems and further settings where we believe the method may be applicable.

\vem

With $\eps$-regularity (see Lemma~\ref{lem: Hessian eps reg} and Corollary~\ref{lem: hessian bound implies small sing set}), Theorem \ref{thm: main} follows from the following curvature estimates, which are of independent interest. Theorem~\ref{thm: p intro} is proven through an analogous scheme and rests upon analogous curvature estimates; see Proposition~\ref{prop: curvature est p}.

\begin{theorem}
\label{thm: curvature estimate intro}
Let $u$ be a nonnegative minimizer of $\mathcal{E}^\Phi(\cdot)$  in $B_2$ with $\Phi\in \II$, then there exist universal constants
$\eta>0$ and $C<\infty$ such that if $0 \in \GU$, then
\begin{equation}
\label{eqn: curvature estimate intro}
\int_{B_1\cap\PosS}|D^2u|^2\le C,
\end{equation}
and
\begin{equation}
\label{eqn: curvature estimate eta intro}
\int_{B_1\cap\PosS\cap\{|\nabla u|\ge 2r_\Lambda\}}|D^2u|^{2+\eta}\le C.
\end{equation}
Here $r_\Lambda$ is the lower radius introduced in \eqref{eqn: holder intro}, which is explicitly defined in Definition \ref{def: annulus}.
\end{theorem}

\vem 

To explain the strategies behind these curvature estimates, we work in the setting when the \textit{free boundary is smooth} {and $\Phi$ is smooth}. 

Under this assumption, the $W^{2,2}$-estimate in \eqref{eqn: curvature estimate intro} can be established for \textit{stable solutions} to the Euler-Lagrange equation {associated to \eqref{eqn: main functional}}.  
For this, we prove a Bochner-type inequality (Lemma~\ref{lem: G eqn}), showing that a function of the gradient is a subsolution with right-hand side involving $|D^2u|^2$. Then, the basic idea to prove \eqref{eqn: curvature estimate intro} is to multiply this function by a cutoff and use this as a test function for the stability inequality (see Lemma~\ref{lem: first second var prelim} and Proposition~\ref{prop: curvature estimates}). Note that this stability inequality holds in the classical sense only when the free boundary is regular.  For the classical Alt-Caffarelli problem (when $\Phi(p)=|p|^2$), a similar estimate has been established by \cite{CFFS} and \cite{FeY2} when the free boundary is smooth. 

To improve \eqref{eqn: curvature estimate intro} to \eqref{eqn: curvature estimate eta intro}, we follow a strategy from \cite{SY}. The key is to rule out the concentration of $|D^2u|$ in its large level sets as in Lemma~\ref{lem: doubling}. Such a non-concentration estimate requires a H\"older estimate of $D^2u$ in the positive phase $\PosS$, which is false since our integrand $\Phi$ is not $C^2$ near the origin. To deal with this difficulty, we show that there is always a significant portion of $\PosS$ that is quantitatively separated from the critical set $\{\nabla u=0\}$; see Lemma~\ref{lem: ball with big gradient}. As a result, it suffices to establish estimates in the set where $\na u(x)$ lies in the annulus $\An$ from \eqref{eqn: holder intro} where the integrand is assumed to be $C^{2,\alpha}.$ {See \eqref{eqn: curvature estimate eta intro} above, as well as Proposition~\ref{prop: curvature estimate 2}, Lemma~\ref{lem: Hessian eps reg}, Corollary~\ref{lem: hessian bound implies small sing set}}. 

\vem

With this, we have the desired curvature estimates in Theorem \ref{thm: curvature estimate intro} when the free boundary is smooth. Similar estimates can be established when $\Hm^{n-2}(\SingU)=0$, see Proposition~\ref{prop: curvature estimates}. Unfortunately, we cannot afford this assumption without running a circular argument, as this is part of our main result in Theorem \ref{thm: main}.

To overcome this difficulty, we adapt a topological argument by Schoen-Simon-Almgren \cite{SSA} and De Philippis-Maggi \cite{DePhilippisMaggi}. We show that the collection of integrands satisfying the desired estimates in Theorem \ref{thm: main} is non-empty, open and closed in $\II$ from \eqref{eqn:class of integrands} and conclude by connectedness of $\II$.  

More concretely, we let
\begin{equation}\label{def: good integrands}
\Istarz = \{ \Phi \in \II\ : \ \text{any minimizer $u$ of $\J(\cdot)$ has } \Hm^{n-2}(\Sing(u))=0\}
\end{equation}
denote the collection of ``good integrands'' such that Theorem~\ref{thm: main} holds with $\eta =0$.

The non-emptiness and openness of $\Istarz$ follow from known results (Theorem~\ref{thm: AC},  Theorem~\ref{thm: eps reg}, and Corollary~\ref{cor: prop of singularities}). To show the closedness of this collection is much more involved. We need to approximate a minimizer of a potentially bad integrand with minimizers of energies with good integrands, using a selection principle. For this, we need to perturb the functional in \eqref{eqn: main functional} by adding a lower order term (see \eqref{def: J f mu}). For minimizers of this perturbed functional, we establish a uniform curvature estimate  as in \eqref{eqn: curvature estimate intro} of Theorem \ref{thm: curvature estimate intro}. 
This bound passes to the limit, yielding an analogous curvature estimate for the original minimizer of the energy with potentially bad integrand. The estimate holds at all points and scales, hence a covering and scaling argument (Corollary~\ref{lem: hessian bound implies small sing set}) leads to the desired bound $\Hm^{n-2}(\Sing(u))=0$.

\vem

Before we conclude this introduction, we note that the estimate in  \eqref{eqn: curvature estimate intro} bears resemblance to Sternberg-Zumbrun estimates for stable solutions to the Allen-Cahn equation \cite{SZcurvatureest} and  curvature estimates for (anisotropic) stable minimal surfaces of Schoen-Simon-Yau, Schoen-Simon, and Allard \cite{SSY,SS,Allard-stability}. They are the key tools to prove compactness of stable minimal hypersurfaces, which is foundational in the regularity theory in min-max constructions \cite{Pitts,DPDRYminmax,DRHW2026,DRH,DeRosaPigati2025}. In the context of anisotropic minimal surfaces, White \cite{WhiteWR} gives an argument, which he attributes to Almgren, to obtain an $\eta$-improvement of the dimension of the singular set.  Our proof strategy is independent and gives a more quantified bound on the improvement. 

\vem

The paper is organized as follows. In Section~\ref{sec: prelim}, we 
introduce some necessary background. Section~\ref{sec: weiss obstruction} is dedicated to a rigidity result for Weiss-type monotonicity formulas. 
In Section~\ref{sec: reg}, we give some regularity and convergence results for minimizers of $\J$ and of the penalized functionals.
 Section~\ref{sec: curvature estimates} is dedicated to proving the curvature estimates \eqref{eqn: curvature estimate intro} and \eqref{eqn: curvature estimate eta intro} for minimizers when the integrand $\Phi$ is smooth and belongs to the class of good integrands $\Phi \in \Istarz$ defined in \eqref{def: good integrands} above. In Section~\ref{sec: main theorem proof 0}, we carry out the topological argument described above to prove Theorems~\ref{thm: main} and \ref{thm: curvature estimate intro}. Finally, Section~\ref{sec: p lap} extends this analysis to the degenerate elliptic and singular settings to establish Theorem~\ref{thm: p intro}.

\noindent{\bf Acknowledgements.}
A.D.R. was funded by the European Union: the European Research Council (ERC), through StG ``ANGEVA'', project number: 101076411. Views and opinions expressed are however those of the authors only and do not necessarily reflect those of the European Union or the European Research Council. Neither the European Union nor the granting authority can be held responsible for them.
R.N. is grateful for the support of NSF CAREER grant DMS-2340195 and NSF RTG grant DMS-2342349. 
H.Y. was supported by the Ministry of Education, Singapore, under its Academic Research Fund Tier 1 (Grant No. 25-2146-A0001).

\noindent{\bf AI Usage Statement.} Claude and Chat GPT were used to aid in the computation of variations in Appendix~\ref{app: variations}, to help choose the constants in the barrier construction of Lemma~\ref{lem: partial harnack} to ensure some finicky set containments, and to obtain the Kato-type inequality with right-hand side for $p<2$ in Proposition~\ref{prop: curvature est p}. AI did not make any serious contributions to the ideas or writing of this paper.

\section{Preliminaries}\label{sec: prelim}
Throughout the paper, $n\geq 2$, $\Lambda\geq 1$, $\alpha \in(0,1)$, and  $\C>0$ are fixed. We refer to these as {universal parameters} and call a constant universal if it depends only on these parameters.

\subsection{Facts about integrands}\label{sec: prelim integrands} Recall the class of integrands $\II$ defined in \eqref{eqn:class of integrands} and
suppose $\Phi \in \II.$ Observe that $D \Phi(0)=0$ thanks to the assumptions that $\Phi$ is nonnegative, $C^{1,1}$, and vanishes at the origin. The Hessian bounds \eqref{eqn: ellipt} and the fundamental theorem of calculus yield the upper and lower bounds 
\begin{equation}\label{eqn: basic bounds for integrands}
    \mfrac{1}{2\La} |p|^2 \leq \Phi(p) \leq   \mfrac{\La}{2}|p|^2,  \qquad \mfrac{1}{\Lambda}|p| \leq |D \Phi(p) |\leq \La |p|,  \qquad { \mfrac{|p|^2}{2\La}\leq D\Phi(p)\cdot p-\Phi(p)\leq\mfrac{\La|p|^2}{2}}.
\end{equation}
One consequence of the third estimate in \eqref{eqn: basic bounds for integrands} is that $D \Phi(p) \cdot p \geq 0$ for all $p \in \R^n$. A second is that for any $p\in \R^n$,
\begin{equation}\label{eqn: first annular constraint}
   \cG(p) :=  D\Phi(p)\cdot p -\Phi(p) =1 \qquad \implies \qquad {\sqrt{\mfrac2\La}\leq|p|\leq\sqrt{2\La}}.
\end{equation}
Similarly, the uniform convexity \eqref{eqn: ellipt} of $\Phi$ implies that $t \mapsto \cG(tp)$ is strictly increasing for all $p \neq 0$. Hence, for any direction $e \in S^{n-1},$ there is a unique $a>0$ such that the function $v_a(x) := a(x\cdot e)_+$ satisfies the free boundary condition $\cG(\na v_a)= \cG(ae)=1$ on $\partial \{v_a>0\}$; c.f. \eqref{eqn: intro EL} below. 

Next, for a sequence $\{\Phi_k\} \subset \II$, we write 
$$\Phi_k \to \Phi$$
to mean $\Phi_k \to \Phi$ locally uniformly on $\R^n$. 
From the definition of $\II$ and the considerations above, we see that if $\Phi_k \in \II$ and $\Phi_k\to \Phi$, then the common \(C^{1,1}\) bound implies $\|\Phi_k-\Phi\|_{C^1(\overline{B_M})} \to 0$ for every \(M<\infty\).
The common H\"older bound on the Hessians \eqref{eqn: holder intro} on \(\An\)
also gives $
    \|D^2\Phi_k-D^2\Phi\|_{C^0(\An)}
    \to 0.$ These assertions follow by Arzel\`a-Ascoli and identification
of limits by distributional differentiation.

\begin{remark}\label{rmk: connected}
  By the same reasoning, we see that $\II$ is compact when endowed with the topology induced by this convergence. Moreover, note that $\II$ is convex and in particular connected. Finally, a basic mollification and vertical translation argument shows that smooth integrands are dense in $\II$ in this topology. 
\end{remark}

\subsection{The main functional and the penalized functional}
Consider the main energy functional $\J(v;\Omega)$ defined in \eqref{eqn: main functional}, with integrand $\Phi \in \II$, where $\II$ is the class of integrands defined in \eqref{eqn:class of integrands}.
 When the domain $\Omega$ is understood or plays no important role, we  write $\J(u)$ in place of $\J(u;\Omega).$
 
 We say a nonnegative function  $u \in W^{1,2}(\Omega)$ is a minimizer of $\J(\cdot)$ in $\Omega$ if  $\J(u;\Omega) \leq \J(u+\varphi; \Omega)$  for every $\varphi \in W^{1,2}_0(\Omega)$. Every time we refer to a minimizer of $\J$ in $\Omega$, we tacitly assume $u\geq0$ and $u \in W^{1,2}(\Omega)$.
 Formally, a minimizer $u$ of $\J(\cdot, \Omega)$ satisfies the quasilinear Bernoulli free boundary problem
\begin{equation}
    \label{eqn: intro EL}
\begin{cases}
    \Div(D \Phi (\na u)) = 0 & \text{ in } \{ u>0\} \cap \Omega,\\
    D \Phi(\na u )\cdot \na u -\Phi(\na u) = 1 & \text{ on } \GU:= \partial\{u>0\} \cap \Omega.
\end{cases}
\end{equation}

The proof of Theorem~\ref{thm: main} calls for a functional $\Jfm$ that is a perturbation of $\J$. Indeed, for a sequence $\Phi_k$ in $\II$ with $\Phi_k \to \Phi$ and a minimizer $\bar{u}$ of $\J$, we will approximate $\bar{u}$ by functions $u_k$ that minimize an energy functional that is the sum of $\E^{\Phi_k}$ and a fidelity term that forces minimizers $u_k$ to be close to $\bar{u}$.

More precisely, given a nonnegative function $f \in C^\infty(\Omega)$ (which will be taken to be a mollification of $\bar{u}$ in the application described above) and a number $\mu \in [0,1]$, consider the energy functional
\begin{equation}
    \label{def: J f mu}
\J_{f,\mu}(v,\Omega) = \J(v,\Omega) + \int_\Omega |\mu v -f|^2.
\end{equation}
The parameter $\mu$ makes class of functionals \(\Jfm\) invariant under the blow-up rescaling:
if \(u\) is minimizer of \(\Jfm(\cdot,\Omega)\), \(x_0\in\Omega\),
and \(0<s\leq1\), then  the Lipschitz rescaling $\tilde u(x):={u(x_0+sx)}/{s}$
is a minimizer of
\( \mathcal{E}^\Phi_{\tilde f,\tilde\mu}(\cdot,\Omega_s)\), where 
\begin{equation}
    \label{eqn: rescale}
    \Omega_s:=({\Omega-x_0})/{s}, \qquad
        \tilde \mu:=s\mu\in[0,1],
    \qquad
    \tilde f:=f(x_0+s\,\cdot),
    \qquad
    \|\tilde f\|_{L^\infty(\Omega_s)}\leq \|f\|_{L^\infty(\Omega)}.
\end{equation}
Indeed, since \(\na\tilde u(x)=\na u(x_0+sx)\), all three terms of \(\Jfm\)
scale by the factor \(s^{-n}\) under this transformation.
Thus estimates for minimizers of $\Jfm$ that are uniform in $\mu \in [0,1]$ can be applied to minimizers of $\J_{f,1}$ at all scales $s \in (0,1]$.

 \begin{lemma}\label{lem: first second var prelim}
 With $f, \mu$, and $\Jfm$ as above, let $u$ be a minimizer of $\Jfm(\cdot)$ in $B_1$ and assume $\GU= \partial\{u>0\} \cap  B_1$ is $C^{2,\alpha}.$
Then $u$ satisfies the Euler-Lagrange equation
\begin{equation}\label{eqn: EL}
    \begin{cases}
        {\rm div}\,(D \Phi(\na u)) =2\mu(\mu u -f) & \text{ in } \{ u>0\}\cap B_1,\\
         D \Phi(\na u )\cdot \na u -\Phi(\na u) = 1 & \text{ on }\GU \,.
    \end{cases}
    \end{equation}
    Provided $\Phi \in C^2(\R^n)$, such a minimizer satisfies the following stability inequality:
\begin{equation}\label{eqn: stability}
\int_{\{u>0\}\cap B_1}\! D^2\Phi(\nabla u)[\nabla \psi,\nabla \psi]+2\mu^2\psi^2 \,dx \geq
\int_{\GU}\big\langle D^2\Phi(\na u)D^2u[\nu ], \nu \big\rangle \frac{ (D \Phi(\na u )\cdot \nu)}{1 +\Phi(\na u)}\psi^2 \, \mathcal{H}^{n-1}
\end{equation}
for all $\psi \in C^1_c(B_1)$.
 \end{lemma}
Here $\nu$ denotes the outer unit normal to $\{u>0\}$.
The proofs of \eqref{eqn: EL} and \eqref{eqn: stability} are given in appendix, in Lemmas~\ref{lem: first var} and \ref{lem: second var} respectively.
 The interior equation in \eqref{eqn: EL} holds in the weak sense, and holds classically on $\{x:  \na u(x) \in \An\}$ with $\An$ as in \eqref{eqn: holder intro}. The free boundary condition in \eqref{eqn: EL} holds classically under this assumed regularity of the free boundary.

In the important special case when $\Phi \in \II$ is two homogeneous, the free boundary condition in \eqref{eqn: EL} simplifies to $\Phi(\na u) =1$. The stability inequality \eqref{eqn: stability} takes a form that is  analogous to that  for the classical Alt-Caffarelli energy, involving the mean curvature of the free boundary; see Remark~\ref{rmk: aniso MC} and in particular \eqref{eqn: 2hom stability 2}.

In general, the free boundary $\GU$ of a minimizer $u$ of $\Jfm(\cdot)$ in $B_1$ need not be everywhere $C^{2,\alpha}$ as assumed above. 
We show in Proposition~\ref{prop: min is viscosity solution} and Corollary~\ref{cor: higher reg} below that if, for a point $x_0 \in \GU$, there is a scale $r>0$ such that $\GU$ is flat at $x_0$ at scale $r$ in the sense of \eqref{eqn: flat 1} and the right-hand side $g = 2\mu(\mu u-f)$ is small at scale $r$ in the sense of \eqref{eqn: small RHS}, then $\GU$ is a $C^{2,\alpha}$ graph in a neighborhood of $x_0$.
\begin{definition}\label{def: sing point}
A point $x_0\in\GU$ is regular if, after a rotation,
$\{u>0\}$ is locally the epigraph of a $C^{1,\alpha}$ function near $x_0$.
We denote the set of regular points by $\Reg(u)$ and define
\begin{equation}\label{eqn: prelim sing set def}
\Sing(u)=\GU\setminus\Reg(u).
\end{equation}
\end{definition}
For $\Phi\in\II$ and smooth data, we improve the regularity of such a graph to $C^{2,\alpha}$, see Corollary~\ref{cor: higher reg} below.

\subsection{Viscosity solutions}\label{sec: prelim visc} 
As above, let $\Phi \in \II$ where $\II$ is the class of integrands defined in \eqref{eqn:class of integrands}. In this section we use the shorthand  
\begin{equation}
    \label{eqn: short hand}
\Lo w= \Div(D \Phi(\na w))\qquad\text{ and } \qquad\cG(\na u) = D \Phi(\na u )\cdot \na u -\Phi(\na u).
\end{equation}
For $g\in C(B_1)\cap L^\infty(B_1)$,  let us define  viscosity solutions to the free boundary problem
\begin{equation}
    \label{eqn: e reg sec EL}
\begin{cases}
    \Lo u = g & \text{ in } \{u>0\} \cap \Omega,\\
    \cG(\na u) = 1 & \text{ on } \GU \,.
    \end{cases}
\end{equation}
For a  nonnegative function $u\in C^0(\Omega)$, we say
\begin{equation}
    \label{eqn: visc BC}
u \text{ satisfies }\cG(\na u)=1\text{ on }\GU \text{ in the viscosity sense} 
\end{equation}
if the following holds for any $x_0 \in \GU$ and $\phi\in C^\infty(\Omega)$: 
if $\phi$ touches $u$ from below 
(resp., $\phi_+$ touches $u$ from above) in $\Omega$ at $x_0$, then $\cG(\na\phi(x_0)) \leq 1$ (resp., $\cG(\na \phi(x_0))\geq 1.$).

A subtlety arises in defining viscosity solutions to $\Lo u=g$ in $\{u>0\}$ since $D^2 \Phi(p)$ only exists a.e. outside of $\An$, thus  $\Lo\phi =\Div(D\Phi (\na \phi(x)))$ may not be defined for a smooth function $\phi$ at a point $x$ where $\na \phi(x)$ is in the set where $\Phi$ is not twice differentiable.  
 To deal with this, we recall the Pucci extremal operators defined for a symmetric matrix $M$ with eigenvalues $\{\mu_i\}$ by 
\[
\mathcal{M}^+(M, \lambda, \Lambda) = \Lambda \sum_{\mu_i>0} \mu_i + \lambda \sum_{\mu_i<0} \mu_i,\qquad \mathcal{M}^-(M, \lambda, \Lambda) = \lambda \sum_{\mu_i>0} \mu_i + \Lambda \sum_{\mu_i<0} \mu_i \,.
\]
Recall that a function $w \in C^0(\Omega)$ satisfies 
 $\mathcal{M}^-(D^2 w(x_0), \lambda,\La) \leq g(x_0)\leq \mathcal{M}^+(D^2 w(x_0), \lambda,\La) $ in the viscosity sense at $x_0 \in \Omega$
 if we have $\mathcal{M}^-(D^2\phi,\lambda, \Lambda) \leq g(x_0)$ for any function $\phi \in C^\infty(\Omega)$ that touches $w$ from below at $x_0$, and $\mathcal{M}^+(D^2\phi,\lambda, \Lambda)\geq g(x_0)$ for any function $\phi \in C^\infty(\Omega)$ that touches $w$ from above at $x_0$.

 So, for a nonnegative function $u\in C^0(\Omega)$, we say 
\begin{equation}
    \label{eqn: visc int cond}
    u \text{ satisfies } \Lo u=g \text{ in }\{ u>0\}\text{ in the viscosity sense} 
\end{equation}
if the following holds for any $x_0 \in \{ u>0\} \cap \Omega$ and $\phi\in C^\infty(\Omega)$:\\
-- if $\phi$  touches $u$ from below (resp., above) in $\Omega$ at $x_0$ and $\na \phi(x_0) \in \An$, then $\Lo \phi(x_0) \leq g(x_0)$ (resp., $\Lo \phi (x_0)\geq g(x_0)$).\\
-- if $\phi$ touches $u$ from below (resp., above) in $\Omega$ at $x_0$ and $\na\phi(x_0) \not\in \An$, then $\mathcal{M}^-(D^2\phi,1/\Lambda, \Lambda) \leq g(x_0)$  (resp., $\mathcal{M}^+(D^2\phi,1/\Lambda, \Lambda)\geq g(x_0)$).\\
 
Finally, we say $u$ is a viscosity solution to \eqref{eqn: e reg sec EL} in $\Omega$ if $u \in C^0(\Omega)$ is nonnegative and satisfies \eqref{eqn: visc BC} and \eqref{eqn: visc int cond}.

Note that if  $u$ satisfies \eqref{eqn: visc int cond} then $\mathcal{M}^-(D^2 u, 1/\La,\La) \leq g\leq \mathcal{M}^+(D^2 u, 1/\La,\La) $ in the viscosity sense
in $\{u>0\}$.

In our application, we will consider minimizers $u$ of $\Jfm$ where $f$ is smooth and bounded; we show  $u$ is bounded (Lemma~\ref{lem:interior-lip}) and a viscosity solution to \eqref{eqn: EL} (Proposition~\ref{prop: min is viscosity solution}), 
thus we treat $g = 2\mu(\mu u -f)$ as a bounded right-hand side and consider viscosity solutions to the problem \eqref{eqn: e reg sec EL}.

\subsection{The classical Alt-Caffarelli problem}
\label{sec:aclit:classical}
When $\Phi(v)=|v|^2$, in their foundational work, Alt and Caffarelli \cite{AC}
established existence, local Lipschitz continuity and nondegeneracy of
minimizers, together with density estimates for the two phases and local
finiteness of the $(n-1)$-dimensional measure of the free boundary. 
Their improvement-of-flatness theorem gives $\Hm^{n-1}(\Sing(u))=0$ and that $\Reg(u)$ is relatively open and locally a $C^{1,\alpha}$
hypersurface. Higher regularity theory then implies that $\Reg(u)$ is analytic; see Kinderlehrer-Nirenberg \cite{KN}.
A different proof of  improvement of flatness was developed by De Silva \cite{DeSilva}, combining a Harnack inequality
with compactness and linearization around a half-plane solution.

The principal tool for improving the bound on $\Sing(u)$ is the
monotonicity formula of Weiss \cite{W}. For
$x_0\in\G(u)$ and $B_r(x_0)\Subset\Omega$, set
\begin{equation}\label{eq:aclit:weiss}
    W(u,x_0,r)
    :=  r^{-n}\int_{B_r(x_0)}
        \bigl(|\na u|^2+\chi_{\{u>0\}}\bigr)\,dx-r^{-n-1}\int_{\partial B_r(x_0)}u^2\,d\Hm^{n-1}.
\end{equation}
For almost every $r$,
\begin{equation}\label{eq:aclit:weiss-derivative}
    \frac{d}{dr}W(u,x_0,r)
    =2r^{-n}\int_{\partial B_r(x_0)}
        \left(\partial_\nu u-\frac{u}{r}\right)^2\,d\Hm^{n-1}
    \geq 0,
\end{equation}
where $\nu$ is the outer unit normal to $\partial B_r(x_0)$.
Consequently, every subsequential blow-up is a nonzero, one-homogeneous global
minimizer. Moreover, a half-plane blow-up implies regularity of $x_0$
by improvement of flatness. Singular points are therefore detected by
nonflat minimizing cones, and a Federer dimension-reduction argument converts low-dimensional cone
classification into a Hausdorff dimension estimate.

Caffarelli-Jerison-Kenig \cite{CJK} ruled out
nonflat stable cones in dimension three, and Jerison-Savin
\cite{JS} extended this conclusion to dimension four. Very recently OpenAI extended this though dimension six \cite{OpenAI}.
This implies  free boundaries are analytic in dimensions $n\leq 6$,
and the universal bound in higher dimensions is
$\dim_{\Hm}\Sing(u)\leq n-7$.
In \cite{DJ}, De Silva-Jerison constructed a singular minimizing cone in $\R^7$; see also \cite{FTW1,FTW2} for further
minimizing cones in $\R^7$.
 All together, these works amount to the following:
\begin{theorem}\label{thm: AC}
    Let $u$ be a minimizer of the classical Alt-Caffarelli energy $\int_{\Omega} |\na u|^2 + \chi_{\{u>0\}}$. Then the free boundary $\GU$ is analytic outside a singular set $\Sing(u)$ with 
    \begin{equation}
        \label{eqn: AC sing}
          \dim_{\mathcal{H}}(\Sing(u)) \leq n-7.
    \end{equation}
\end{theorem}

Fern\'andez-Real-Yu \cite{FeY1, FeY2} establish improved bounds for the singular set for generic boundary data. Finer properties  of the singular set are provided in the form of  rectifiability and Minkowski estimates for the strata of the singular set (Edelen-Engelstein \cite{ACsurvey-EE19});  uniqueness of blow-ups at isolated singularities (Engelstein-Spolaor-Velichkov
\cite{ESV}), and Hardt-Simon type foliations for homogeneous global minimizers (De Silva-Jerison-Shahgholian
\cite{DJS} and Edelen-Spolaor-Velichkov \cite{EdSV}; see related recent work of  Engelstein-Restrepo-Zhao \cite{EDZ}).

\subsection{PDE Estimates}

In the course of proving the main theorems, we will make use of the following interior and boundary regularity estimates for solutions to the interior equation in \eqref{eqn: EL}.
  
\begin{lemma}
\label{lem: C1beta}
Let $\Phi\in\II$. The following estimates hold.
\begin{enumerate}
    \item (Interior estimates) Let
$\psi\in L^\infty(B_r(x_0))$ and let
$w\in H^1(B_r(x_0)) \cap L^\infty(B_r(x_0))$ be a weak solution to
\[
    \Div\bigl(D\Phi(\na w)\bigr)=\psi
    \qquad\text{in }B_r(x_0).
\]
There are $\beta=\beta(n,\La)\in(0,1)$ and $C=C(n,\La)$ such that
    $w\in C^{1,\beta}_{\mathrm{loc}}(B_r(x_0))
    \cap H^2_{\mathrm{loc}}(B_r(x_0))$ and
\begin{equation}
    \label{eqn: gradient reg interior}
r\|\na w\|_{L^\infty(B_{r/2}(x_0))}
    +r^{1+\beta}[\na w]_{C^{0,\beta}(B_{r/2}(x_0))}
\leq C\left(
        \|w\|_{L^\infty(B_{r}(x_0))}
        +r^2\|\psi\|_{L^\infty(B_{r}(x_0))}
    \right)\,.
\end{equation}
Moreover, if  $\psi=0$ and $\na w(x)\in \An$ for all $x \in B_{r}(x_0)$, then
$w\in C^{2,\alpha}_{\mathrm{loc}}
    \bigl(B_r(x_0))$ with 
\begin{equation}
    \label{eqn: Hessian bound interior}
r^2\|D^2 w\|_{L^\infty(B_{r/2}(x_0))} + r^{2+\alpha}[D^2 w]_{C^{0,\alpha}(B_{r/2}(x_0))}\leq C\,.
\end{equation}
for a universal constant $C>0$.
\item (Boundary estimates) Let $\Omega$ be an open bounded set such that $\G:=\partial \Omega \cap B_r(x_0)$ is of class $C^{2,\alpha}$. Let $w\in H^1(\Omega \cap B_r(x_0)) \cap C^0(\overline{\Omega}\cap B_r(x_0))$ be a weak solution to 
\[
\Div\bigl(D\Phi(\na w)\bigr)=0
    \quad\text{in }\Omega \cap B_r(x_0), \qquad\qquad w=0 \text{ on }\G.
\]
{Assume that, after rescaling by $r^{-1}$, the boundary admits $C^{2,\alpha}$ graph charts of fixed radius, with their $C^{2,\alpha}$ norms bounded by $M$.} There are $\beta=\beta(n,\La)\in(0,1)$ and $C=C(n,\La,{M})$ such that $w \in C^{1,\beta}_{loc}(\overline{\Omega}\cap B_{r}(x_0)) \cap H^2_{loc}(\overline{\Omega}\cap B_r(x_0))$ with 
\begin{equation}
    \label{eqn: gradient reg boundary}
   r\| \na w\|_{L^\infty(\overline{\Omega}\cap B_{r/2}(x_0))}  + r^{1+\beta} [\na w]_{C^{0,\beta}(\overline{\Omega}\cap B_{r/2}(x_0))} \leq C \| w\|_{L^\infty(\overline{\Omega}\cap B_{r}(x_0))}\,.
\end{equation}
Moreover, if  $\na w(x)\in \An$ for all $x \in \Omega \cap B_{r}(x_0)$, then
$w\in C^{2,\alpha}_{\mathrm{loc}}(\overline{\Omega}\cap B_r(x_0))$ and 
\begin{equation}
    \label{eqn: Hessian bound boundary}
    r^2\|D^2 w\|_{L^\infty({\overline{\Omega}\cap}B_{r/2}(x_0))} + r^{2+\alpha}[D^2 w]_{C^{0,\alpha}({\overline{\Omega}\cap}B_{r/2}(x_0))}\leq C\,.
\end{equation}
\end{enumerate}
\end{lemma}
Note that the constants for the gradient estimates \eqref{eqn: gradient reg interior} and \eqref{eqn: gradient reg boundary} depend only on $n$ and $\La$, not on $\alpha$ or $\C$. These estimates hold by applying De Giorgi to the equation satisfied by $\partial_i u$; see \cite[Ch. 8]{GT},  first for a regularization $\Phi^\e \in \II$ of $\Phi$ (recall Remark~\ref{rmk: connected}) and then passing the estimates to the limit. The Hessian estimates \eqref{eqn: Hessian bound interior} and \eqref{eqn: Hessian bound boundary} follow from Schauder estimates; note that assumption $\na w\in \An$ in particular implies $\| \na w \|_{L^\infty} \leq R_\La$.

\section{Rigidity of Weiss-type monotonicity formulas}
\label{sec: weiss obstruction}
As described in Section~\ref{sec:aclit:classical}, the Weiss monotonicity formula \eqref{eq:aclit:weiss} plays a central role in singularity analysis for the classical Alt-Caffarelli problem. In this section, we show that the specific structure of the integrand $p\mapsto |p|^2$ is quite special, and the space of integrands admitting a Weiss-type monotonicity formula is rigid. This is an analogue of the  rigidity result in
\cite[Section~3]{AllardArea} in the context of anisotropic minimal surfaces, where the lack of the monotonicity formula for the density ratio makes the anisotropic minimal surface theory sensibly harder \cite{DeR-survey,DeRosa2026Lectures}. 

Let $\rho:\R^n\to[0,\infty)$ be positive on $\R^n\setminus\{0\}$,
$C^2$ away from the origin, and positively one-homogeneous. Set
\[
    B_r^\rho(x_0):=\{x:\rho(x-x_0)<r\},
    \qquad \nu_\rho(z):=\frac{D\rho(z)}{|D\rho(z)|}.
\]
Thus $\nu_\rho$ is the outer unit normal to $\partial B_1^\rho(0)$ and
$z\cdot\nu_\rho(z)>0$. Let
$h\in C^1(\partial B_1^\rho(0)\times[0,\infty))$, with $h(z,0)=0$.
The normalization at zero only removes an additive constant from the
quantity below. For a nonnegative locally Lipschitz minimizer $u$ of
$\J$ in $\Omega$, define
\begin{equation}\label{eqn: weiss local functional}
    W_{\rho,h}(u,x_0,r)
    :=\frac{\J(u;B_r^\rho(x_0))}{r^n}
      -\int_{\partial B_1^\rho(0)}
          h\left(z,\frac{u(x_0+rz)}{r}\right)
          \,\dd\Hm^{n-1}(z),
\end{equation}
whenever $B_r^\rho(x_0)\Subset\Omega$. 
Introduce the algebraic defect
\begin{equation}\label{eqn: weiss local defect}
    \sigma_{\Phi,\rho,h}(p,z,s)
    :=(p\cdot z-s)
       \bigl(D\Phi(p)\cdot\nu_\rho(z)-\partial_s h(z,s)\bigr),
    \qquad p\in\R^n,\quad z\in\partial B_1^\rho(0),\quad s\geq0.
\end{equation}

\begin{theorem}
\label{thm: weiss obstruction}
Let $n\geq2$, $\Phi\in \II$, and let $\rho,h$ be as above.
For every minimizer $u$ as above,
$W_{\rho,h}(u,x_0,\cdot)$ is locally absolutely continuous and,
for almost every admissible $r$,
\begin{equation}\label{eqn: weiss local derivative}
    \frac{d}{dr}W_{\rho,h}(u,x_0,r)
    =\frac1r\int_{\partial B_1^\rho(0)}
       \sigma_{\Phi,\rho,h}
       \left(\na u(x_0+rz),z,\frac{u(x_0+rz)}r\right)
       \,\dd\Hm^{n-1}(z).
\end{equation}
Moreover, the following are equivalent:
\begin{enumerate}[label=\textup{(\roman*)}]
    \item\label{item: weiss defect positive}
    The pointwise sign condition
    \begin{equation}\label{eqn: weiss pointwise positivity}
        \sigma_{\Phi,\rho,h}(p,z,s)\geq0
        \qquad\text{for every }p\in\R^n,\ z\in\partial B_1^\rho(0),
        \ s\geq0
    \end{equation}
    holds.

    \item\label{item: weiss quadratic case}
    There exist an invertible linear map $L:\R^n\to\R^n$ and $c>0$
    such that
    \begin{equation}\label{eqn: weiss quadratic classification}
        \rho(x)=|L^{-1}x|,
        \qquad \Phi(p)=c|L^T p|^2,
        \qquad
        h(L\theta,s)=\frac{c s^2}{|L^{-T}\theta|}
    \end{equation}
    for all $x,p\in\R^n$, $\theta\in S^{n-1}$, and $s\geq0$.
\end{enumerate}
In particular, for Euclidean balls the sign condition holds if and
only if $\Phi(p)=c|p|^2$ and $h(\theta,s)=cs^2$.
For a genuinely nonquadratic integrand, no choice of $\rho$ and $h$
in the above class gives an everywhere nonnegative defect in
\eqref{eqn: weiss local derivative}, even when $\Phi$ is
two-homogeneous.
\end{theorem}

\begin{remark}
\label{rmk: weiss obstruction scope}
For $\rho(x)=|x|$ and $h(\theta,s)=cs^2$, the quantity \eqref{eqn: weiss local functional} above is
\[
    W_c(u,x_0,r)
    =\frac{\J(u;B_r(x_0))}{r^n}
     -\frac{c}{r^{n+1}}
        \int_{\partial B_r(x_0)}u^2\,\dd\Hm^{n-1},
\]
and \eqref{eqn: weiss local derivative} becomes
\begin{align*}
    \frac{d}{dr}W_c(u,x_0,r)
    ={}&\frac{2c}{r^n}\int_{\partial B_r(x_0)}
          \left(\partial_\nu u-\frac ur\right)^2
          \,\dd\Hm^{n-1}+\frac1{r^n}\int_{\partial B_r(x_0)}
          \left(\partial_\nu u-\frac ur\right)
          \bigl(D\Phi(\na u)-2c\na u\bigr)\cdot\nu
          \,\dd\Hm^{n-1}.
\end{align*}
The second term is the obstruction to the usual square-completion
argument that yields monotonicity in the classical Alt-Caffarelli setting \cite{W}. In the quadratic case in
\eqref{eqn: weiss quadratic classification}, the change of variables
$v(y):=u(x_0+Ly)$ gives instead
\[
    W_{\rho,h}(u,x_0,r)
    =|\det L|\left[
       \frac1{r^n}\int_{B_r}
             \bigl(c|\na v|^2+\chi_{\{v>0\}}\bigr)\,dy
       -\frac{c}{r^{n+1}}
          \int_{\partial B_r}v^2\,\dd\Hm^{n-1}\right],
\]
so that the classical Weiss formula is recovered on the corresponding
ellipsoids.

{The pointwise condition \eqref{eqn: weiss pointwise positivity} is stronger than monotonicity of $W_{\rho,h}$ along minimizers. Our argument does not show that \eqref{eqn: weiss pointwise positivity} is necessary for the integral in \eqref{eqn: weiss local derivative} to be nonnegative for every minimizer. In particular, the theorem does not rule out cancellations in the boundary integral.}
\end{remark}
\begin{remark}
    It is important to remark that
Theorem~\ref{thm: weiss obstruction} does not
exclude corrections involving gradients, nonlocal quantities, or
auxiliary functions. The correction used for a class
of radial integrands in
\cite{karakhanyan2025newmonotonicityformulaquasilinear} lies outside
the class considered here.
\end{remark}

\begin{proof}[Proof of Theorem~\ref{thm: weiss obstruction}]
\emph{Step 1: The radial first-variation identity.}
By translation, assume $x_0=0$.
Inner variations and the positivity-preserving variations
$u_\varepsilon=(1+\varepsilon\psi)u$ give, respectively,
\begin{align}
    &\int_\Omega
       \left((\Phi(\na u)+\chi_{\{u>0\}})\divv X
       -D\Phi(\na u)\cdot DX^T\na u\right)\,dx=0,
       \label{eqn: weiss inner stationarity}\\
    &\int_\Omega D\Phi(\na u)\cdot
       \left(\psi\na u+u\na\psi\right)\,dx=0,
       \label{eqn: weiss amplitude stationarity}
\end{align}
for $X\in C_c^1(\Omega;\R^n)$ and $\psi\in C_c^1(\Omega)$; see Appendix~\ref{app: variations}.
These identities only require $\Phi\in C^{1,1}$ and do not require
regularity of $\GU$.

Use radial cutoffs $X(x)=\eta(\rho(x))x$ and
$\psi(x)=\eta(\rho(x))$, with $\eta$ constant near zero, and let
$\eta$ approximate the characteristic function of $[0,r)$.
For almost every $r$, the coarea formula yields
\begin{align}
    n\J(u;B_r^\rho)-\int_{B_r^\rho}D\Phi(\na u)\cdot\na u\,dx
    &=\int_{\partial B_r^\rho}
       \left((\Phi(\na u)+\chi_{\{u>0\}})\,x\cdot\nu
       -(x\cdot\na u)D\Phi(\na u)\cdot\nu\right)
       \,\dd\Hm^{n-1},
       \label{eqn: weiss radial domain}\\
    \int_{B_r^\rho}D\Phi(\na u)\cdot\na u\,dx
    &=\int_{\partial B_r^\rho}
       uD\Phi(\na u)\cdot\nu\,\dd\Hm^{n-1},
       \label{eqn: weiss radial amplitude}
\end{align}
where $\nu$ is the outer unit normal to $\partial B_r^\rho$.
Euler's identity for $\rho$ gives
$x\cdot\nu=r/|D\rho(x)|$ on $\partial B_r^\rho$, and hence
\[
    r\frac{d}{dr}\bigl(\J(u;B_r^\rho)\bigr)=\int_{\partial B_r^\rho}
              (\Phi(\na u)+\chi_{\{u>0\}})\,x\cdot\nu\,\dd\Hm^{n-1}.
\]
Combining these identities, we obtain
\begin{equation}\label{eqn: weiss bulk derivative}
    \frac{d}{dr}\bigl(r^{-n}\J(u;B_r^\rho)\bigr)
    =\frac1{r^{n+1}}\int_{\partial B_r^\rho}
       (x\cdot\na u-u)D\Phi(\na u)\cdot\nu
       \,\dd\Hm^{n-1}.
\end{equation}
For almost every $(r,z)$,
\[
    \partial_r \left(\frac{u(rz)}{r}\right)=\frac{z\cdot\na u(rz)-\frac{u(rz)}{r}}r.
\]
Changing variables $x=rz$ in
\eqref{eqn: weiss bulk derivative} and differentiating the boundary
term in \eqref{eqn: weiss local functional} proves
\eqref{eqn: weiss local derivative}.
Local absolute continuity follows from the coarea formula and the
local Lipschitz regularity of $u$.
Notice that no homogeneity assumption on $\Phi$ was used.

\emph{Step 2: The sign condition and Euclidean rigidity.}
Assume \eqref{eqn: weiss pointwise positivity}.
Fix $z\in\partial B_1^\rho(0)$ and $p\in\R^n$ with $p\cdot z>0$.
For every $t\in \R$ with  $|t|$ small enough, we can test \eqref{eqn: weiss pointwise positivity} at $s=p\cdot z+t$ and we get
\[
-t\bigl(D\Phi(p)\cdot\nu_\rho(z)-\partial_s h(z,p\cdot z+t)\bigr)\geq0.
\]
Letting $t \downarrow 0^+$ and $t\uparrow0^-$ yields
\begin{equation}\label{eqn: weiss flux factorization}
    D\Phi(p)\cdot\nu_\rho(z)=\partial_s h(z,p\cdot z)
    \qquad\text{whenever }p\cdot z>0.
\end{equation}

First suppose $\rho(x)=|x|$, so that $z=\nu_\rho(z)=\theta$.
For $q\cdot\theta=0$ and $s>0$, the preceding identity reads
\begin{equation}\label{eqn: weiss euclidean flux}
    D\Phi(q+s\theta)\cdot\theta=\partial_s h(\theta,s).
\end{equation}
Taking $q=0$ and letting $s\downarrow0$, we obtain
$\partial_s h(\theta,0)=0$, since $D\Phi(0)=0$.
Letting $s\downarrow0$ for arbitrary $q\perp\theta$ then gives
\[
    D\Phi(q)\cdot\theta=0
    \qquad\text{whenever }q\cdot\theta=0.
\]
Thus $D\Phi(q)$ is parallel to $q$ for every $q\ne0$.
Since $n\geq2$ and spheres are connected, there is a function $F$
such that $\Phi(p)=F(|p|)$.

Given $0<s\leq R$, choose $q\perp\theta$ with
$|q|^2=R^2-s^2$. Equation~\eqref{eqn: weiss euclidean flux} becomes
\[
    \partial_s h(\theta,s)=s\frac{F'(R)}R.
\]
For any $R_1,R_2>0$, choosing $0<s<\min\{R_1,R_2\}$ shows that
$F'(R_1)/R_1=F'(R_2)/R_2$. Consequently $F'(R)=2cR$ for a constant
$c>0$. Since $\Phi(0)=0$ and $h(\theta,0)=0$, it follows that
\[
    \Phi(p)=c|p|^2,
    \qquad h(\theta,s)=cs^2.
\]
This Euclidean argument uses only $C^1$ regularity,
$\Phi(0)=D\Phi(0)=0$, and strict convexity of $\Phi$.

\emph{Step 3: Rigidity of the domains.} 
Return to an arbitrary $\rho$.
Write $Q(p):=D^2\Phi(p)$ for $p$ in $\An$,
where $\Phi$ is $C^2$. Choose $p_0\ne0$ in this annulus.
For $p_0\cdot z>0$, differentiate
\eqref{eqn: weiss flux factorization} with respect to $p$ along any
$v\perp z$. The right-hand side stays constant, so $Q(p_0)\nu_\rho(z)\cdot v=0.$
It follows that
\begin{equation}\label{eqn: weiss hessian normal}
    Q(p_0)\nu_\rho(z)\in(0,\infty)z
    \qquad\text{if }p_0\cdot z>0.
\end{equation}
The multiple is positive because $Q(p_0)$ is positive definite and
$z\cdot\nu_\rho(z)>0$.

Choose $q_0$ in $\An$ which is not collinear
with $p_0$. On the nonempty relatively open set
\[
    \{z\in\partial B_1^\rho(0):p_0\cdot z>0,\ q_0\cdot z>0\},
\]
\eqref{eqn: weiss hessian normal} implies that
$Q(q_0)Q(p_0)^{-1}z$ is parallel to $z$.
The positive rays through this set form an open subset of $\R^n$.
A linear map for which an open set of vectors consists of eigenvectors
is a scalar multiple of the identity. Thus
$Q(q_0)$ is a positive scalar multiple of $Q(p_0)$.
The same argument applied to $q_0$ and $-p_0$ shows that
$Q(-p_0)$ is also a positive scalar multiple of $Q(p_0)$.
Therefore, with $Q:=Q(p_0)$, equation~\eqref{eqn: weiss hessian normal}
and continuity imply
\[
    \nu_\rho(z)\in(0,\infty)Q^{-1}z
    \qquad\text{for every }z\in\partial B_1^\rho(0).
\]
The tangential derivative of $z\mapsto z\cdot Q^{-1}z$ consequently
vanishes on $\partial B_1^\rho(0)$. This hypersurface is connected, since it
is a radial graph over $S^{n-1}$. Hence
\[
    z\cdot Q^{-1}z=R_0^2\qquad\text{on }\partial B_1^\rho(0)
\]
for some $R_0>0$. Since every positive ray meets $\partial B_1^\rho(0)$
exactly once, setting $L:=R_0Q^{1/2}$ gives
\[
    \partial B_1^\rho(0)=L S^{n-1},\qquad \rho(x)=|L^{-1}x|.
\]

\emph{Step 4: Reduction to the Euclidean case and the converse.}
For $\theta\in S^{n-1}$, set
\[
    \widetilde\Phi(q):=\Phi(L^{-T}q),
    \qquad
    \widetilde h(\theta,s):=|L^{-T}\theta|h(L\theta,s).
\]
Since
\[
    \nu_\rho(L\theta)=\frac{L^{-T}\theta}{|L^{-T}\theta|},
\]
a direct computation yields
$$(q\cdot\theta-s)
      \bigl(D\widetilde\Phi(q)\cdot\theta
                     -\partial_s\widetilde h(\theta,s)\bigr)
     =|L^{-T}\theta|\,
        \sigma_{\Phi,\rho,h}(L^{-T}q,L\theta,s)\geq0.
$$
Applying the Euclidean argument from Step~2 gives
$\widetilde\Phi(q)=c|q|^2$ and
$\widetilde h(\theta,s)=cs^2$. These are precisely
\eqref{eqn: weiss quadratic classification}.

Conversely, if \eqref{eqn: weiss quadratic classification} holds,
then
\[
    \sigma_{\Phi,\rho,h}(p,L\theta,s)
    =\frac{2c}{|L^{-T}\theta|}
       (p\cdot L\theta-s)^2\geq0.
\]
This proves the equivalence and all the assertions.
\end{proof}

\section{Regularity and compactness for minimizers of the penalized functional}\label{sec: reg}
Minimizers of the penalized functional $\J_{f,\mu}$ introduced in \eqref{def: J f mu} enjoy many of the regularity properties of minimizers of the classical Alt-Caffarelli functionals,  and sequences of minimizers have good compactness properties. 

\subsection{Geometric properties of minimizers}\label{sec: geom reg}
We start with the classical Lipschitz and non-degeneracy estimates and their consequences. Throughout this subsection, we let $\Phi \in C^{1,1}(\R^n)$ have $\Phi\ge 0$ in $\R^n$ and $\Phi(0)=0$ and satisfy the ellipticity bound \eqref{eqn: ellipt}, without assuming any further regularity of $\Phi$ on an annulus. The estimates of this subsection depend only on $\La$ and $n$, not $\alpha$ or $\C$,  and will determine the  annulus $\An$ in Definition~\ref{def: annulus} below.

These results were established by Petrosyan and Valdinoci \cite{PV05} for the non-perturbed functional $\J$.  In \cite{PV05} the authors consider integrands $\Phi$ with more general $p$-Laplace type degeneracy, and hence assume $\Phi$ is $C^2$ away from the origin so they can apply interior gradient estimates of \cite{Tolks}. Instead, when $\Phi\in \II$, and in particular satisfies the uniform ellipticity  \eqref{eqn: ellipt}, estimate \eqref{eqn: gradient reg interior} of Lemma~\ref{lem: C1beta} takes the place of \cite{Tolks}. For minimizers of the perturbed functional  $\Jfm$ rather than $\J$, the proofs from \cite{PV05} go through with minor modification and thus we do not repeat them here.

We first recall the Lipschitz bound, compare \cite[Theorem 4.1]{PV05}.

\begin{lemma}\label{lem:interior-lip}
Let $u$ minimize $\Jfm$ in $B_1$, with $0\leq\mu\leq1$, $f\geq0$, and $\|f\|_{L^\infty(B_1)}\leq F$.
For $0<\rho<\rho'<1$, if $\|u\|_{L^\infty(B_{\rho'})}\leq K$, then
\[
 \|\nabla u\|_{L^\infty(B_\rho)}\leq C(n,\La,F,\rho,\rho')(1+K).
\]
If $\{u=0\}\cap B_{1/4}\ne\varnothing$, then
\[
 \|u\|_{L^\infty(B_\rho)}+\|\nabla u\|_{L^\infty(B_\rho)}\leq C(n,\La,F,\rho).
\]
\end{lemma}

Next, we recall nondegeneracy of minimizers of \(\Jfm\); compare \cite[Lemma 5.2]{PV05}.

\begin{lemma}
    \label{lem:nondeg}
Let \(u\) be a minimizer of \(\Jfm\) in \(\Omega\) and let
\(x_0\in\{u>0\}\). For
\(0<r<\min\{1,\operatorname{dist}(x_0,\partial\Omega)\}\), denote by
\(U_r\) the connected component of \(\{u>0\}\cap B_r(x_0)\) containing
\(x_0\). Then
\[
    \sup_{U_r\cap B_{r/2}(x_0)}u
    \geq
    c\bigl(n,\La,\|f\|_{L^\infty(\Omega)}\bigr)\,r .
\]
In particular, for every connected component \(U\) of \(\{u>0\}\cap\Omega\)
and every \(x_0\in U\), we have \(\sup_{U\cap B_{r/2}(x_0)}u\geq c\,r\) for
such \(r\).
\end{lemma}

By combining the Lipschitz and nondegeneracy estimates, we obtain volume density estimates for the positivity set of a minimizer $u$, compare \cite[Theorem 2.2]{PV05}.

\begin{corollary}\label{cor: density estimates}
Let \(u\) be a minimizer of \(\Jfm\) in \(\Omega\). Then for any
\(x_0\in\GU\) and
\(0<r<\frac12\min\{1,\operatorname{dist}(x_0,\partial\Omega)\}\), we have
\[
    c\bigl(n,\La,\|f\|_{L^\infty(\Omega)}\bigr)
    \leq
    \frac{|\{u>0\}\cap B_r(x_0)|}{|B_r(x_0)|}
    \leq
    1-c\bigl(n,\La,\|f\|_{L^\infty(\Omega)}\bigr) .
\]
Moreover, for any compact subset \(\K\Subset\Omega\), we have
\(\Hm^{n-1}(\GU\cap\K)<\infty\).
\end{corollary}

\begin{remark}\label{rmk: connected components}
 Suppose $u$ is a minimizer of $\J$ in $B_1$  such that $\{ u>0\}\cap B_1$ has $k$ connected components $V_1,\dots, V_k$, and on each connected component, 
\[
u(x) = (x\cdot \xi_i + c_i)_+ \qquad \text{ for }x \in V_i,
\]
for some $c_i\geq 0$ and $\xi_i \in \R^n$. Suppose $c_1 =0$ so $0\in \GU$.  By Corollary~\ref{cor: density estimates} and disjointness of the connected components, there exists $\rho_0=\rho_0(n, \La)>0$ such that $V_1$ is the only connected component intersecting $B_{\rho_0}$. In particular,
\[u = (x\cdot \xi_1)_+\qquad \text{ in } B_{\rho_0}.
\]
The first variation \eqref{eqn: EL} then guarantees that $D\Phi(\xi_1 ) \cdot \xi_1 -\Phi(\xi_1)=1.$
\end{remark}

The following lemma is the key tool that allows us to focus on regions where $|\nabla u|$ has a positive lower bound, since this always captures a definite ball in the
positivity set.

\begin{lemma}\label{lem: ball with big gradient}
Suppose that $u$ minimizes $\E^\Phi$ in $B_1$ and $0\in\GU$. There is $\db=\db(n,\La)>0$ such that every connected component $V$ of $B_1\cap\{u>0\}$ which intersects $B_{1/8}$ contains a ball
\[
 B_{\db}(y_V)\subset V\cap B_{1/2} \cap\{|\nabla u|\geq8\db\}.
\]
We may require $\db<\min\{1/64,\sqrt{2/\La}/16\}$.
\end{lemma}
\begin{proof}
Let $L=L(n,\La)$ bound the Lipschitz constant of $u$ on $B_{3/4}$.
Choose $x\in V\cap B_{1/8}$. Componentwise nondegeneracy, applied
in $B_{1/8}(x)$, gives $q\in V\cap B_{1/4}$ with $u(q)\geq c_0$,
where $c_0>0$ depends only on $n,\La$.
Let $z$ be a closest zero of $u$ to $q$, and put $d=|q-z|$.
Since $u(0)=0$,
\[
 c_0/L\leq d\leq1/4,
 \qquad B_d(q)\subset V\cap B_{1/2}.
\]
Set $t_0=c_0/(4L)$ and let $y=z+t_0(q-z)/d$. Then
$u(y)\leq c_0/4$. The mean value theorem on $[y,q]$ provides
$a\in(y,q)$ with $|\nabla u(a)|\geq3c_0$.
Moreover $B_{t_0}(a)\subset B_d(q)$.
The interior $C^{1,\beta}$ estimate \eqref{eqn: gradient reg interior} of Lemma~\ref{lem: C1beta} therefore gives a radius
$\rho=\rho(n,\La)>0$, with $\rho<t_0/2$, such that
$|\nabla u|\geq2c_0$ on $B_\rho(a)$.
Choose $\db\leq\min\{\rho,c_0/4,1/64,\sqrt{2/\La}/16\}$,
decreasing it slightly if necessary. The result follows.
\end{proof}

With these lemmas in hand, we can now specify the radii $0<r_\La<R_\La$ that define the annulus $\An$ in \eqref{eqn: holder intro}.
\begin{definition}\label{def: annulus}
Let $\db$ be fixed as in Lemma~\ref{lem: ball with big gradient}.
Choose $L_*=L_*(n,\La)\geq1$ so that every minimizer of
$\E^\Phi$ in $B_2$ with $0\in\GU$ satisfies
$\|\nabla u\|_{L^\infty(B_{15/8})}\leq L_*$, and set
\[
 r_\La=\db/2,
 \qquad R_\La=4\max\{L_*,\sqrt{2\La},1\},
 \qquad \An=B_{R_\La}\setminus\overline{B_{r_\La}}.
\]
\end{definition}

\begin{remark}\label{rmk: annulus}
Thanks to \eqref{eqn: first annular constraint}, any $p\in\R^n$ satisfying the identity $D \Phi(p)\cdot p -\Phi(p)=1$ corresponding to the free boundary condition \eqref{eqn: EL} lies in a compact subannulus of $\An$.
 Moreover, for any minimizer of $\J$ in $B_2$ with $0\in\G$, for each point $x$ in each ball $B_{\db}(y_V)$ obtained in Lemma~\ref{lem: ball with big gradient},  $\na u(x)$ lies in a compact subset of $\An$.
\end{remark}

\subsection{Compactness properties of minimizers}

The Lipschitz and nondegeneracy estimates of the previous subsection imply
various compactness properties for minimizers of our penalized functionals.
All statements are formulated on balls, which is all that is needed in the sequel, but could of course be shown on general domains. The results of this section also hold for the more general class of integrands considered in Section~\ref{sec: geom reg}, but now that the annulus $\An$ has been specified in Definition~\ref{def: annulus}, for the remainder of the paper we restrict our attention to $\Phi \in \II$ for simplicity.

\begin{lemma}\label{lem:compactness}
Let \(\{\Phi_k\}\subset \II\), \(\{\mu_k\}\subset[0,1]\), and let
\(\{f_k\}\subset L^\infty(B_1)\) be nonnegative with
\(\sup_k\|f_k\|_{L^\infty(B_1)}\leq F<\infty\). For each \(k\), let \(u_k\)
be a minimizer of \(\mathcal{E}^{\Phi_k}_{f_k,\mu_k}(\cdot,B_1)\) with
\[
    \{u_k=0\}\cap B_{1/4}\neq\emptyset .
\]
Then, for every \(\rho\in(0,1)\),
\begin{equation}\label{eq:cpt-unif-lip}
    \sup_k\Bigl(\sup_{B_\rho}u_k
    +\operatorname{Lip}(u_k;B_\rho)\Bigr)
    \leq C(n,\La,F,\rho) ,
\end{equation}
and there is a nonnegative \(u\in C^{0,1}_{\mathrm{loc}}(B_1)\) such that, up
to a subsequence (unrelabeled),
\begin{equation}
    \label{eqn: conv from cpt}
    u_k\to u\ \text{ locally uniformly in }B_1,
    \qquad
    u_k\rightharpoonup u\ \text{ weakly in }H^1(B_\rho)
    \ \text{ for every }\rho\in(0,1) .
\end{equation}
\end{lemma}
 
\begin{proof}
The bound \eqref{eq:cpt-unif-lip} is Lemma~\ref{lem:interior-lip}. By the
Arzel\`a-Ascoli theorem and a diagonal argument over an exhaustion
\(B_{\rho_m}\uparrow B_1\), a subsequence converges locally uniformly to a
nonnegative \(u\), which is locally Lipschitz with the same constants. For
each \(\rho<1\), the sequence \(u_k\) is bounded in \(H^1(B_\rho)\) by
\eqref{eq:cpt-unif-lip}, so, after a further diagonal extraction,
\(u_k\rightharpoonup u\) weakly in \(H^1(B_\rho)\) for every \(\rho<1\).
\end{proof}

\begin{lemma}\label{lem:compactness2}
Suppose \(\Phi_k,\Phi\in \II\) and
\(\Phi_k\to\Phi\) locally uniformly in \(\mathbb R^n\). Suppose \(\{f_k\}\subset C^0(\overline{B_1})\) are
nonnegative with \(f_k\to f\) uniformly in \(B_1\), and
\(\{\mu_k\}\subset[0,1]\) with \(\mu_k\to\mu\). For each \(k\), let \(u_k\)
be a nonnegative minimizer of \(\mathcal{E}^{\Phi_k}_{f_k,\mu_k}\) in \(B_1\), and
suppose that \(u_k\to u\) locally uniformly and weakly in
\(H^1_{\mathrm{loc}}(B_1)\) (by Lemma~\ref{lem:compactness}, this holds
along a subsequence whenever \(\{u_k=0\}\cap B_{1/4}\neq\emptyset\)). Then:
\begin{enumerate}
    \item \(u_k\to u\) strongly in \(H^1_{\mathrm{loc}}(B_1)\);
    \item \(u\) is a local minimizer of \(\Jfm\) in \(B_1\).
If, in addition, \(u\in H^1(B_1)\), then \(u\) is a minimizer
with respect to all perturbations in \(H^1_0(B_1)\);
    \item \(\{u_k>0\}\to\{u>0\}\) locally in the Hausdorff
distance, and
\(\chi_{\{u_k>0\}}\to\chi_{\{u>0\}}\) in
\(L^1_{\mathrm{loc}}(B_1)\);
    \item if \(\{x_k\}\Subset B_1\) is a sequence with
    \(x_k\in\partial\{u_k>0\}\) for each \(k\) and \(x_k\to x\), then
    \(x\in\partial\{u>0\}\). Conversely, every point of \(\GU\) is a limit of points
of \(\G_{u_k}\). In particular,
\(\G_{u_k}\to\GU\) locally in the Hausdorff distance;
    \item for any \(r\in(0,1]\) and \(p\in(1,\infty)\),
    \(\liminf_k\int_{\{u_k>0\}\cap B_r}|D^2u_k|^p
    \geq\int_{\{u>0\}\cap B_r}|D^2u|^p\).
\end{enumerate}
\end{lemma}

 The proof of Lemma~\ref{lem:compactness2} is postponed to Appendix~\ref{sec: cptness proof}.
\begin{remark}\label{rem:cv-limit-min}
By item (2), all the results of the previous sections apply to the limit
\(u\): in particular, \(u\) satisfies the Lipschitz estimate, the
nondegeneracy of Lemma~\ref{lem:nondeg}, and the density estimates
of Corollary~\ref{cor: density estimates}; the latter imply
\(|\partial\{u>0\}\cap B_1|=0\).
\end{remark}

\subsection{Minimizers are viscosity solutions}

Next we show that any minimizer of $\Jfm$ is a viscosity solution to the associated Euler-Lagrange equation. The following proof is inspired by an argument from anisotropic minimal surfaces; see 
\cite{Hardt77, philippis2015regularity, DPDRminmax,DePhilippisDeRosaHirsch, DPDRYminmax, DRNR}. Throughout this section, $\Phi \in \II$ is fixed, $f\in C^\infty(\Omega)\cap L^\infty(\Omega)$ is nonnegative, and $\mu \in [0,1]$. We again use the short-hand $\Lo w= \Div(D \Phi(\na w))$ and $\cG(p) = D \Phi(p)\cdot p - \Phi(p)$ from \eqref{eqn: short hand}.

\begin{proposition}\label{prop: min is viscosity solution} If $u$ is a  minimizer of $\Jfm$ in $\Omega$, then $u$ is a viscosity solution to \eqref{eqn: e reg sec EL} with $g = 2\mu(\mu u -f)$. 
\end{proposition}
\begin{proof}[Proof of Proposition~\ref{prop: min is viscosity solution}]
 {\it Step 1:} We begin by showing $u$ satisfies \eqref{eqn: visc int cond}.  Suppose $x_0 \in \{u>0\}\cap \Omega$ and $\phi\in C^\infty(\R^n)$ touches $u$ from below in $\{u>0\}$ at $x_0$.  Recall that $u \in C^{1,\alpha}(\{u>0\})\cap C^{2,\alpha}(\{x: \na u(x) \in \An\}).$

   {\it Case 1: }
  First suppose that $\na u(x_0) \in \An$, and hence $u$ is $C^2$ in a neighborhood of $x_0$ and $D^2 u(x_0)- D^2\phi(x_0)$ is nonnegative definite. Since $\na \phi(x_0) = \na u(x_0):=p_0$, we see that  $\Lo u(x_0 )- \Lo\phi(x_0) = D^2\Phi(p_0) : (D^2 u(x_0)-D^2\phi(x_0)) \geq 0$ by ellipticity. Since $\Lo u(x_0)=g(x_0)$, it follows that $\Lo\phi(x_0) \leq g(x_0).$ 

  {\it Case 2:} Next suppose that $\na u(x_0)\not\in \An$.
  We assume by way of contradiction that $\mathcal{M}^-(D^2 \phi(x_0) ,1/\Lambda, \Lambda)>g(x_0)$. Since $\phi$ is smooth and $g$ is continuous by Lemma~\ref{lem:interior-lip}, we may choose $r_0>0$ so that $\mathcal{M}^-(D^2 \phi(x) ,1/\Lambda, \Lambda)\geq g(x) + r_0$ for all $x$ in $B_{r_0}(x_0)$. 
  This also implies that $\Lo \phi(x) \geq g(x)+r_0$ weakly in $B_{r_0}(x_0)$. 
Indeed, letting $\Phi_\e$ be a scale-$\e$ mollification of $\Phi$ (which satisfies the same ellipticity bounds), for any nonnegative test function $\psi \in C^1_c(B_{r_0}(x_0))$ we have 
\[
-\int D\Phi_\e(\na \phi) \cdot \na \psi = \int (D^2\Phi_\e(\na \phi):D^2\phi)\,\psi   \ge
\int \mathcal{M}^-(D^2 \phi ,1/\Lambda, \Lambda)\,\psi  \geq \int (g+ r_0) \psi.
\]
Passing $\e\to 0$ shows that $\Lo \phi(x) \geq g(x)+r_0$ weakly in $B_{r_0}(x_0)$. By taking a small vertical translation of $\phi$, which leaves the equation unchanged, this contradicts the comparison principle for weak solutions. Thus $\mathcal{M}^-(D^2 \phi(x_0) ,1/\Lambda, \Lambda)\leq g(x_0)$.

Pairing this with an analogous argument when $\phi\in C^\infty(\R^n)$ touches $u$ from above in $\{u>0\}$ at $x_0$ shows that $u$ satisfies $\Lo u=g$ in $\{u>0\}$ in the viscosity sense. 

{\it Step 2:} Next we show that $u$ satisfies \eqref{eqn: visc BC}. The sub and super solution arguments are analogous, so we just consider the case when $x_0 \in \GU$ and $\phi$ touches $u$ from below.  Suppose by way of contradiction that $\cG(\na \phi(x_0)) >1.$ Write $\na \phi(x_0) = a_*e$ with $a_*>0$ and $e \in S^{n-1}$. Recall from Section~\ref{sec: prelim integrands} that $t\mapsto \cG(te)$ is strictly increasing, so there is a unique $b>0$ such that $\cG(be)=1$ and $b<a_*.$ 

Now, consider a blowup sequence $u_k = u(x_0 + r_kx)/r_k$ with $r_k \to0$. For each $k$, $u_k$ minimizes $\J_{f_k,\mu_k}$ in $B_{r_k^{-1}}$, where $\mu_k, f_k$ are determined by the rescaling \eqref{eqn: rescale}. So, by Lemmas~\ref{lem:compactness} and \ref{lem:compactness2}, up to a subsequence, $u_k$ converges to a function $u_\infty:\R^n \to \R$ with $0 \in \G(u_\infty)$ such that
\begin{equation}\label{eqn: ordered}
    u_\infty(x) \geq a_*(x\cdot e)_+ \qquad \text{ for all }  x \in \R^n,
\end{equation}
and $u_\infty$ is an entire minimizer of $\J$. Indeed, by Lemma~\ref{lem:compactness2}, for each $R>0$, $u_\infty$ minimizes $\E^\Phi(\cdot; B_R) + |f(x_0)|^2|B_R|$ and thus minimizes $\E^\Phi(\cdot; B_R)$. 

Let $\mathcal{B}(u_\infty)$ denote the collection of all blowups of $u_\infty$ at $0$. The argument above shows that every $v \in \mathcal{B}(u_\infty)$ is again an entire minimizer of $\J$  with $0 \in \G(v)$ that satisfies \eqref{eqn: ordered} with $v$ in place of $u_\infty$. In particular, $ \{y\cdot e>0\}\subset \{v>0\}$ for  every $v \in \mathcal{B}(u_\infty)$. A diagonal argument shows that $\mathcal{B}(u_\infty)$ is closed, and hence Lemma~\ref{lem:compactness} shows $\mathcal{B}(u_\infty)$ is sequentially compact, with respect to the convergence of the lemma.

{\it Claim:} There exist $\bar{v} \in \mathcal{B}(u_\infty)$ and $\bar{a}>0$ such that $\bar{v}(x) = \bar{a}(x\cdot e)_+$ for all $x \in \{y\cdot e>0\}.$

Before proving the claim, let us see how it will yield a contradiction. On one hand, \eqref{eqn: ordered} guarantees that $ \bar{a} \geq a_* >b$. On the other hand, the blowup $\bar{v}$ produced by the claim is an entire minimizer of $\J$.  Since the restriction of a minimizer to one component of its positivity set is again a minimizer, this means $\bar{a}(x\cdot e)_+$ is also a minimizer of $\J$. But as noted above, the free boundary condition cannot be satisfied unless $\bar{a} = b$, a contradiction.

Let us now prove the claim. For each $v \in \mathcal{B}(u_\infty),$ let 
\[
a(v) = \sup\{a: a(x\cdot e)_+ \le v(x) \text{ for all } x\in \R^n\}.
\]
By \eqref{eqn: ordered}, $ a(v) \geq a_*,$ while the Lipschitz bound of Lemma~\ref{lem:interior-lip} guarantees that $a(v)\leq L$ for $L=L(n,\La)$. 
Moreover, one easily checks that $v \mapsto a(v)$ is upper semicontinuous with respect to local uniform convergence on $\R^n.$
So, $\bar{a}:=\sup\{a(v) : v\in \mathcal{B}( u_\infty)\}$ is achieved by some $\bar{v} \in \mathcal{B}(u_\infty)$. It is for this $\bar{v}$ that we establish the claim.

Suppose by way of contradiction that $\bar{v}(x) -\bar{a}(x\cdot e)_+$
 is not identically zero in $\{x\cdot e>0\}$, so there exists some $\bar{x}\in \{x\cdot e>0\}$ such that $\bar{v}(\bar{x}) > \bar{a}(\bar{x}\cdot e)_+.$  Let $\bar{r}= |\bar{x}|$ and let $B^+ = \{ x  : |x|<\bar{r} , x\cdot e>0\}$. Now, we can find a function $w \in C^{1,1}(B^+) \cap \text{Lip}(\partial B^+)$ such that   $w=0$ on $\partial B^+ \cap \{x\cdot e=0\}$ and
$ \bar{a}(x\cdot e)_+ \leq w(x) \leq \bar{v}(x)$ on $\partial B^+\cap \{x\cdot e>0\}$, with strict inequality at $\bar{x} \in \partial B^+\cap \{x\cdot e>0\}$.

Now, let $\tilde{w}$ solve $\divv(D\Phi(\na \tilde{w}) ) =0$ in $B^+$ with $\tilde{w}=w$ on $\partial B^+.$ Since $B^+ \subset \{ \bar{v}>0\}$, $\bar{v}$ solves the same equation on $B^+$ and by the maximum principle, $\tilde{w} \leq \bar{v}$ in $B^+.$ Moreover, $\tilde{w} \in C^{1,\beta} (\bar{B}^+_{\bar{r}/2})$ by estimate \eqref{eqn: gradient reg boundary} of  Lemma~\ref{lem: C1beta}, and in particular $\na \tilde{w}(0)$ exists. Note that $\na \tilde{w}(0) = \tilde{a} e$ for some $\tilde{a}>0$ thanks to our choice of boundary data. By strong maximum principle, $\tilde{w} >\bar{a}(x\cdot e)_+$ in $B^+$ and then by Hopf lemma, $\tilde{a}>\bar{a}$. But then since $\hat{v}\geq \tilde{w}$ in $B^+,$ this means that we can find a blowup $\hat{v}$ of $\bar{v}$ (which is also a blowup of $u_\infty)$ with $a(\hat{v}) \geq \tilde{a} >\bar{a}$. This contradicts the definition of $\bar{a}$ and proves the claim.
 \end{proof}

\subsection{Epsilon regularity for viscosity solutions}
\label{sec: e reg}
Next, viscosity solutions to \eqref{eqn: e reg sec EL}, which we recall were defined in Section~\ref{sec: prelim}, enjoy an $\e$-regularity theorem.
Throughout this section, $\Phi \in \II$, and we continue to use the short-hand $\Lo w= \Div(D \Phi(\na w))$ and $\cG(p) = D \Phi(p)\cdot p - \Phi(p)$ as in \eqref{eqn: short hand}. We consider viscosity solutions $u$ to \eqref{eqn: e reg sec EL} in $B_r$ with $0 \in \Gamma_u$ that are flat in the sense that
\begin{equation}
    \label{eqn: flat 1}
\left(x \cdot e -\bar{\e}r\right)_+ \leq u(x) \leq \left(x \cdot e +\bar{\e} r\right)_+\qquad \text{ in } B_r
\end{equation}
 for some $e \in \R^n$ with $\cG(e)=1$ and such that 
\begin{equation}\label{eqn: small RHS}
r \|g\|_{L^\infty} \leq \bar{\e}.
\end{equation}
For such solutions, we have the following.
\begin{theorem}[$\e$-Regularity]\label{thm: eps reg} 
There are universal constants $\bar{\e}>0$ and $\delta>0$ such that the following holds.
Fix $r>0$, $g \in C(B_r)\cap L^\infty(B_r)$, and $e \in \R^n$ with $\cG(e)=1$. 

Suppose $u$ is a viscosity solution to \eqref{eqn: e reg sec EL} in $B_r$ with $0 \in \Gamma_u$ satisfying \eqref{eqn: flat 1} and \eqref{eqn: small RHS}.
Then $\GU$ is $C^{1,\beta}$  for some universal $\beta\in (0,1/2)$  at scale $\delta r$ near the origin.
\end{theorem}
 More precisely, let $p, q$ be the projections onto $\{x\cdot e =0\} \simeq \R^{n-1}$ and $\{x\cdot e = \pm |x||e|\} \simeq \R$ respectively and let $B_r' =p(B_{r})$. The conclusion of Theorem~\ref{thm: eps reg} says there is a function $h: B_{\delta r}'\to [-\delta r,\delta r]$  such that 
\[
\{u>0\}\cap C_{e}(\delta r) = \{x=(px, qx)\ : \ px \in B_{\delta r}',\  qx >h(px)\}
\]
in the cylinder $C_{e}(\delta r):= \{ (px, qx) : |px|< \delta r,\,  |qx|<\delta r\}$ and $\|h(r\cdot )/r\|_{C^{1,\beta}(B_{\delta}')}\leq C$ for a universal constant $C>0$. 

Theorem~\ref{thm: eps reg} is proven using De Silva's method \cite{DeSilva}. 
This flexible approach has been used to prove $\e$-regularity theorems for viscosity solutions of one- and two-phase problems with various operators and right-hand sides. The equation \eqref{eqn: e reg sec EL} does not precisely fit into the existing frameworks since $\Phi$ is only $C^{1,1}$ outside of $\An$, thus $D^2\Phi$ is not defined everywhere. However, $\Phi$ is $C^{2,\alpha}$ in the annulus $\An,$ which includes in particular a neighborhood of $\{ p \in\R^n : \cG(p)=1\}$ as noted in Remark~\ref{rmk: annulus}. Thus, the proof goes through without major modifications from, e.g.  \cite{DeSilva, DFS15,CV}.

We sketch the proof in Appendix~\ref{sec: e reg proof}, emphasizing the places where the proof differs from the standard one.

For minimizers $u$ of $\Jfm$, the assumed $C^{2,\alpha}$ regularity of $\Phi$ on $\An$ means that flat free boundaries are $C^{2,\alpha}$.
\begin{corollary}\label{cor: higher reg}
    With the same assumptions and notation of Theorem~\ref{thm: eps reg}, if $g \in C^{0,\alpha}(B_r)$, then the function $h$ is $C^{2,\alpha}$ in $B'_{\delta r/2}$. In particular, this holds for minimizers of $\Jfm$ with  $g =2\mu(\mu u-f)$ for $\mu \in [0,1]$ and $f \in C^\infty(B_r)\cap L^\infty(B_r)$ nonnegative.
\end{corollary}
\begin{proof}[Proof of Corollary~\ref{cor: higher reg}]
By Theorem~\ref{thm: eps reg}, the free boundary condition $\cG(\na u )=1$ is achieved classically in $B_{\delta r}(0)$. In particular, by Remark~\ref{rmk: annulus}, $\{ \na u(x)  : x \in \GU \cap {B_{\delta r}(0)}\}$ is contained in a compact subset of $\An.$
The proof then follows from a classical argument using the hodograph transform \cite{KN}. 
To see the ``in particular,'' if $u$ is a minimizer of $\Jfm$, then  $u$ is a viscosity solution to \eqref{eqn: e reg sec EL} with $g =2\mu(\mu u -f)$ by Proposition~\ref{prop: min is viscosity solution}. Since $f$ is assumed to be (qualitatively) smooth and $u$ is Lipchitz thanks to Lemma~\ref{lem:interior-lip}, this choice of $g$ is H\"{o}lder  continuous for any $\beta\in (0,1)$. 

{For the unperturbed problem, the argument above is quantitative. More precisely, there are universal constants $c,C>0$ such that
$\|h(r\,\cdot)/{r}\|_{C^{2,\alpha}(B'_c)}\leq C$.
Indeed, after normalizing $r=1$, Theorem~\ref{thm: eps reg} gives a uniformly controlled $C^{1,\beta}$ graph. The boundary gradient estimates and the Bernoulli condition give a uniformly nondegenerate hodograph map on a fixed smaller neighborhood. The Schauder estimates in the hodograph argument are therefore uniform, since the ellipticity constants and the annular $C^{2,\alpha}$ bounds for $\Phi$ are uniform.
}
\end{proof}

A direct consequence of Theorem~\ref{thm: eps reg} is the following propagation of singularities statement; we omit its standard proof. In this statement, $\Hm^s_\infty$ is the Hausdorff content.
\begin{corollary}
    \label{cor: prop of singularities} 
   Suppose $\Phi_k, \Phi \in \II$ with $\Phi_k \to \Phi$, $f_k, f \in C^\infty(B_1)\cap L^\infty(B_1)$ nonnegative with $f_k \to f$ uniformly, and take $\mu_k,\mu \in [0,1]$ with $\mu_k\to \mu$. Let $u_k, u$ be minimizers of $\E^{\Phi_k}_{f_k,\mu_k}$ and $\Jfm$ respectively in $B_1$. If $u_k \to u$ in the sense of Lemma~\ref{lem:compactness2} above and $x_k \in \Sing(u_k)\cap B_1$ with $x_k \to x\in B_1$, then $x \in \Sing(u)\cap B_1.$ Moreover, for any compact set $K\Subset B_1$ and $s \geq 0$, if $\Hm_\infty^s(\Sing(u_k)\cap K) \geq c$, then $\Hm_\infty^s(\Sing(u)\cap K)\geq c.$
\end{corollary}
\section{$L^q$ Hessian estimates in the space of good integrands}\label{sec: curvature estimates}

As described in the introduction, the key ingredients in the proof of Theorem~\ref{thm: main} are a priori $L^q$ estimates for the Hessian akin to those in Theorem~\ref{thm: curvature estimate intro}. In this section, we establish these curvature estimates for minimizers of $\J$ (and its perturbations) when $\Phi$ belongs to the restricted class of integrands $\Istarz \cap C^\infty(\R^n)$. Later on, in Section~\ref{sec: main theorem proof 0} we will extend these estimates to minimizers of $\J$ for any $\Phi \in \II$, establishing Theorem~\ref{thm: curvature estimate intro}.

\subsection{$L^2$ Hessian Estimates}
We begin by proving an $L^2$ estimate for the Hessian; compare to \eqref{eqn: curvature estimate intro}. Note that this estimate holds for minimizers of the penalized functional $\Jfm$.
\begin{proposition}
    \label{prop: curvature estimates}  Fix $\Phi \in \Istarz {\cap C^\infty(\R^n)}$, a nonnegative function $f \in C^\infty(\Omega)\cap L^\infty(\Omega)$ , and $\mu \in [0,1]$. 
There is $C>0$ depending on the universal parameters and on an upper bound for $\|f\|_{L^\infty(\Omega)}$ such that if $u$ is a minimizer of $\Jfm$ in $\Omega$, then for all $p \in \GU$ and $r\in(0,1]$ with $B_{2r}(p)\subset \Omega$, we have 
    \begin{equation}
        \label{eqn: curvature estimates}
        \int_{\{u>0\} \cap B_r(p)} |D^2 u|^{2} \leq Cr^{n-2}.
    \end{equation}
\end{proposition}

 To prove Proposition~\ref{prop: curvature estimates},
we need two preparatory lemmas. 
The first is a Bochner-type inequality for the function $G=\Phi(\na u)+1$. Similar computations have been carried out in various settings; see e.g. \cite[Theorem 4]{WangXia11}.
\begin{lemma}\label{lem: G eqn}
{Let $\Phi \in \II \cap C^\infty(\R^n)$.} 
Let $\Omega \subset \R^n$ be an open bounded set, let $f \in C^\infty(\Omega)\cap L^\infty(\Omega)$ and let $\mu \in [0,1]$. Suppose $u \in W^{1,2}(\Omega) $ is a continuous nonnegative function satisfying 
    \begin{equation}
    \label{eqn: G comp eqn for u}
    \begin{cases}
         \divv(D\Phi(\nabla u))=2\mu(\mu u- f) & \text{ in } \{u>0\} \cap \Omega \\
         D \Phi(\na u ) \cdot \na u -\Phi(\na u) = 1 & \text{ on } \Reg(u)\,.
    \end{cases}
\end{equation}
Then the function $G =1+ \Phi(\na u)$ satisfies
\begin{equation}\label{eqn: G equation}
 \begin{cases}
         \divv(D^2\Phi(\nabla u)\nabla G) \geq \Lambda^{-2}|D^2u|^2
 -2\mu \na\Phi(\nabla u)\cdot\nabla f & \text{ on }  \{u>0\} \cap \Omega\\
 D^2\Phi(\nabla u)[\nabla G, \nu]  = (D \Phi(\na u) \cdot \nu) \, (D^2\Phi(\nabla u)D^2 u[\nu] \cdot \nu )& \text{ on }\Reg(u)\,.
    \end{cases}
\end{equation}
Here $\nu$ is the outer unit normal to $\{u>0\} \cap\Omega$ at $x$.
\end{lemma}

\begin{proof} We let subscripts denote partial derivatives. Thanks to the smoothness assumptions on $\Phi$ and $f$, all derivatives are classical. 
Keeing in mind that $\partial_iG=D\Phi(\nabla u)\cdot \nabla u_i$, we directly compute in $\{u>0\} \cap \Omega$
\begin{align*}
 \divv(D^2\Phi(\nabla u)\nabla G)
 &=\sum_k\partial_{k}
 \big(\sum_j\Phi_{kj}(\nabla u)G_j\big)
 =\sum_k\partial_{k}
 \big(\sum_{j,m}\Phi_{k,j}(\nabla u)\Phi_m(\nabla u)u_{mj}\big)\\
 &=\sum_{k,j,m}\Phi_m(\nabla u)
\partial_{k}\bigl(\Phi_{kj}(\nabla u)u_{mj}\bigr)+\sum_{k,j,m}\Phi_{kj}(\nabla u)u_{mj}
\partial_{k}\bigl(\Phi_m(\nabla u)\bigr)=:I+II.
\end{align*}
For term $I$, we use the interior equation \eqref{eqn: G comp eqn for u} satisfied by $u$. Differentiating \eqref{eqn: G comp eqn for u} in the $m$-direction, we find that 
$ \divv(D^2\Phi(\nabla u)\nabla u_m)
 =2\mu(\mu \, u_m- f_m)$
in $\{u>0\} \cap \Omega$. Recalling from Section~\ref{sec: prelim integrands} that $D \Phi(p)\cdot p \ge 0$, we thus find
\begin{equation}
    \label{eqn: G comp term 1}
    \begin{split}
 I =\sum_m\Phi_m(\nabla u)
 \operatorname{div}(D^2\Phi(\nabla u)\nabla u_m)
 &=2\mu(\mu \,D \Phi(\nabla u)\cdot\nabla u
 - D \Phi(\nabla u)\cdot\nabla f) \geq -2\mu D \Phi(\nabla u)\cdot\nabla f.
    \end{split}
\end{equation}
 Next, for term $II$, we note that $
 \partial_k\Phi_m(\nabla u)
 =\sum_l\Phi_{ml}(\nabla u)u_{lk}$ in $U$.
Thus,
\[
 II=\sum_{kjml}\Phi_{kj}(\nabla u)u_{mj}\Phi_{ml}(\nabla u)u_{lk} = |D^2\Phi(\na u)^{1/2} D^2 u D^2\Phi(\na u)^{1/2} |^2
 \geq \Lambda^{-2}|D^2u|^2.
\]
 Combining this with \eqref{eqn: G comp term 1} gives the first part of \eqref{eqn: G equation}.

Now, for $x \in \Reg(u),$ write $D \Phi(\na u) =a \nu + b \tau$ where $a = D \Phi (\na u) \cdot \nu$ and $|\tau|=1$ with $\tau \cdot \nu =0$.
Then, differentiating the free boundary condition $D \Phi(\na u ) \cdot \na u -\Phi(\na u) = 1$ in the $\tau $ direction, we find that $ D^2\Phi(\na u) D^2 u[\tau] \cdot \na u =0$. So, keeping in mind that $\na u = \partial_\nu u \nu$ since $u$ is constant on $\G$, we have 
\begin{align*}
 D^2\Phi(\nabla u)[\nabla G, \nu]  = (D \Phi(\na u) \cdot \nu)\, ( D^2\Phi(\nabla u)D^2 u[\nu]\cdot \nu)\,.
\end{align*}
This gives the boundary term in \eqref{eqn: G equation} and completes the proof.
\end{proof}

Next, for $\Phi \in \Istarz$, the assumed bound $\Hm^{n-2}(\Sing(u))=0$ for minimizers $u$ of $\J$ also holds for minimizers of the perturbed functional $\Jfm.$ 
\begin{lemma}\label{lem: perturbed functional small singular set}
    Let $\Phi \in \Istarz$. For a nonnegative function $f \in C^\infty(B_2)\cap L^\infty(B_2)$ and $\mu \in [0,1]$, any minimizer $u$ of $\Jfm$ in $B_2$ satisfies $\Hm^{n-2}(\Sing(u)\cap B_2 )= 0.$
\end{lemma}
\begin{proof}
Let $u$ be a minimizer as in the lemma and suppose by way of contradiction that $\Hm^{n-2}(\Sing(u) \cap B_2) >0$.  By a general measure theory fact, see \cite[Theorem 3.26(2)]{SimonLN}, for $\Hm^{n-2}$-a.e. $x \in \Sing(u)\cap B_2$ we have 
    \begin{equation}
        \label{eqn: content LB}
    \limsup_{r\to0}\frac{\Hm_{\infty}^{n-2}(\Sing(u)\cap B_2\cap B_r(x))}{\omega_{n-2}r^{n-2}} \geq \frac{1}{2^{n-2}}\,.
     \end{equation}
Choose one such point $\bar{x} \in \Sing(u) \cap B_2$, and take a  sequence $r_k\to 0$ such that  $u_k = u(\bar{x} +r_kx)/r_k$ satisfies $\Hm_{\infty}^{n-2}(\Sing(u_k) \cap B_1) \geq \omega_{n-2} 2^{-n}$. By scaling, $u_k$ is a minimizer of $\E^{\Phi}_{f_k, r_k \mu}$ in $B_{r_k^{-1}}$ where $f_k(x) = f(\bar{x} + r_k x)$.

By Lemma~\ref{lem:compactness}, up to a subsequence, $u_k \to u_\infty$ in $B_2$ in the sense of \eqref{eqn: conv from cpt}, and then by Lemma~\ref{lem:compactness2}, $u_\infty$ minimizes $\E^\Phi(\cdot) + |B_{3/2}|f(\bar{x})^2$, and thus also $\J$,  in $B_{3/2}.$  Since $\Phi \in \Istarz$, we know $\Hm^{n-2}(\Sing(u)\cap B_{3/2})=0$. In particular, $\Hm^{n-2}(\Sing(u_\infty)\cap \overline{B_1})=0$ and thus {$\Hm^{n-2}_{\infty}(\Sing(u_\infty)\cap \overline{B_1}) =0$} as well. This contradicts Corollary~\ref{cor: prop of singularities}. 
\end{proof}

We now prove Proposition~\ref{prop: curvature estimates}.
The idea is to use the fact that the $G=1+ \Phi(\nabla u)$ is a subsolution to the linearized equation with right-hand side bounding $|D^2u|^2$ (recall \eqref{eqn: G equation}), and relate this to the stability inequality \eqref{eqn: stability}.  

\begin{proof}[Proof of Proposition~\ref{prop: curvature estimates}]
    By scaling it suffices to assume that $p$ is the origin and $r =1$.
    Moreover, by Lemma~\ref{lem: perturbed functional small singular set} and since $\Phi \in \Istarz,$ we know $\mathcal H^{n-2}(\Sing(u)\cap B_2)=0$. Let $U = \{ u>0\} \cap B_{3/2}$,  let $\G = \partial U\cap B_{3/2}$ and let $\nu$ be the outer unit normal to $U.$ Since $u$ is a minimizer of $\Jfm$ on $B_2$, we have $\|\na u\|_{L^\infty(B_{3/2})} \leq L$ where $L$ depends on $n, \La$ and an upper bound for $\| f\|_{L^\infty(B_2)}$ by Lemma~\ref{lem:interior-lip}.
Throughout the proof, $C>0$ denotes a constant that may change from line to line and depends only on the universal parameters (in fact, only $n$ and $\La$) and an upper bound for  $\| f\|_{L^\infty(B_2)})$.

{\it Step 1: } In the first step we test the equation \eqref{eqn: G equation} for $G$ against $G$ multiplied by a cutoff function that vanishes on $\Sing(u)$. This step does not require smallness of the singular set.
Let $\zeta\in C_c^1(B_{3/2} \setminus \Sing(u)).$ We multiply $\divv(D^2\Phi(\nabla u)\nabla G)$ by $\zeta^2(x)G$ and integrate by parts to obtain 
\begin{equation}\label{eqn: int eqn G1}
\begin{split}
  \int_{U}\zeta^2G\, \, \divv(D^2\Phi(\nabla u)\nabla G)
  = &\int_{\G}\ \zeta^2 G D^2\Phi(\nabla u)[\nabla G, \nu] \, d\Hm^{n-1}
  - \int_{U} 
 \langle D^2\Phi(\nabla u)\nabla G,\nabla (G\zeta^2)\rangle 
= I - II.   
\end{split} 
\end{equation}
We bound the right-hand side above as follows. Using the free boundary condition \eqref{eqn: G equation} for $G$ and regrouping terms, we see that $I$ is coincides with the left-hand side of the stability inequality \eqref{eqn: stability}. {Since $\zeta$ vanishes on $\Sing(u)$, the stability inequality \eqref{eqn: stability} holds for $\zeta$.} Therefore, 
\begin{equation}
\begin{split}
I&=\int_{\G}\zeta^2G \, \langle D\Phi(\nabla u), \nu \rangle \langle D^2\Phi(\nabla u)D^2u\nu, \nu\rangle=\int_{\G}\frac{\langle D\Phi(\nabla u), \nu \rangle \langle D^2\Phi(\nabla u)D^2u\nu, \nu\rangle}{1+\Phi(\nabla u)}(G\zeta)^2\\
& \leq \int_U D^2 \Phi(\na u )[ \na (G\zeta), \na (G\zeta)] + 2 \mu^2 G^2\zeta^2= \int_U (G^2\,D^2\Phi(\nabla u)[\nabla\zeta,\nabla\zeta]+2\mu^2G^2\zeta^2) + II.
\end{split}
\end{equation}
Combining this with \eqref{eqn: int eqn G1} and  applying the Lipschitz bound for $u$ (which implies $G\leq C$), we find
\begin{align*}
 \int_{U}\zeta^2G\divv(D^2\Phi(\nabla u)\nabla G) & \leq C\int_U( |\na \zeta|^2 + \zeta^2)\,.
\end{align*}
 At the same time, we can bound the left-hand side of  below using  \eqref{eqn: G equation}, which, after rearranging terms, yields
\begin{align}\label{eqn: Hessian bound intermediate}
{\La^{-2}}\int_{U} G \zeta^2|D^2u|^2 \leq 
{C}
 \int_{U}(|\nabla\zeta|^2 +\zeta^2 )\,dx  + 2\mu\int_{U}\zeta^2\,G
 D \Phi(\nabla u)\cdot\nabla f\,dx
\end{align}
 Here we have once again used the Lipschitz bound for $u$.
 It remains bound the second term on the right-hand side of \eqref{eqn: Hessian bound intermediate} from above.\footnote{We could do this immediately to get a constant in \eqref{eqn: curvature estimates} depending on $\|\na f\|_{L^\infty}$, but as shown here, this dependence is not necessary.}. We
 integrate by parts 
 to find
\begin{equation}\label{eqn: bound f term}
    \begin{split}
     \int_{U}\zeta^2\,G D \Phi(\nabla u)\cdot\nabla f
 =&-  \int_{U} \zeta^2Gf \,\divv(D \Phi(\nabla u))
 - \int_{U} ( 2\zeta f\, G D \Phi(\nabla u)\cdot\nabla\zeta + \zeta^2 f \na G \cdot D\Phi(\na u))\\
 &
 +\int_{\G}\zeta^2 G f\,D\Phi(\nabla u)\cdot\nu\, d\Hm^{n-1}\\
     \end{split}
 \end{equation}
 The integrand of the boundary term is nonpositive a.e., since $f$ is assumed to be nonnegative and for all $x \in \Reg(u)$ (in particular for $\Hm^{n-1}$-a.e. $x \in \G$), we have $D \Phi(\na u)\cdot \nu =- \frac{1}{|\partial_\nu u|} D \Phi(\na u) \cdot \na u \leq 0$, as $D\Phi(p)\cdot p \geq 0 $ for all $p$. Moreover, thanks to the equation for $u$ and the nonnegativity of $u$, $G$, and $f$, we see that 
$f G \divv(D \Phi(\nabla u))= 2\mu^2 u f G - 2\mu G f^2 \geq- 2\mu G  f^2$.  Therefore, also using the Lipschitz bound to esitmate the second term on the right-hand side of \eqref{eqn: bound f term}, we find
 \begin{equation}\label{eqn: bound f term2}
    \begin{split}  
 \int_{U}\zeta^2\,G D \Phi(\nabla u)\cdot\nabla f  &\leq 2\mu\int_{U}\zeta^2f^2
+ C \int_U (\zeta f G |\na \zeta| + \zeta^2 f |D^2u|) \\
&\leq C\int (|\na \zeta|^2 +\zeta^2) + \frac{C}{\delta} \int_U \zeta^2  + C\delta \int_U  |D^2 u|^2 \zeta^2 .    \end{split}
\end{equation}
Here we choose $\delta>0$ small enough so that $C\delta < \Lambda^{-2}/2$. Therefore, applying the bound  \eqref{eqn: bound f term2} to the right hand side of \eqref{eqn: Hessian bound intermediate} and keeping in mind that $G\geq 1$, we can absorb this term into the left-hand side of   \eqref{eqn: Hessian bound intermediate} to obtain
\begin{align}\label{eqn: almost done}
\int_{U}\zeta^2|D^2u|^2 \leq 
{C}
 \int_{U}(|\nabla\zeta|^2+\zeta^2 )\,dx  \qquad \forall \zeta\in C_c^1(B_{3/2} \setminus \Sing(u))\,.
\end{align}

{\it Step 2:} In this step we shrink the cutoff function $\zeta$ to extend \eqref{eqn: almost done} to general $\varphi \in C^1_c(B_{5/4})$, using a capacity argument going back to Sternberg and Zumbrun, compare \cite[Lemma 2.4]{SZ}. 

Thanks to Lemma~\ref{lem: perturbed functional small singular set} and the assumption $\Phi \in \Istarz,$ we know $\Hm^{n-2}(\Sing(u)\cap B_2)=0.$  Let $\Sigma = \Sing(u)\cap \overline{B}_{5/4}$.  Since $\Hm^{n-2}(\Sigma)=0$, for any $\e>0$ we may find a countable collection of balls $\{B_{r_k}(z_k)\}$ covering $\Sigma$ with $\sum_k r_k^{n-2}<\e$.  By Theorem~\ref{thm: eps reg}, $\Sigma$ is a compact set, so we can choose a finite subcover, again denoted $\{B_{r_k}(z_k)\}$ after re-indexing.
For each $k$, let $\psi_k \in C^\infty(\R^n)$ with $\psi_k = 0$ in $B_{r_k}(x_k)$, $\psi_k = 1$ on $\R^n \setminus B_{2r_k}(x_k)$, and $|\na \psi_k|\leq 2/r_k$. Let $\psi =\min_k\psi_k$, so that $\psi=0$ on $U' =\cup_{k}B_{r_k}(z_k) \supset \Sigma$ and $\psi=1$ on $\R^n \setminus U$ where $U=\cup_{k}B_{2r_k}(z_k)$. Note that $U$ contains an $\delta$-tubular neighborhood of $\Sigma$ for some $\delta>0$. 
Now, for any $\varphi \in C^{\infty}_c(B_2)$ with support contained in $\overline{B}_{5/4}$, we let $\zeta =  \psi\varphi$, which satisfies $\zeta \in C^1_c(B_2\setminus \Sing(u))$, and thus \eqref{eqn: almost done} holds for this choice of $\zeta.$

By the chain rule and the fact that $0 \leq \psi \leq 1$ by construction, we have 
\begin{equation}
    \label{eqn: capacity thing}
|\na \zeta(x)|^2 \leq 2|\na \varphi(x)|^2+ 2|\varphi|^2 \sum_{k} r_k^{-2} \chi_{B_{r_k}(z_k)} 
\end{equation}
and thus using this to bound the the right-hand side of \eqref{eqn: almost done} with $\zeta$ from above yields 
$$
\int_{U}\zeta^2|D^2u|^2 \leq  {C} \int_{U}|\na \varphi|^2\,dx + \| \varphi\|_{L^\infty}^2 ( \sum_k r_k^{n-2} +1 + \| f\|_{L^\infty})
$$
Here we again use $\|\na u\|_{L^\infty(B_{3/2})}\leq L$.
Take a sequence $\e_k \to 0$ and the corresponding $\psi = \psi_{\e_k}$ and $\zeta_k = \psi_{\e_k}\varphi.$
By Fatou's lemma,  $\int_U \varphi^2 |D^2 u|^2 \leq C \| \varphi\|_{C^1}^2.$
Choosing a cutoff $\varphi$ with $\varphi\equiv 1$ in $B_{1}$ and $|\na \varphi|\leq C(n)$, we conclude. 
\end{proof}

\begin{remark}
Observe that the proof of Proposition~\ref{prop: curvature estimates} only used that $u$ is a {\it stable} - not minimizing - solution to the free boundary problem \eqref{eqn: EL} whose free boundary was smooth outside a codimension-two null set.
    Curvature estimates like \eqref{eqn: curvature estimates} for stable {classical} solutions (or stable solutions for which $\Sing(u)$ is known to be small) have various applications, see e.g. \cite{CFFS}. So, let us note that the proof also works for classical stable solutions, or stable solutions with $\Hm^{n-2}(\Sing(u)) = 0$, when $\Phi \in \II$ is regular outside the origin. This in particular includes the important $2$-homogeneous case. To deal with the singularity of $\Phi$ at the origin, the only modification to the proof of Proposition~\ref{prop: curvature estimates} is to test the equation for $G$ with a function $\zeta^2 \,  \eta \circ G$ where $\eta$ is a cutoff function on the level sets of $G$; this contributes an extra term in the integration-by-parts that has a favorable sign. This cutoff is defined and the proof is carried out in a degenerate elliptic setting in Proposition~\ref{prop: curvature est p} in Section~\ref{sec: p lap} below, and works with the obvious modifications  when $\Phi \in\II $ is $C^{2,\alpha}$ away from the origin. 
\end{remark}

\subsection{$L^{2+\eta}$ Hessian Estimates }
Next, we establish an improved $L^q$ bound for the Hessian; compare \eqref{eqn: curvature estimate eta intro}.  We give statement and proof for minimizers of $\J$, not $\Jfm$, for clarity of presentation and since this is all that is needed in our application. The same proof yields analogous estimates for minimizers of $\Jfm$ with constants depending on $\|f\|_{C^1}$. 
 Recall the universal constant $\db >0$ was defined in Lemma~\ref{lem: ball with big gradient} and that $r_\La = \db/2$ from Definition~\ref{def: annulus}. The main estimate of this section is the following.
\begin{proposition}
    \label{prop: curvature estimate 2}
Let $\Phi \in \II$ and suppose $u$ is a minimizer of $\J(\cdot)$ in $B_2$ with $0\in\partial\{u>0\}$ satisfying the $L^2$ Hessian estimate \eqref{eqn: curvature estimates} for all $p\in \GU$ and $r \in (0,1]$ with $B_{2r}(p)\subset B_2$. There are universal constants $\eta >0$  and $C>0$ such that
\[
\int_{B_{3/4}\cap\{|\nabla u|\geq\db \}}|D^2u|^{2+\eta}\leq C.
\]
\end{proposition}

In the course of proving Proposition~\ref{prop: curvature estimate 2}, we will make use of the following lemmas.
\begin{lemma}\label{lem:rigiditystrict}
    Let $U\subset\mathbb R^n$ be connected and open, and let $u\in C^1(U)\cap W^{2,1}_{\mathrm{loc}}(U)$.
Fix $\delta\geq0$ and suppose that $D^2u=0$
a.e. in $E_\delta:=U\cap\{|\nabla u|>\delta\}$.
If $E_\delta\neq\emptyset$, then $u(x)=e\cdot x+c$ in $U$
for some constant vector $e$ with $|e|>\delta$.
\end{lemma}
\begin{proof}
    Fix a connected component $V$ of $E_\delta$. Since the weak derivatives of $\nabla u$ vanish on $V$, by connectedness and continuity we have $\na u\equiv e$ on $V$, for a constant vector $e\in\mathbb R^n$ with $|e|>\delta$.

    We claim that $V$ is relatively closed in $U$, so that since it is also relatively open in $U$, $V\neq \emptyset$, and $U$ is connected, we deduce $V=U$.

    If $x\in\overline V\cap U$, continuity gives $\na u(x)=e$. Since $|e|>\delta$, there is a ball $B_r(x)$ contained in $E_\delta$. This ball intersects $V$, so connectedness of $B_r(x)$ and $V$ being connected component implies $B_r(x)\subset V$. Hence $x\in V$, as desired.
\end{proof}
{The proof only needs $W^{2,1}_{\mathrm{loc}}$ regularity on $U\cap\{|\nabla u|>\delta\}$ and continuity of $\nabla u$ on $U$. This localized version will be used in Section~\ref{sec: p lap} when establishing the $p$-Laplacian analogues of Lemma~\ref{lem: Hessian eps reg}, in the verification of Ingredient~\eqref{p2}, and Lemma~\ref{lem: doubling}, in the verification of Ingredient~\eqref{p3}.
}

\begin{lemma}\label{lem: local half-plane rigidity}
Let $u$ minimize $\E^\Phi$ in $B_1$, with $0\in\GU$. Let $0\leq\delta\leq\db$, and suppose that $D^2u=0$ a.e. in $B_1\cap\{u>0,|\nabla u|>\delta\}$.
There exists $\rho_0=\rho_0(n,\La)\in(0,1]$ and $e \in \R^n$ such that
\[
 u(x)=(e\cdot x)_+\quad\text{in }B_{\rho_0},
 \qquad\cG(e)=1.
\]
\end{lemma}

\begin{proof}
By Lemma~\ref{lem: ball with big gradient}, every connected component $V$ of $B_1\cap\{u>0\}$, intersecting $B_{1/8}$ nontrivially, contains a ball $B_{\db}(p_V){\subset B_{1/2}}$ on which $|\nabla u|\geq8\db>\delta$. Lemma~\ref{lem:rigiditystrict} thus guarantees that $u$ coincides with an nonzero affine function on each connected component of $\{u>0\}\cap B_1$. Since the balls $B_{\db}(p_V)$ are disjoint and contained in $B_{1/2}$, there are only finitely many such components. By Remark~\ref{rmk: connected components}, we conclude the proof.
\end{proof}

The main tool to prove Proposition~\ref{prop: curvature estimate 2} is the following doubling estimate.
\begin{lemma}\label{lem: doubling}
Let $\Phi \in \II$ and suppose  $u$ is a minimizer of $\J(\cdot)$ in $B_4$ satisfying the $L^2$ Hessian estimate \eqref{eqn: curvature estimates}  for all $p\in \GU$ and $r>0$ such that $B_{2r}(p)\subset B_4$.
Suppose additionally that $|\nabla u|\leq L$ in $B_4$ for some $L\in(0, R_\La)$. There is a constant $
A\gg1$ depending on $L$ and the universal parameters
such that, if $
u(0)>0,$ $ 
|\nabla u(0)|\geq \db$, and $|D^2u|(0)\geq A,$
then
\begin{equation}
    \label{eqn: doubling esimate}
\int_{B_1\cap\{u>0\}}|D^2u|^2
\leq
A\int_{B_{1/5}\cap\mathcal{B}_A(u)}
|D^2u|^2,
\end{equation}
where $\mathcal{B}_A(u)= \{u>0\}\cap\{|\nabla u|\geq\db\}
\cap\{\frac1A\leq|D^2u|\leq A\}.$
Here  $\db$ is the universal constant from Lemma~\ref{lem: ball with big gradient} and $R_\La$ is defined in Definition~\ref{def: annulus}.
\end{lemma}

\begin{proof}
 {\it Step 0: Initial proximity to the free boundary.} Let $\rho_0=\rho_0(n,\La)\in (0,1]$ be as in Lemma~\ref{lem: local half-plane rigidity}. To begin, we claim that for $A$ sufficiently large, the assumptions of Lemma~\ref{lem: doubling} guarantee that  
 \begin{equation}\label{eqn: see FB}
\GU \cap B_{\rho_0/100}\neq\varnothing.
 \end{equation}
Suppose not; since $u(0)>0$ this means $B_{\rho_0/100}\subset \{u>0\}$. In this case, applying the interior estimates \eqref{eqn: gradient reg interior} of Lemma~\ref{lem: C1beta} to $u-u(0)$ in $B_{\rho_0/100}$ shows that $
\|\nabla u\|_{C^{0,\beta}(B_{\rho_0/200})} \leq CL$ for a universal $C>0$. {Since $|\nabla u(0)|\geq\db$, if we set $s=\min\{(\frac{\db}{4CL})^{1/\beta},\frac{\rho_0}{200}\}$, the H\"older estimate gives
$3\db/4\leq|\nabla u|\leq L$ on $B_s$.
Since $3\db/4>r_\La$ and $L<R_\La$, this is a fixed compact subannulus of $\An$, and the interior Hessian estimate applies.}
 In turn, applying the Hessian bound \eqref{eqn: Hessian bound interior} of  Lemma~\ref{lem: C1beta} in this smaller ball yields 
\[
|D^2u(0)|\leq \frac{C}{s^2}
\]
for $C$ depending on $L$ and the universal parameters. This is  a contradiction for $
A>\frac{C}{s^2}.$ Thus \eqref{eqn: see FB} holds.

\emph{Step 1: Set up and compactness.}
Suppose the conclusion of the lemma fails. We find a sequence of integrands $\{\Phi_k\} \subset \II$ and minimizers $u_k$ of $\E^{\Phi_k}$ in $B_4$ with $u_k(0)>0, |\na u_k(0)|\geq \db$, $|D^2 u_k(0)|\geq k$, and $|\na u_k|\leq L$ 
such that
\begin{equation}
    \label{eqn: contra 1}
\int_{B_1\cap\{u_k>0\}}|D^2u_k|^2
\geq
k\int_{B_{1/5}\cap\mathcal{B}_k(u_k)}
|D^2u_k|^2.
\end{equation}
For $k$ large, Step 0 guarantees that \eqref{eqn: see FB} holds with $u_k$ in place of $u$. In particular, there is a point in $B_{\rho_0/100}$ for which \eqref{eqn: curvature estimates} holds with $r=1+1/100.$ In particular, by set containment, \eqref{eqn: curvature estimates} implies $\int_{B_1\cap\{u_k>0\}} |D^2u_k|^2 \leq C.$
Pairing this with \eqref{eqn: contra 1} and the trivial bound $\int_{B_{1/5}\cap \{|D^2u_k|\leq 1/k\}} |D^2u_k|^2 \leq C/k^2$  gives
\begin{equation}\label{eqn: hessian bound on B_k}
\int_{B_{1/5}\cap
 \mathcal{D}_k}
|D^2u_k|^2
\leq \frac{C}{k}, 
\qquad \text{where} \qquad \mathcal{D}_k=\{u_k>0\}\cap\{|\nabla u_k|\geq\db\}\cap\{|D^2u_k|\leq k\}
\end{equation}
for all $k$ large and for a universal constant $C>0$.

Up to a subsequence, $\Phi_k \to \Phi \in \II$. Moreover, 
since \eqref{eqn: see FB} holds for each $u_k$, Lemma~\ref{lem:compactness} and Lemma~\ref{lem:compactness2} guarantee that,  up to a further subsequence, $u_k$ converges locally uniformly and locally in $H^1$ in $B_4$ to a function $u$ that minimizes  $\E^\Phi(\cdot)$ on compact subsets of $B_4$. 
Moreover, let $\bar x_k$ be a closest point of $\G(u_k)$ to the origin. We claim that $|\bar x_k|\to0$. Otherwise, after passing to a subsequence, there would be a fixed $\sigma>0$ such that $B_\sigma\subset\{u_k>0\}$ for every $k$. Repeating the interior argument of Step 0 on $B_\sigma$, using $|\nabla u_k(0)|\geq\db$ and $|\nabla u_k|\leq L$, would give a bound for $|D^2u_k(0)|$ independent of $k$, contradicting $|D^2u_k(0)|\geq k$.
Thus $\bar x_k\to0$, and convergence of the free boundaries gives $0\in\G(u)$. 
\medskip

\noindent\emph{Step 2: Identifying $u$.} Let $E=B_{1/5}\cap\{u>0\}\cap\{|\nabla u|>\db\}$. 
We claim $D^2 u = 0$ a.e. on $E$. Indeed, take open sets $D \Subset D'\Subset  E$. The local uniform convergence of $u_k$ and $\na u_k$ (the latter thanks to estimate \eqref{eqn: gradient reg interior} of Lemma~\ref{lem: C1beta}) guarantee that $
 D'\subset\{u_k>0\}\cap\{|\nabla u_k|>\db\}$ for all sufficiently large $k$.
Hence,  $\na u_k(x) \in \An$ for all such $k$ and for all $x \in \overline{D}'$. Therefore, $|D^2 u_k|\leq \bar{C}$ on $D$ by estimate \eqref{eqn: Hessian bound interior} of Lemma~\ref{lem: C1beta}. Here $\bar{C}$ depends on $D, D'$ as well as on $L$ and the universal parameters. This means, by definition, $D\subset \mathcal{D}_k$ for all $k$ large, with $\mathcal{D}_k$ as defined in \eqref{eqn: hessian bound on B_k}. So,  \eqref{eqn: hessian bound on B_k} implies that $D^2u_k\to0$ in $L^2(D)$. Distributional convergence of the second derivatives then gives
$D^2u=0$ in $D$. Exhausting $E$, we obtain $D^2u=0$ a.e. in $E$.

{Applying Lemma~\ref{lem: local half-plane rigidity} to $v(y)=10u(y/10)$ in $B_1$, we obtain, up to a rotation, that $u(x)=a(x\cdot e_1)_+$ in $B_{\rho_0/10}$, where $a>0$ satisfies $D\Phi(ae_1)\cdot(ae_1)-\Phi(ae_1)=1$.}
As observed in Remark~\ref{rmk: annulus}, thus guarantees that $ae_1$ lies in a compact subset of $\An.$

\medskip
\noindent\emph{Step 3: $C^{2,\alpha}$ estimates on $\{u_k>0\}$ and conclusion.} {By Step 2 and $\bar x=0$, $u(x)=a(x\cdot e_1)_+$ in $B_{\rho_0/10}$, with $\cG(ae_1)=1$. For each $k$, let $a_k>0$ be the unique number satisfying
$$D\Phi_k(a_ke_1)\cdot(a_ke_1)-\Phi_k(a_ke_1)=1.$$
The convergence of the integrands imply $a_k\to a$. Since $\bar x_k\to0$, then $B_{\rho_0/20}(\bar x_k)\subset B_{\rho_0/10}$ for $k$ large. Local uniform convergence and nondegeneracy imply that, for all such $k$,
$$(a_k(x-\bar x_k)\cdot e_1- \bar{\e}\rho_0/20)_+ \leq u_k(x)\leq (a_k(x-\bar x_k)\cdot e_1+ \bar{\e}\rho_0/20)_+ \qquad \text{in }B_{\rho_0/20}(\bar x_k),$$
where $\varepsilon_0$ is the flatness threshold in Theorem~\ref{thm: eps reg}.
Theorem~\ref{thm: eps reg} and the quantitative observation in the proof of Corollary~\ref{cor: higher reg} give $C^{2,\alpha}$ boundary graphs with uniformly bounded rescaled norms on a fixed smaller neighborhood of $\bar x_k$. Consequently, the constants in Lemma~\ref{lem: C1beta} are independent of $k$.
The boundary gradient estimate and the Bernoulli condition also place $\nabla u_k$ in a fixed compact subannulus of $\An$ on a further fixed smaller neighborhood. Applying \eqref{eqn: Hessian bound boundary} with $w=u_k$ and $\Omega=\{u_k>0\}$, we therefore obtain
$$
\|D^2u_k\|_{L^\infty(B_{c\rho_0/20}(\bar x_k)\cap\{u_k>0\})}\leq \frac{C}{\rho_0/20},
$$
where $c,C>0$ are independent of $k$. Since $\bar x_k\to0$ and $u_k(0)>0$, this estimate applies at the origin for all sufficiently large $k$, contradicting $|D^2u_k(0)|\geq k$.}
\end{proof}

We now see how Lemma~\ref{lem: doubling} is applied to get improved integrability in the curvature estimate.

\begin{proof}[Proof of Proposition~\ref{prop: curvature estimate 2}]
We proceed in several steps. \\

\emph{Step 0: Set up.}
Since $u(0)=0$, we may apply Lemma~\ref{lem:interior-lip} in $B_2$ to find $
|\nabla u|\leq L=L(\Lambda,n)$ in $B_{15/8}.$ Keeping in mind Definition~\ref{def: annulus},  $L<R_\La$. 
Let $A$ be the constant from Lemma~\ref{lem: doubling} corresponding to this choice of $L$.
Let $M \geq 4A$ be a constant depending on the universal parameters to be specified below. For each
$k\in\mathbb N$, set
\begin{equation}
    \label{eqn: D def}
D_k:=B_{\frac34+(\frac14)^k}
\cap\{u>0\}
\cap\{|\nabla u|>\db\}
\cap\{|D^2u|>M^k\}.
\end{equation}

\medskip
\noindent\emph{Step 1: Applying the doubling estimate to the rescaled problem.}
For each $x_0\in D_k$, take $
r=\frac{A}{M^k} $
and $
u_r(x):=\frac1r u(x_0+rx).$
Since $x_0\in B_1$ and $M\geq 4A$, we see that $u_r$ is a minimizer of
$\E^\Phi(\cdot)$ in $B_4$ satisfying
\[
u_r(0)=\frac1r u(x_0)>0,
\qquad
|\nabla u_r|(0)=|\nabla u(x_0)|>\delta,\qquad
|D^2u_r(0)|=r|D^2u(x_0)|
>\frac{A}{M^k}M^k=A.
\]
Moreover, $
|\nabla u_r|=|\nabla u(x_0+r\,\cdot)|\leq L$ in $B_4$.
Therefore, we can apply the doubling estimate of Lemma~\ref{lem: doubling} to conclude
\[
\int_{B_1\cap\{u_r>0\}}|D^2u_r|^2
\leq
A\int_{B_{1/5}\cap\{u_r>0\}\cap\{|\nabla u_r|\geq\delta\}
\cap\{\frac1A\leq|D^2u_r|\leq A\}}
|D^2u_r|^2.
\]
In terms of the original minimizer $u$, this gives
\[
\int_{B_r(x_0)\cap\{u>0\}}|D^2u|^2
\leq
A\int_{B_{r/5}(x_0)\cap\{u>0\}\cap\{|\nabla u|\geq\delta\}
\cap\{\frac{M^k}{A^2}\leq|D^2u|\leq M^k\}}
|D^2u|^2.
\]
As a result, if we pick $M=A^2$ and let $\mathcal{K}_k = \{u>0\}\cap\{|\nabla u|\geq\delta\}
\cap\{M^{k-1}\leq|D^2u|\leq M^k\}$, this becomes  
\begin{equation}\label{eqn: doubling est 2}
\int_{B_r(x_0)\cap\{u>0\}}|D^2u|^2
\leq
A\int_{B_{r/5}(x_0)\cap \mathcal{K}_k}
|D^2u|^2.
\end{equation}

\medskip
\noindent\emph{Step 2: Inductive step.}
For each $x\in D_k$, we choose $r$ as in the previous step. Thus
$\{B_r(x)\}_{x \in D_k}$ forms a covering of $D_k$. We take a Vitali subcovering $\{B_r(x_j)\}$. Since $x\in B_{\frac34+(\frac14)^k}$ and  $r=\frac{A}{M^k}=\frac{A}{A^{2k}},$
after possibly increasing $A$, we have
\begin{equation}\label{eqn: containment}
B_r(x)\subset B_{\frac34+(\frac14)^{k-1}}
\qquad\text{for every }x\in D_k.
\end{equation}
Now, apply \eqref{eqn: doubling est 2} in each ball $\{B_r(x_j)\}$ in the covering. Using \eqref{eqn: containment} followed by \eqref{eqn: doubling est 2} and the fact that the balls $B_{r/5}(x_k)$ are pairwise disjoint, we find
\begin{align}\label{eqn: eta intermed}
\int_{D_k}|D^2u|^2
\leq \sum_j\int_{B_r(x_j)\cap\{u>0\}}|D^2u|^2
&\leq A\sum_j
\int_{B_{r/5}(x_j)\cap \mathcal{K}_k}
|D^2u|^2
\leq A
\int_{B_{\frac34+(\frac14)^{k-1}}\cap
\mathcal{K}_k}
|D^2u|^2   \,. 
\end{align}

The domain of integration in the last integral satisfies the containment
\[
B_{\frac34+(\frac14)^{k-1}}\cap
\mathcal{K}_k \subset D_{k-1}\setminus
\big(
B_{\frac34+(\frac14)^{k-1}}
\cap\{u>0\}
\cap\{|\nabla u|\geq\delta\}
\cap\{|D^2u|>M^k\}
\big)
\subset D_{k-1}\setminus D_k.
\]
Thus, we conclude from \eqref{eqn: eta intermed} that $
\int_{D_k}|D^2u|^2
\leq A\int_{D_{k-1}}|D^2u|^2
-A\int_{D_k}|D^2u|^2,$
or, in other words,
\begin{equation}\label{eqn: decay}
\int_{D_k}|D^2u|^2
\leq \frac{A}{1+A}\int_{D_{k-1}}|D^2u|^2.
\end{equation}

\medskip
\noindent\emph{Step 3: Conclusion.}
An iteration of the previous step gives
\[
\int_{D_k}|D^2u|^2
\leq
\left(\frac{A}{1+A}\right)^{k-1}
\int_{D_1}|D^2u|^2
\leq
\left(\frac{A}{1+A}\right)^{k-1}
\int_{B_{1}\cap\{u>0\}}|D^2u|^2.
\]
Since $0 \in \GU$, we may applying the curvature estimate of Proposition~\ref{prop: curvature estimates} to the right-hand side to find
\begin{equation}
    \label{eqn: integral bound}
\int_{D_k}|D^2u|^2
\leq C
\left(\frac{A}{1+A}\right)^{k-1}.
\end{equation}
Recalling the definition of $D_k$ in \eqref{eqn: D def} and applying Chebychev inequality to \eqref{eqn: integral bound}, we find
\[
\left|
\{|D^2u|\geq M^k\}
\cap\{|\nabla u|\geq\delta\}
\cap\{u>0\}
\cap B_{3/4}
\right|\leq |D_k|
\leq C\left(\frac{A}{1+A}\right)^{k-1}M^{-2k}
\]
for all $k$.
The conclusion follows by the layer-cake formula.
\end{proof}

\section{Proofs of the main theorems}\label{sec: main theorem proof 0}

In this section, we prove Theorems~\ref{thm: main} and \ref{thm: curvature estimate intro}. As we described in the introduction, to prove Theorem~\ref{thm: main}, we first show that $\Istarz = \II$ using a topological argument to show that $\Istarz$ is both open and closed in the connected nonempty set $\II$. Then, we apply the improved $L^{2+\eta}$ curvature estimates to complete the proof.  Theorem~\ref{thm: curvature estimate intro}
will then follow from Theorem~\ref{thm: main}, the corresponding estimates established in the previous section, and an approximation argument.

\subsection{Preparatory lemmas}
We need two additional ingredients before proving Theorem~\ref{thm: main}. The first is an $\e$-regularity statement with flatness measured in terms of the norm of the Hessian.  
It will be applied twice: once with $\eta=\delta=0$ and then with $\eta>0$ and $\delta = \db$. Recall $\db>0$ is the constant determined in  Lemma~\ref{lem: ball with big gradient}.
\begin{lemma}\label{lem: Hessian eps reg} 
For any $\eta \in [0, n-2)$ and $\delta  \in [0,\db]$, there exists $\e>0$ depending on  $\delta$,  $\eta$, and the universal parameters such that if $\Phi \in \II$ and  $u$ is a minimizer of $\J$ in $B_1$ with $ 0\in \GU$ and 
\[
\int_{B_1\cap\{u>0\}\cap\{|\nabla u|\geq\delta\}}|D^2u|^{2+\eta}
<\varepsilon,
\] 
then $0 \in \Reg (u)$.
\end{lemma}
\begin{proof}  Suppose for a contradiction that we have $\eta \in [0,n-2)$, $\delta \in [0,\db]$, and a sequence of integrands $\Phi_k \in \II$ and minimizers $u_k$ of $\E^{\Phi_k}(\cdot)$ in $B_1$ with $0 \in \G_{u_k}$ such that $\int_{B_1\cap\{u_k>0\}\cap\{|\nabla u_k|\geq\delta\}}|D^2u_k|^{2+\eta}\to 0$ but $0 \in \Sing(u_k)$ for each $k$. By compactness for integrands and Lemmas~\ref{lem:compactness} and \ref{lem:compactness2}, up to passing to a subsequence, $\Phi_k \to \Phi \in \II$  and $u_k \to u$ such that $0 \in \G(u)$ and $u$ is a minimizer of $\J$ on compact subsets of  $B_1$.

In the case $\delta=0$, Lemma~\ref{lem:compactness2} guarantees that 
\[
\int_{\{ u>0\} \cap B_1} |D^2 u|^{2+\eta}\,dx \leq
\liminf_k  \int_{\{ u_k>0\} \cap B_1} |D^2 u_k|^{2+\eta}=0.
\]
So, by Lemma~\ref{lem: local half-plane rigidity}, in $\{u>0\} \cap B_{\rho_0}$,
$u = (x \cdot e)_+$ for some $e \in \R^n$ satisfying the free boundary condition $D\Phi(e)\cdot e -\Phi(e)=1.$ In particular, $\GU$ is a hyperplane and $0 \in \Reg(u)$. This contradicts Corollary~\ref{cor: prop of singularities}, completing the proof in this case.

For $\delta>0$, we argue as in Step 2 of the proof of Lemma~\ref{lem: doubling} above. Set $E_\delta=B_1\cap\{u>0\}\cap\{|\nabla u|>\delta/2\}$ and take an open set $D\subset\subset  E_\delta$. By local uniform convergence of $u_k$ and $\na u_k$, for all
sufficiently large $k$ we have $
 D\subset\{u_k>0\}\cap\{|\nabla u_k|\geq\delta\}$.
 Hence $D^2u_k\to0$ in $L^p(D)$.
Distributional convergence of the second derivatives gives
$D^2u=0$ in $D$.
Exhausting $E_\delta$, we obtain $D^2u=0$ a.e. in $E_\delta$.

So, once more, Lemma~\ref{lem: local half-plane rigidity} guarantees that in $\{u>0\} \cap B_{\rho_0}$, we have  $u(x)=(e\cdot x)_+$ for some $e\neq0$ with $D\Phi(e)\cdot e -\Phi(e)=1$.
In particular, $0\in\Reg(u)$, again contradicting
propagation of singularities of Corollary~\ref{cor: prop of singularities} from the sequence $u_k$.
\end{proof}
The second preparatory ingredient, which is a consequence of Lemma~\ref{lem: Hessian eps reg}, shows how curvature estimates like Proposition~\ref{prop: curvature estimates} and Proposition~\ref{prop: curvature estimate 2} give a measure bound on the singular set. Once again, this statement will be applied twice: first with $\eta=\delta =0$, and then with $\delta,\eta>0.$
\begin{corollary}
    \label{lem: hessian bound implies small sing set} Let $\eta \in [0, n-2)$ and $\delta \in [0,\db].$
    Suppose $\Phi \in \II$ and let $u$ be a minimizer of $\J$ in $B_1$ with 
    \begin{equation}
        \label{eqn: Hessian est hp for lemma}
    \int_{B_1\cap\{u>0\}\cap\{|\nabla u|\geq\delta\}}|D^2u|^{2+\eta}<\infty\,.
        \end{equation}
 Then $\Hm^{n-2-\eta}(\Sing(u)\cap B_1)=0.$
\end{corollary}
\begin{proof} 
Let $x \in \Sing(u)\cap B_1$ and $r \in (0, \text{dist}(x,\partial B_1))$. Applying the contrapositive of Lemma~\ref{lem: Hessian eps reg} to the rescaled function $u_r(y) = u(x+ ry)/r$, which leaves the gradient lower bound invariant, we find that 
    \begin{equation}\label{eqn: Hessian big}
        \int_{ B_r(x) \cap\{u>0\} \cap \{|\na u|\geq \delta\}} |D^2u|^{2+\eta} \,dx \geq \e r^{n-2-\eta}
    \end{equation}
Here $\e>0$ is the constant from Lemma~\ref{lem: Hessian eps reg}. For $s \in (0,1)$, let $\Sigma_s = \Sing(u)\cap B_s$, and let $N_r(E) = \{ x: \text{dist}(x, E) <r\}$ 
    denote the radius-$r$ tubular neighborhood of a set $E$.
    We claim that for each $s \in (0,1)$,
    \begin{equation}
        \label{eqn: content bound}
    \lim_{r\to 0} \frac{|N_r(\Sigma_s)|}{r^{2+\eta}} = 0\,.
        \end{equation}
    To see this, fix $r <1-s$ and take a maximal $r$-dense subset $\{x_j\}_{j=1}^{J_r}$ of $\Sigma_s$. This way, the balls $B_{r}(x_j)$ are pairwise disjoint and their union is contained in $N_r(\Sigma_s)\subset B_1$, while 
    $\{B_{2r}(x_j)\}$ covers $\Sigma_s$ and $\{B_{3r}(x_j)\}$ covers $N_r(\Sigma_s)$. Applying \eqref{eqn: Hessian big} on each $B_r(x_j)$ and summing up, we find
    \begin{equation}\label{eqn: covering}
    \begin{split}
        |N_r(\Sigma_s)| \leq 
        3^n J_r \omega_n r^n 
       & \leq \frac{3^n \omega_n r^{2+\eta}}{\e} \sum_{i=1}^{J_r} \int_{\{u>0\}\cap \{|\na u|\geq \delta\} \cap B_r(x_j)} |D^2 u|^{2+\eta} \,dx\\
        &\leq \frac{3^n \omega_n r^{2+\eta}}{\e}  \int_{\{u>0\}\cap \{|\na u|\ge \delta\}\cap N_r(\Sigma_s)} |D^2 u|^{2+\eta} \,dx \,.
            \end{split}
    \end{equation}
  
 By \eqref{eqn: Hessian est hp for lemma} and the fact that $|\Sigma_s|=0$ (in fact $\Hm^{n-1}(\Sigma_s)<\infty$) by Corollary~\ref{cor: density estimates}, the integral on the right-hand side of \eqref{eqn: covering} tends to zero as $r\to 0$. Thus, dividing both sides of \eqref{eqn: covering} by $r^{2 +\eta}$ and sending $r \to 0$, we obtain \eqref{eqn: content bound}. By a standard argument, this in particular shows that $\Hm^{n-2-\eta}(\Sigma_s)=0$. Taking $s \to 1$ completes the proof.
\end{proof}

\subsection{Proofs of the main theorems}
We now prove Theorem~\ref{thm: main}. 
\begin{proof}[Proof of Theorem~\ref{thm: main}] We proceed in several steps.
Recall the space of good integrands $\Istarz$ defined in \eqref{def: good integrands}.

{\it Step 1: $\Istarz$ is open in $\II$.} To establish the openness of $\Istarz$ in $\II$, we show that the complement is closed. 
    Consider a sequence $\{\Phi_k\} \subset \II \setminus \Istarz$ with $\Phi_k\to \Phi \in \II$. We claim that $\Phi \not \in \Istarz.$ Since $\Phi_k \not \in  \Istarz$, for each $k$ we may find a minimizer $u_k$ in $B_1$ such that $0 \in \G_{u_k}$ and 
    \[
    \Hm^{n-2}(\Sing(u_k)\cap B_1) >0.
    \]
    Applying the same measure theory fact as in the proof of Lemma~\ref{lem: perturbed functional small singular set}, see \cite[Theorem 3.26(2)]{SimonLN}, for $\Hm^{n-2}$-a.e. $x \in \Sing(u_k)\cap B_1$ we have 
    $$\limsup_{r\to0}\frac{\Hm_{\infty}^{n-2}(\Sing(u_k)\cap B_1\cap B_r(x))}{\omega_{n-2}r^{n-2}} \geq \frac{1}{2^{n-2}}\,.
    $$
    In particular, for each $k$ we may find $x_k\in B_1$ and $r_k\in (0, \text{dist}(x_k,\partial B_1))$ such that the rescaled function $v_k(y)=u_k(x_k +r_ky)/r_k$, which again is a minimizer of $\cJ^{\Phi_k}$ in $B_1$, has 
    \[
    \Hm_\infty^{n-2}(\Sing(v_k)\cap \overline{B}_{1/2}) \geq 2^{-2n}.
    \]
    By Lemmas~\ref{lem:compactness} and \ref{lem:compactness2}, up to a subsequence, $v_k \to v$ (in the various senses of this lemmas) where $v$ is a minimizer of $\J$ on compact subsets of $B_1$. Then by Corollary~\ref{cor: prop of singularities}, we find $\Hm_\infty^{n-2}(\Sing(v)\cap \overline{B}_{1/2}) \geq 2^{-2n}$. Since $\Hm_\infty^{n-2}(E) \leq \Hm^{n-2}(E)$ for any set $E$, this means that
$\Hm^{n-2}(\Sing(v)\cap B_1)\geq 2^{-2n}>0$,
and therefore $\Phi\notin\Istarz$.

\medskip

{\it Step 2: $\Istarz$ is closed in $\II$.}
     Take a sequence $\{\Phi_k\} \subset \Istarz$ with $\Phi_k\to \Phi \in \II$. 
 As observed in Remark~\ref{rmk: connected}, smooth integrands $\Phi \in C^\infty(\R^n)$ are dense in $\II$. Therefore, thanks to the openness shown in Step 1 above, we may assume without loss of generality that $\Phi_k \in C^\infty(\R^n) \cap \Istarz.$ 
    
 Let $u$ be a minimizer of $\J(\cdot, \Omega)$. To prove that $\Istarz$ is closed, we must show that 
 $\Hm^{n-2}(\Sing(u)\cap \Omega) =0$ so that $\Phi \in \Istarz.$ Taking a countable exhaustion of $\Omega$ and using compactness, it suffices to show that for any $x \in \GU \cap \Omega$,  we have 
 $\Hm^{n-2}(\Sing(u) \cap B_{d/8}(x)) = 0$ where $d= \text{dist}(x, \partial \Omega)>0$. After rescaling, this claim becomes the following: 
 
 {\it Claim: } If $u \in H^1(B_4)$ minimizes $\J$ in $B_4$ and $0 \in \GU$, then 
 $\Hm^{n-2}(\Sing(u) \cap B_{1/2}) = 0$. 
 
 Applying Corollary~\ref{lem: hessian bound implies small sing set} with $\eta=\delta=0$, to prove the claim it suffices to prove that 
 \begin{equation}
     \label{eqn: hessian bound closedness proof}
 \int_{\{u>0\} \cap B_{1/2}}|D^2u|^{2} \,dx <\infty\,.
  \end{equation}

By Lemma~\ref{lem:interior-lip}, we have $\sup_{B_{2}} u + \sup_{B_{2}} |\na u| \leq L$ for $L=L(n,\Lambda)$. So, letting $f_k = u*\xi_{1/k}$ where $\xi$ is a standard mollifier and $\xi_\e = \frac{1}{\e^n} \xi(x/\e)$, we have the same bound  $\sup_{B_{2}} f_k + \sup_{B_{2}} |\na f_k| \leq L$ for each $f_k$. In particular,  up to a subsequence, $f_k \to u$ uniformly in $B_1$  and in $L^2(B_1)$. 

Now, for each $k\in \mathbb{N}$,  let $u_k$ be a minimizer of the energy functional $\mathcal{E}^{\Phi_k}_{f_k,1}$ on $B_1$ with boundary data $u_k=u$ on $\partial B_1$. Taking $u$ as a competitor for each $u_k$, we find that 
\begin{align*}
    \mathcal{E}^{\Phi_k}_{f_k,1}(u_k) \leq \mathcal{E}^{\Phi_k}_{f_k,1}(u) 
    &= \int_{B_1} \Phi_k(\na u) + |u-f_k|^2 \,dx \leq \int_{B_1} \Phi(\na u) + \int_{B_1} |\Phi_k(\nabla u) -\Phi(\nabla u)|dx +\int_{B_1} |u-f_k|^2 \,dx\\
       & \leq \int_{B_1} \Phi(\na u) + \|\Phi_k-\Phi\|_{C^0(B_L)} |B_1| +\int_{B_1} |u-f_k|^2 \,dx \to \int_{B_1} \Phi(\na u)
\end{align*}
In the final inequality, we use local $C^1$ convergence of $\Phi_k$ to $\Phi$ and the Lipschitz bound for $u$. 
In particular, $\sup_k  \mathcal{E}^{\Phi_k}_{f_k,1}(u_k)<\infty$. In particular, the sequence $\{u_k\}$ is uniformly bounded in $H^1(B_1)$, so up to a subsequence, $u_k \rightharpoonup v$ in $H^1(B_1)$ for some $v \in H^1(B_1)$. A slight modification of the proof of Lemma~\ref{lem:interior-lip} shows that $\sup_{B_\rho} (|\na u_k|+|u_k|)<\infty$ for each $\rho<1$ (here the functions are anchored by their boundary data rather than their zero set). Hence up to a diagonal subsequence as in the proof of Lemma~\ref{lem:compactness}, $u_k \to v$ locally uniformly in $B_1$. Since $f_k \to u$ uniformly, Lemma~\ref{lem:compactness2}(2) guarantees that $v$ is a minimizer of $\mathcal{E}^\Phi_{u,1}$ in $B_1$. Using the minimality of $v$ and $u$ for $\mathcal{E}^\Phi_{u,1}$ and $\J$ respectively, we find that
\begin{align*}
  \J(v) + \int_{B_1} |u-v|^2 =  \mathcal{E}^\Phi_{u,1}(v) \leq \mathcal{E}^\Phi_{u,1}(u) = \J(u)  \leq \J(v).
\end{align*}
The ordering of the left-most and right-most expressions imply that  $\int_{B_1} |v-u|^2 =0$ and thus $u=v$ in $B_1.$ In summary, $u_k\to u$ in $B_1$ locally uniformly and weakly in $H^1(B_1)$, and all conclusions of Lemma~\ref{lem:compactness2} hold.

Since $0 \in \GU$ and from the Hausdorff convergence of the free boundaries from Lemma~\ref{lem:compactness2}(4), we may find a sequence $\{x_k\}\subset B_1$ with $x_k \in \G({u_k})$ for each $k \in \mathbb{N}$ and  $x_k \to 0$. 
For any $\e>0$, we have $\text{dist}(x_k,\partial B_1) \geq 1-\e$ for $k$ large enough, so we may apply Proposition~\ref{prop: curvature estimates} with $x=x_k$ and $r=(1-\e)/2$. Thus $u_k$ satisfies the curvature estimate 
\[
   \int_{\{u_k>0\} \cap B_{1/2-2\e}} |D^2 u_k|^{2} \leq 
   \int_{\{u_k>0\} \cap B_{(1-\e)/2}(x_k)} |D^2 u_k|^{2}
   \leq C(1-\e)^{n-2}.
\]
In the final inequality, $C>0$ is the constant from  Proposition~\ref{prop: curvature estimates}, which depends on the universal parameters only thanks to the upper bound $\|f_k\|_{L^\infty} \leq L$ above. 
By Lemma~\ref{lem:compactness2}(5), the same bound holds for $u$. Letting $\e \to 0$, we obtain \eqref{eqn: hessian bound closedness proof}, completing the proof of the claim. Thus $\Istarz$ is closed.

\medskip

{\it Step 3: Conclusion.} The previous two steps show that $\Istarz$ is an open and closed subset of $\II.$ 
 Since $\II$ is connected, recall Remark~\ref{rmk: connected}, and nonempty by Theorem~\ref{thm: AC}, we have $\Istarz =\II.$
In other words, for every $\Phi\in\II$ and
for every minimizer $u$ of $\J(\cdot,\Omega)$,
\[
    \Hm^{n-2}(\Sing(u)\cap\Omega)=0.
\]
If $n=2$, this already gives $\Sing(u)\cap\Omega=\emptyset$, so we
may assume that $n\geq3$.

We first observe that the approximation argument in Step~2, Proposition~\ref{prop: curvature estimates},
and Lemma~\ref{lem:compactness2}(5) give
\[
    \int_{B_{1/4}\cap\{w>0\}}|D^2w|^2\,dx\leq C
\]
for every minimizer $w$ of $\J$ in $B_4$ with $0\in\G(w)$, where
$C$ is universal. Indeed, by Step~3 the smooth approximating integrands
can all be chosen in $\Istarz$, and the penalization data used in
Step~2 have a uniform $L^\infty$ bound. Applying this estimate to
$w(y)=u(p+4ry)/(4r)$, we obtain
\begin{equation}\label{eqn: conclusion all-scale L2}
    \int_{B_r(p)\cap\{u>0\}}|D^2u|^2\,dx
    \leq Cr^{n-2}
    \qquad\text{whenever }p\in\GU
    \text{ and }B_{16r}(p)\subset\Omega.
\end{equation}
Thus this estimate
holds for the original minimizer, without any additional smoothness
assumption on $\Phi$.

Fix $x_0\in\GU\cap\Omega$ and $R>0$ such that
$B_{8R}(x_0)\subset\Omega$, and set $v(y):={u(x_0+Ry)}/{R}$ for $y\in B_8.$ 
Then $v$ minimizes $\J$ in $B_8$ and $0\in\G(v)$.
If $q\in\G(v)$ and $B_{4\rho}(q)\subset B_2$, then
$B_{16\rho}(q)\subset B_8$. Consequently,
\eqref{eqn: conclusion all-scale L2}, applied to $v$, verifies the
$L^2$ hypothesis of Proposition~\ref{prop: curvature estimate 2} in
$B_2$. This proposition therefore gives a universal exponent
$\eta>0$ such that
\[
    \int_{B_{3/4}\cap\{v>0\}\cap\{|\na v|\geq\db\}}
        |D^2v|^{2+\eta}\,dy\leq C.
\]
Decreasing $\eta$ if necessary, using H\"older's inequality, we may
assume that $0<\eta<n-2$. Since
$\na v(y)=\na u(x_0+Ry)$, rescaling back gives
\begin{equation}\label{eqn: conclusion improved Hessian}
    \int_{B_{3R/4}(x_0)\cap\{u>0\}\cap\{|\na u|\geq\db\}}
        |D^2u|^{2+\eta}\,dx
    \leq CR^{n-2-\eta}.
\end{equation}
Applying Corollary~\ref{lem: hessian bound implies small sing set}
on $B_{3R/4}(x_0)$, after translation and rescaling, with
$\delta=\db$ and this choice of $\eta$, yields
$\Hm^{n-2-\eta}
        \bigl(\Sing(u)\cap B_{3R/4}(x_0)\bigr)=0$.
The exponent $\eta$ depends only on the universal parameters, and not
on $\Phi$, $u$, $x_0$, or $R$. Finally, a countable collection of such
balls covers $\GU\cap\Omega$. By countable subadditivity,
we deduce \eqref{eqn: intro measure est},
which completes the proof of Theorem~\ref{thm: main}.
\end{proof}

\begin{remark}\label{rmk: no holder}
The proof of Theorem~\ref{thm: main}  shows that 
      \begin{equation}\label{eqn: codim 2}
   \mathcal{H}^{n-2}(\Sing(u)) = 0 \qquad \text{ for any minimizer $u$ of $\J(\cdot)$}
   \end{equation}
   for $\Phi$ in a more general class of integrands $\tilde{I}_\Lambda.$ We say $\Phi \in \tilde{I}_\La$
whenever $\Phi$ satisfies properties \eqref{1}  and \eqref{2} in the definition of $\II$ in \eqref{def: good integrands}, i.e. $\Phi$ is nonnegative with $\Phi(0)=0$ and satisfies the ellipticity bounds \eqref{eqn: ellipt}, and additionally
   \[
   \Phi \in C^2(\overline{B_{4\sqrt{\Lambda}}}\setminus B_{1/4\sqrt{\Lambda}})\,.
   \]
Compared to $\II$, this class of integrands has no uniform modulus of continuity for the Hessian on an annulus.    
This class  in particular contains (a) integrands satisfying  \eqref{1} and \eqref{2} that are globally $C^2,$ and (b) nonnegative $2$-homogeneous integrands satisfying the ellipticity assumption \eqref{eqn: ellipt} whose restriction to the sphere is $C^2$.

To see the claimed statement, first note that for any such integrand, there is some 
modulus of continuity $\omega: \R_+\to \R_+$ such that 
  \begin{equation}
      \label{eqn: mod of cty}
       \left|D^2u(x)-D^2u(y)\right|\leq \omega(|x-y|) \qquad \text{for all } x,y \in \overline{B_{\Lambda}}\setminus B_{1/\Lambda}.
  \end{equation}
 Since the conclusion  \eqref{eqn: codim 2} does not depend on $\omega,$ it thus suffices to prove \eqref{eqn: codim 2} hold for the class of integrands $
\tilde{I}_\La^{\omega}$ in the class above that additionally satisfy \eqref{eqn: mod of cty}.
 Observe that $\tilde{I}_\La^\omega$ is compact, again with respect to the topology induced by local uniform convergence. Moreover, thanks to \eqref{eqn: first annular constraint}, for any $p\in\R^n$ satisfying the equation $D\Phi(p)\cdot p -\Phi(p)=1$ corresponding to the free boundary condition \eqref{eqn: EL}, $p$ lives in the annulus where $\Phi$ is $C^2$. These two properties allow the proof of the $\e$-regularity theorem, Theorem~\ref{thm: eps reg}, to go through. Since additionally $\tilde{I}_\La^\omega$ is connected and nonempty and smooth functions are dense in $\tilde{I}_\La^\omega$, the remainder of the proof of \eqref{eqn: codim 2} goes through verbatim.

 It may be possible to prove Theorem~\ref{thm: eps reg} for the class of integrands satisifying \eqref{1} and \eqref{2}, which would then imply that \eqref{eqn: codim 2} for every such integrand as well. We leave this as an interesting open problem.

Finally, we note that in the proof of Theorem~\ref{thm: main}, we use the quantitative H\"{o}lder modulus of continuity of $\Phi$ in $\An$ to prove the $\eta$-gain in the improved curvature estimates of Proposition~\ref{prop: curvature estimate 2}. Here what is needed is some higher regularity of the free boundary as in Corollary~\ref{cor: higher reg}, and an interior and boundary $C^0$ estimate for the Hessian, \eqref{eqn: Hessian bound interior} and \eqref{eqn: Hessian bound boundary}, where the latter assumes the same regularity of the boundary as what holds at flat points of the free boundary. Thus, any other assumptions on $\Phi$ giving rise to such uniform estimate would suffice. 
\end{remark}

Finally, the proof of Theorem~\ref{thm: curvature estimate intro} is essentially contained in the proof of Theorem~\ref{thm: main} above:
\begin{proof}[Proof of Theorem~\ref{thm: curvature estimate intro}]
The first estimate \eqref{eqn: curvature estimate intro} follows from \eqref{eqn: conclusion all-scale L2}, with $p=0$ and $r=1$. 
For the second estimate, we cover $\GU\cap\overline B_{5/4}$ by a bounded
number (depending only on $n$) of balls $B_{1/8}(x_j)$, with centers $x_j\in\GU\cap\overline B_{5/4}$. Since $B_{1/2}(x_j)\subset B_2$, by
\eqref{eqn: conclusion all-scale L2}, we can apply Proposition~\ref{prop: curvature estimate 2} to
$4u(x_j+y/4)$ in $B_2$. By rescaling, we get
$$
\int_{B_{3/16}(x_j)\cap\{u>0, |\nabla u|\geq\db\}}|D^2u|^{2+\eta}dx\leq C.
$$
We observe that $B_{3/16}(x_j)$ cover every point of $B_1$ whose distance from $\GU$ is less than $1/16$. At the remaining points of $B_1\cap\{u>0,\,|\nabla u|\geq\db\}$, then the interior
$C^{1,\beta}$ estimate provides a ball of universal radius
on which $|\nabla u|\geq3\db/4>r_\La$. By the uniform upper bound on the gradient, $\nabla u$ belongs to a fixed compact subannulus of $\An$. The interior Hessian estimate therefore gives $|D^2u|\leq C$ on the remaining set. We can just sum the two contributions, we deduce
\eqref{eqn: curvature estimate eta intro}, since $\db=2r_\La$.
\end{proof}
\section{Extensions of the method and the $p$-Laplacian case}\label{sec: p lap}
The arguments used to prove Theorem~\ref{thm: main} are robust. In this section we show how they can be applied to obtain singular set estimates for the free boundary for minimizers of the Alt-Caffarelli energies of  $p$-Laplace type.
Our goal is to demonstrate the method and to show how minor adaptations of the arguments above already yield the best known results on the size of the singular set in this context. 

\subsection{A Meta Theorem}
To begin, let us recall in broad strokes what ingredients are needed to run the proof of Theorem~\ref{thm: main}. For a compact and connected collection of integrands $I$ containing $|\cdot|^2$, let $I^*$ denote the (nonempty) subset of those integrands such that any minimizer has $\Hm^{n-2}(\Sing(u))=0$. 
To show that $I=I^*,$ we need:
\begin{enumerate}
    \item\label{p1} Geometric estimates, compactness and convergence, and $\e$-regularity that are uniform for all integrands in $I$, for minimizers of the base functional $\J$ as well as the perturbed functional \eqref{def: J f mu}. 
    \item\label{p2} A uniform curvature estimate along the lines of \eqref{eqn: curvature estimate intro} that holds for minimizers when $\Phi \in I^*$, which has suitable lower semicontinuity properties and gives rise to an $\e$-regularity theorem like Lemma~\ref{lem: Hessian eps reg}. Again, this should also be shown  for the perturbed functional \eqref{def: J f mu}.
\end{enumerate}
To then upgrade the estimate to $\Hm^{n-2-\eta}(\Sing(u))=0,$ we additionally need:
\begin{enumerate}[resume]
    \item \label{p3}
    A curvature estimate along the lines of \eqref{eqn: curvature estimate eta intro}, which again gives rise to an $\e$-regularity theorem like Lemma~\ref{lem: Hessian eps reg}. This upgraded estimate follows from a doubling estimate like Lemma~\ref{lem: doubling}, which basically needs uniform $C^0$ estimates for the Hessian of the interior equation, both in the interior and up to a ``nice enough'' boundary, and  a higher regularity result like Corollary~\ref{cor: higher reg} giving the corresponding ``nice enough'' regularity at flat free boundary points.
\end{enumerate}
Indeed, given ingredient \eqref{p1}, and in particular uniform $\e$-regularity and its standard corollary of  propagation of singularities (c.f. Corollary~\ref{cor: prop of singularities}), the argument of Step 1 of the proof of Theorem~\ref{thm: main}  shows $I^*$ is open in $I$. To show closedness, let $\Phi_k \in I^*$ with $\Phi_k\to \Phi$ and a minimizer $u$ of $\J$. Ingredient \eqref{p1} will yield  compactness and convergence lemmas akin to Lemma~\ref{lem:compactness} and \ref{lem:compactness2}. This means that, for the perturbed functionals chosen as in Step 2 of the proof of Theorem~\ref{thm: main},
 minimizers $u_k$ of the perturbed functionals will converge to $u$. Ingredient \eqref{p2} guarantees that the $u_k$ satisfy a  uniform curvature estimate and that this curvature estimate passes to the limit. Then, the corresponding $\e$-regularity theorem and the argument of Step 2 of the proof of Theorem~\ref{thm: main} shows that $I^*$ is closed in $I$.
Finally, Ingredient \eqref{p3}  and the argument of Step 3 of the proof of Theorem~\ref{thm: main} give the upgraded estimate.

We have already seen that it is possible to deal with cases where $\Phi$ degenerates or is singular outside a good set. To do so, it is important to additionally have a ``good ball lemma'' like Lemma~\ref{lem: ball with big gradient} showing that $\na u$ lies in the good set for a definite set of points. It is also important for $\Phi$ to be well-behaved in a neighborhood of the slopes satisfying the free boundary condition.

\subsection{The $p$-Laplacian}

Let us run through this framework in the context of the $p$-Laplacian counterpart of the Alt-Caffarelli energy, establishing Theorem~\ref{thm: p intro}. For $p \in (1,\infty)$, consider
\[
\E_p(u, \Omega) = \int_{\Omega} (|\na u|^p + \chi_{\{u>0\}} ) \,dx  \,
\]
and its corresponding penalized functional $\E_{p,\mu,f} = \E_p(\cdot ) + \int |\mu u-f|^2$ for $\mu \in [0,1]$ and $f \in C^\infty(\Omega)\cap L^\infty(\Omega)$;  compare \eqref{def: J f mu}.
A minimizer of $\E_{p,\mu,f}(\cdot)$ with suitably regular free boundary $\GU$ satsifies the free boundary problem
\begin{equation}
    \label{eqn: p EL}
    \begin{cases}
    \divv(|\na u|^{p-2} \na u) =\frac{2}{p}\mu(\mu u-f) & \text{ on }U\\
     |\na u|^{p} = \frac{1}{p-1}& \text{ on } \GU.
    \end{cases}
\end{equation}
Here we let $U= \Omega \cap \{u>0\}$. {For bounded right-hand sides, interior $C^{1,\beta}$ estimates for the $p$-Laplacian follow from \cite{DiBenedetto1983p}, while estimates up to the boundary on controlled smooth domains are provided by \cite{Lieberman1988p}, with estimates just like \eqref{eqn: gradient reg interior} and \eqref{eqn: gradient reg boundary}. Their constants can be chosen uniformly for $p$ in a fixed compact interval in $(1,\infty)$. On a region where $0<\delta\leq|\nabla u|\leq L$, the equation is uniformly elliptic. For H\"older right-hand sides, Schauder estimates then give $C^{2,\beta}$ bounds akin to \eqref{eqn: Hessian bound interior} and \eqref{eqn: Hessian bound boundary}.} 
\vem

Let $I\subset (1,\infty)$ be a connected compact interval containing $2$, and let  $I^*$ be the set of those $p \in I$ such that $\Hm^{n-2}(\Sing(u))= 0$ for any minimizer $u$ of $\E_p(\cdot)$. Note $I^*$  is nonempty, as $2\in I^*$ by Theorem~\ref{thm: AC}. Let us verify that each of the ingredients of the previous section hold in order to prove Theorem~\ref{thm: p intro}.

Ingredient \eqref{p1} above has already been established in this setting. Geometric estimates were shown by Petrosyan-Valinoci \cite{PV05} for the main functional $\E_p(\cdot)$, and as in Section~\ref{sec: reg}, minor adaptations of their proofs extend their results to minimizers of the perturbed functional $\E_{p,\mu,f}(\cdot)$. Minimizers can be shown to be viscosity solutions by arguing as in Proposition~\ref{prop: min is viscosity solution} or as in \cite{DK}. For viscosity solutions, $\e$-regularity was shown by Danielli-Petrosyan \cite{DP} with zero right-hand side and with nonzero right-hand side \cite{FL21}. An inspection of the proofs and a simple contradiction argument shows that all estimates are uniform for $p$ in the compact set $I$ and for $\|f \|_{L^\infty}\leq C$; in fact, this is already made explicit in \cite{FL21}. The geometric estimates also give rise to compactness and convergence lemmas by repeating the proofs of Lemma~\ref{lem:compactness} and \ref{lem:compactness2}.

We need to establish Ingredient \eqref{p2}. The key curvature estimate is the following, compare to Proposition~\ref{prop: curvature estimates}:
\begin{proposition}\label{prop: curvature est p}
For $p\in I^*$, $\mu \in [0,1]$, and $f \in C^\infty(\Omega)\cap L^\infty(\Omega)$ nonnegative, there exist $C>0$ depending on $n, I$, and an upper bound for $ \| f\|_{L^\infty(\Omega)}$ such that if $u$ is  a minimizer of $\E_{p, \mu,f}$ in $\Omega,$ then for all $x_0\in \GU$ and $r \in (0,1]$ with $B_{2r}(x_0)\subset \Omega,$
we have
    \begin{equation}
        \label{eqn: p curv}
    \int_{B(x_0,r)\cap \{u>0\} \cap \{|\na u |>0\}} |D^2 u|^2 |\na u |^{p-2 + (p-2)_+} \leq C r^{n-2}\,. 
    \end{equation}
\end{proposition}
We prove Proposition~\ref{prop: curvature est p} in Section~\ref{ssec: pf p curvature} below. The weight appearing on the left-hand side of \eqref{eqn: p curv} is natural in the $p$-Laplacian setting. When $p\in (1,2)$, this weight is bounded below thanks to the Lipschitz bound, but can be genuinely degenerate in the case $p>2$. 

Nevertheless, the estimate \eqref{eqn: p curv} is lower semicontinuous with respect to local $H^1$ and local uniform convergence, and gives rise to an $\e$-regularity lemma like Lemma~\ref{lem: Hessian eps reg}, arguing along the same lines as we have above.

Indeed, suppose $u_k \to u$ in $B_2$ the sense of Lemma~\ref{lem:compactness} and \ref{lem:compactness2}, suppose $0\in \G$, and suppose each $u_k$ satisfies the conclusion of Proposition~\ref{prop: curvature est p} in $B_2.$ We claim that \eqref{eqn: p curv} holds for $u$ with $x_0=0$ and $r=1$. 

To see this, we argue as in Step 2 of the proof of  Lemma~\ref{lem: doubling} and the proof of Lemma~\ref{lem: hessian bound implies small sing set}. For $\delta>0$, let $E_\delta = B_1\cap \{|\na u|>\delta\}\cap \{u>0\}$, and fix open sets $D\Subset D'\Subset E_\delta$. Since $u_k \to u$ locally uniformly, interior $C^{1,\beta}$ estimates for the $p$-Laplacian also guarantee that $\na u_k \to \na u$ locally uniformly. Hence for $k$ large enough, $|\na u_k|>\delta$ and $u_k>0$ on $D'$. In turn, uniform $C^{2,\beta}$ estimates away from the critical set show that 
\begin{equation}
    \label{eqn: intermed p}
\int_{D} |\na u|^{p-2+(p-2)_+} |D^2u|^2 = \lim_{k\to \infty} \int_{D} |\na u_k|^{p-2 +(p-2)_+} |D^2u_k|^2  \leq C.
\end{equation}
In the second inequality we apply \eqref{eqn: p curv} to $u_k$ with $x_0=x_k$ and $r = 1- o_k(1).$
Taking an exhaustion of $E_\delta$ extends the domain of integration to $E_\delta$ on the left-hand side of \eqref{eqn: intermed p}. Since this holds for all $\delta>0$, we obtain the claim. 

The $\e$-regularity lemma akin to Lemma~\ref{lem: Hessian eps reg} is deduced from the $\e$-regularity of \cite{DP, FL21} by repeating this argument, as we did in  the proof of Lemma~\ref{lem: Hessian eps reg}. {In identifying the limiting minimizer, we use the localized version of Lemma~\ref{lem:rigiditystrict}: on each positivity component, $\nabla u$ is continuous, while $W^{2,1}_{\mathrm{loc}}$ regularity is needed only where $|\nabla u|>\delta$. Lemma \ref{lem: uti} below ensures that this set is nonempty.} Therefore, Ingredient \eqref{p2} holds. Thus $I=I^*$.

This shows that $\Hm^{n-2}(\Sing(u))=0$ for a minimizer of $\E_p(\cdot)$ for any $p\in I$. To improve the estimate to $\Hm^{n-2-\eta}(\Sing(u))=0$, where $\eta$ depends on $n$ and $I$, we need Ingredient \eqref{p3}. The same challenge arises here as arose in our main theorem: the interior equation only admits uniform $C^2$ estimates outside a neighborhood of the critical set. (Thanks to the Lipschitz estimate \cite{PV05}, we do not have to worry about degeneracy when the gradient blows up.) This challenge can be dealt with in the same way, with a ``good ball lemma'' like Lemma~\ref{lem: ball with big gradient}:

\begin{lemma}\label{lem: uti}
    Suppose that $u$ minimizes $\E_p$ in $B_1$ and $0\in \GU$. 
    There is a $\db=\db(n, I)$ such that every connected component
    $V$ of $B_1\cap \{u>0\}$ that intersects $B_{1/8}$ 
    contains a ball $B_{\db}(y_V) \subset V\cap B_{1/2}
    \cap \{|\na u| \geq 8\db\}$.
\end{lemma}
This lemma is proven using the uniform Lipschitz, nondegeneracy, and interior $C^{1,\beta}$ estimates via exactly the same argument as the proof of Lemma~\ref{lem: ball with big gradient}. In view of the uniform interior and boundary $C^{2,\beta}$ estimates that hold away from the critical set, Ingredient~\eqref{p3} is proven exactly as Proposition~\ref{prop: curvature estimate 2}. We thus see that Theorem~\ref{thm: p intro} holds.

\subsection{Proof of the curvature estimate}\label{ssec: pf p curvature}
As we have seen above, the new estimate needed for Theorem~\ref{thm: p intro} is Proposition~\ref{prop: curvature est p}. We now prove the proposition. Like Proposition~\ref{prop: curvature estimates}, its proof is based on the stability inequality and a positive supersolution to the linearized equation.  
Similar computations to those in Appendix~\ref{app: variations} show that a minimizer of $\E_{p, \mu,f}(\cdot)$ in $\Omega$ satisfies the following stability inequality:   
\begin{equation}
    \label{eqn: stability p}
\int_{\{u>0\}\cap\{|\na u|>0\}} \!\!\! A(\na u)[\na \xi, \na \xi] + \mfrac{2}{p} \mu^2 \xi^2\, dx  \geq 
-(p-1)^{3/p}\int_{\GU} \partial_{\nu \nu} u\, \xi^2  \, d\Hm^{n-1}
\qquad \text{for all }\xi \in C^1_c(\Omega \setminus \Sing(u))\,.
\end{equation}
Here and below, $\nu$ is the outer unit normal to $\{u>0\}$ and 
\begin{equation}
    \label{A}
A(\na u) = |\na u|^{p-2}\Big(\text{Id} + (p-2) \mfrac{\na u}{|\na u|} \otimes \mfrac{\na u}{|\na u|}\Big)\,.
\end{equation}

\begin{proof}[Proof of Proposition~\ref{prop: curvature est p}]

The basic scheme of the proof is the same as the in Proposition~\ref{prop: curvature estimates}: a power of $|\na u|$ is a subsolution to the linearized equation and its right-hand side controls the norm of the Hessian. We test this equation with a cutoff function and relate it to the stability inequality. The case $p\geq 2$ is essentially the same as Proposition~\ref{prop: curvature estimates}, while $p<2$ requires a more refined analysis using the Kato-type inequality \eqref{eqn: kato} below.
In what follows, $C$ and $c$ denote positive constants that may change from line to line and depends on $n$, $I$, and $\|f\|_{L^\infty(B_2)}$. We will also write $A$ as shorthand for $A(\na u)$.

{\it Step 0: Set up.}
 By scaling, we may assume $x_0=0$ and $r=1$, so  $u$ minimizes $\E_{p,\mu,f}(\cdot)$ on $B_2$. Thanks to the Lipschitz estimates, we may assume $|\na u|\leq L$ in $B_{3/2}$. Let $U = \{u>0\} \cap B_{3/2}$, $\G = \partial U \cap B_{3/2}$, and let $\nu$ denote the outer unit normal to $U$.  
We will argue separately for $p\geq 2$ and $p<2$, starting with the simpler case $p\geq 2.$

{\it Step 1 for $p\geq 2$: Cutoff function vanishing on $\Sing(u)$.} In this step, let $p\geq 2$ and let  $G = |\na u|^p$.
For $\e >0$, let $U_\e = U\cap \{G>\e\}$ be the superlevel sets of $G$. Direct computation a la Lemma~\ref{lem: G eqn} shows that   for each $\e>0$, 
    \begin{equation}
        \label{eqn: G eqn p ge 2}
    \begin{cases}
        \divv(A(\na u)\na G) \ge C |\na u|^{2(p-2)}|D^2 u|^2 -2 \mu |\na u|^{p-2} \na u \cdot \na f& \text{ on }U_\e \\
       \langle A \na  G, \nu\rangle = -p(p-1)^{(3-p)/p}\partial_{\nu \nu} u & \text{ on }\GU,
    \end{cases}
        \end{equation}

Let  $\varphi\in C_c^1(B_{3/2}\setminus \Sing(u)).$ Let $\e \ll 1$ be a regular value of $G$ so that $\partial U_\e \cap U$ is $C^1$ and $U_\e$ is a set of finite perimeter. Let $\eta: \R_+ \to \R$ be an increasing function such that $\eta(t)=0$ for $t \leq \e$ and $\eta \equiv 1$ on $\{t>\frac{1}{2(p-1)}\}$.  
We multiply $\divv(A(\nabla u)\nabla G)$ by 
$$\varphi^2 \, \eta\circ G$$
on $U_\e$ and integrate by parts. The portion of boundary term on $\partial U_\e \cap U$ vanishes since $\eta(G(x))= \eta(\e)$ is zero there.
 Moreover, for all $x \in \G \cap \text{spt}{\varphi}$ we have $x\in U_\e$ and $\eta(G(x))=1$, thanks to the assumption that $\varphi$ vanishes on $\Sing(u)$ and  the free boundary condition $G(x) = 1/(p-1)$ on $\Reg(u)$. Thus, writing $A=A(\na u)$, we obtain
\begin{equation}\label{eqn: int eqn G}
\begin{split}
  \int_{U_\e}\varphi^2\eta(G) \divv(A\nabla G)
  & =  -\int_{U_\e}\varphi^2 \,\eta'(G)\, 
 \langle A\nabla G,\nabla G\rangle -   \int_{U_\e }2 \eta(G)\varphi\, 
 \langle A\nabla G,\nabla\varphi\rangle 
 + \int_{\G} \varphi^2 A[\nabla G, \nu] \, d\Hm^{n-1}\\
 &= I + II +III .   
\end{split} 
\end{equation}
We bound the right-hand side above. Term $I$ has a sign that allows us to discard it: since $\eta'\geq 0$ and $\langle A q,q\rangle \geq 0$, we see that $I\leq 0.$ For term $II$, keeping in mind \eqref{A}, we have $ |A \nabla G|\leq C |\na  u|^{2p-3}|D^2u| $. So, for $\delta>0$ small to be specified later and using $\eta \leq 1$ and the Lipschitz bound for $u$,
\begin{equation}
    \label{eqn: term 1}
 |II|\leq C\delta \int_{U_\e}\eta(G)\varphi^2|\na u|^{2(p-2)}|D^2u|^2
 +\frac{C}{\delta}
 \int_{U}|\nabla\varphi|^2.
\end{equation}
The boundary condition  for $G$ in \eqref{eqn: G eqn p ge 2} 
shows that $III$ is the left-hand side of the stability inequality \eqref{eqn: stability p}, so
$$
III =-C\int_{\G }
 \partial_{\nu \nu} u\, \varphi^2 
 \leq C \int_U A[\na \varphi, \na \varphi] \leq C\int_U |\na \varphi|^2\,.
 $$
 In the final inequality, we used $p-2\geq 0$ and the Lipchitz bound for $u$ to say $A[\na \varphi, \na \varphi] \leq C|\na \varphi|^2$.
Next, we  bound the left-hand side of \eqref{eqn: int eqn G} below using  \eqref{eqn: G eqn p ge 2}.
Choosing $\delta$ to absorb the first term on the right-hand side of \eqref{eqn: term 1} into this lower bound, exactly as in the proof of Proposition~\ref{prop: curvature estimates}, we obtain
\begin{align*}
\int_{U_\e}\eta(G)\varphi^2 |\na u|^{2(p-2)}|D^2u|^2 \leq C \int_{U_\e}\eta(G)\varphi^2\,
 |\na u|^{p-2} \nabla u\cdot\nabla f\,dx
+C
 \int_{U}|\nabla\varphi|^2 \,dx  
\end{align*}
Now, take an increasing sequence $\eta_k$ tending to  $\chi_{(\e,\infty)}$ and then send $\e \to 0$. Using dominated convergence theorem for the right-hand side and monotone convergence theorem on the left-hand side, we find
\begin{align}\label{eqn: Hessian bound intermediate p}
\int_{U}\varphi^2|D^2u|^2 |\na u|^{2(p-2)} \leq 
C\int_{U}\varphi^2\,
 |\na u|^{p-2}\nabla u\cdot\nabla f
\, dx+C
 \int_{U}|\nabla\varphi|^2 \,dx  \,.
\end{align}
Finally, integrating by parts in the first term on the right-hand side and bounding this term precisely as in \eqref{eqn: bound f term} in the proof of Proposition~\ref{prop: curvature estimates}, using $f\geq 0$ and the interior equation for $u$,  we find
\begin{equation}\label{eqn: bound f term p}
    \begin{split}
     \int_{U}\varphi^2\, |\na u|^{p-2}\nabla u\cdot\nabla f
  &\leq 2\mu\int_{U}\varphi^2f^2
 -2\int_{U}\varphi f\, |\na u|^{p-2}\nabla u \cdot\nabla\varphi  \leq C \int_U (\varphi^2 f^2  + |\na u|^{2(p-1)} |\na \varphi|^2)\,dx.
    \end{split}
\end{equation}
Combined with \eqref{eqn: Hessian bound intermediate p}, this gives that, for all $\varphi \in C^1_c(B_{3/2} \setminus \Sing(u)),$
\begin{align}\label{eqn: Hessian bound final p}
\int_{U}\varphi^2|D^2u|^2 |\na u|^{2(p-2)} \leq  C
 \int_{U} \varphi^2 +|\nabla\varphi|^2 \,dx  \,.
\end{align}

\vem

{
\noindent\emph{Step 1 for $1<p<2$: Cancellation with a linear truncation.}
Let $\varphi\in C_c^1(B_{3/2}\setminus\Sing(u))$. We denote
$$
 Q=\sum_{i=1}^n A[\nabla u_i,\nabla u_i]-A[\nabla |\nabla u|,\nabla |\nabla u|]\geq0.
$$
Differentiating \eqref{eqn: p EL} we get
\begin{equation}\label{eqn: p norm equation}
\begin{split}
\operatorname{div}(A\nabla |\nabla u|)&=\frac{Q}{|\nabla u|}+\frac{2\mu^2}{p}|\nabla u| -\frac{2\mu}{p}\frac{\nabla u}{|\nabla u|}\cdot\nabla f \quad\text{in }U\setminus \{|\nabla u|(x)=0\},\\
A[\nabla |\nabla u|,\nu]&=-(p-1)^{3/p}u_{\nu\nu}|\nabla u|\quad\text{on }\G.
 \end{split}
\end{equation}
We choose regular values $\varepsilon\downarrow0$ with $\varepsilon<(p-1)^{-1/p}/2$, and set $U_\varepsilon=U \cap \{|\nabla u|(x)>\varepsilon\}$, and $h_\varepsilon=(|\nabla u|-\varepsilon)_+$. We multiply \eqref{eqn: p norm equation} by $\varphi^2h_\varepsilon$, integrate over $U_\varepsilon$, and apply \eqref{eqn: stability p} with $\xi=\varphi h_\varepsilon$. The stability test can be extended across the regular free boundary; its support in the positive phase stays away from $\{|\nabla u|(x)=0\}$. The boundary term on $\{|\nabla u|=\varepsilon\}$ vanishes. Since $\nabla h_\varepsilon=\nabla |\nabla u|$ in $U_\varepsilon$, the terms containing $A[\nabla |\nabla u|,\nabla |\nabla u|]$ and the mixed gradient terms cancel exactly. Using $|\nabla u|=h_\varepsilon+\varepsilon$, we obtain
$$
\int_{U_\varepsilon}\varphi^2\frac{h_\varepsilon}{|\nabla u|}Q+\frac{2\mu^2\varepsilon}{p}\int_{U_\varepsilon}\varphi^2h_\varepsilon \leq \int_{U_\varepsilon}h_\varepsilon^2 A[\nabla\varphi,\nabla\varphi] +\frac{2\mu}{p}\int_{U_\varepsilon} \varphi^2h_\varepsilon \frac{\nabla u}{|\nabla u|}\cdot\nabla f +\varepsilon\int_\G -(p-1)^{3/p}u_{\nu\nu} h_\varepsilon\varphi^2 d\Hm^{n-1}.
$$
The last term tends to zero, since on $\Gamma$ we have that $-(p-1)^{3/p}u_{\nu\nu}$ is integrable and $h_\varepsilon=(p-1)^{-1/p}-\varepsilon$. Moreover,
$$
h_\varepsilon^2A[\nabla\varphi,\nabla\varphi] \leq |\nabla u|^p|\nabla\varphi|^2.
$$
Discarding the nonnegative zero-order term on the left-hand side, and passing to the limit by monotone and dominated convergence we deduce
\begin{equation}\label{eqn: p geometric cancellation}
 \int_{U\setminus \{|\nabla u|(x)=0\}}\varphi^2Q \leq\int_{U\setminus \{|\nabla u|(x)=0\}}|\nabla u|^2A[\nabla\varphi,\nabla\varphi] +\frac{2\mu}{p}\int_U \varphi^2\nabla u\cdot\nabla f.
\end{equation}
A direct computation, which we postpone until the end of the proof, yields the Kato-type inequality since $\Delta_p u=g:=\frac{2}{p}\mu(\mu u -f)$:
\begin{equation}\label{eqn: kato}
 Q\geq c_p |\nabla u|^{p-2}|D^2u|^2-|\nabla u|^{2-p}g^2, \qquad c_p=\frac{(p-1)^2}{2n-1}.
\end{equation}
In particular, since $u$, $|\nabla u|$, and $f$ are bounded and $2-p>0$, \eqref{eqn: p geometric cancellation} and \eqref{eqn: kato} imply that
\begin{equation}\label{eq: finite}
\int_{U\setminus \{|\nabla u|(x)=0\}} \varphi^2|\nabla u|^{p-2}|D^2u|^2<\infty.
\end{equation}
We now remove the dependence on $\nabla f$ in \eqref{eqn: p geometric cancellation}. Integrating by parts, we know that
$$
\int_{U_\varepsilon}\varphi^2h_\varepsilon \frac{\nabla u}{|\nabla u|}\cdot\nabla f=-\int_{U_\varepsilon} f\varphi^2\operatorname{div}(h_\varepsilon \frac{\nabla u}{|\nabla u|})-2\int_{U_\varepsilon} f\varphi h_\varepsilon \frac{\nabla u}{|\nabla u|}\cdot\nabla\varphi-((p-1)^{-1/p}-\varepsilon) \int_\Gamma f\varphi^2\,d\mathcal H^{n-1},
$$
as $h_\varepsilon=0$ on $\{|\nabla u|=\varepsilon\}$ and the outward normal component of $h_\varepsilon \nabla u/|\nabla u|$ on $\G$ is $-((p-1)^{-1/p}-\varepsilon)$. Using that
$$
\operatorname{div}(h_\varepsilon \frac{\nabla u}{|\nabla u|})=\left(1-\frac{\varepsilon}{|\nabla u|}\right)\Delta u +\frac{\varepsilon}{|\nabla u|}D^2u[\frac{\nabla u}{|\nabla u|},\frac{\nabla u}{|\nabla u|}] \quad\text{in }U_\varepsilon,
$$
and recalling that $f\geq 0$, we obtain
$$
\int_{U_\varepsilon}\varphi^2h_\varepsilon \frac{\nabla u}{|\nabla u|}\cdot\nabla f \leq \sqrt n\,\|f\|_\infty \int_{U_\varepsilon}\varphi^2|D^2u| + 2\|f\|_\infty \int_{U_\varepsilon}|\varphi||\nabla\varphi| |\nabla u|.
$$
Thanks to \eqref{eq: finite}, we can pass the latter inequality to the limit as $\e \to 0$ and deduce
$$
 \int_U\varphi^2\nabla u\cdot\nabla f \leq \sqrt n\,\|f\|_\infty \int_{U\setminus \{|\nabla u|(x)=0\}}\varphi^2|D^2u| +2\|f\|_\infty\int_U |\varphi||\nabla u||\nabla\varphi|,
$$
which in turn, by Young's inequality and the bound $|\na u|\leq L$, provides
$$
\frac{2\mu}{p}\int_U\varphi^2\nabla u\cdot\nabla f \leq\frac{c_p}{2}\int_{U\setminus \{|\nabla u|(x)=0\}} \varphi^2|\nabla u|^{p-2}|D^2u|^2 +C\int_U\bigl(\varphi^2+|\nabla\varphi|^2\bigr).
$$
Combining this with \eqref{eqn: p geometric cancellation} and \eqref{eqn: kato} proves
\begin{equation}\label{eqn: step 2 final}
\int_{U\setminus \{|\nabla u|(x)=0\}}\varphi^2|\nabla u|^{p-2}|D^2u|^2\leq C\int_U\bigl(\varphi^2+|\nabla\varphi|^2\bigr),
\end{equation}
where the constant $C$ depends only on $n$, $I$, and $\|f\|_\infty$.
}

{\it Step 3: Capacity argument.} We now unify the cases for $p\geq 2$ and $p<2$. Arguing as in Step 2 of the proof of Proposition~\ref{prop: curvature estimates}, we can localize a cutoff function  around the singular set to extend \eqref{eqn: Hessian bound final p} and \eqref{eqn: step 2 final} to general $\varphi \in C^1_c(B_{5/4})$. 
Indeed, by the same blowup argument as Lemma~\ref{lem: perturbed functional small singular set} and the assumption $\Phi \in I^*,$ we know $\Hm^{n-2}(\Sing(u)\cap B_2)=0.$  The set $\Sigma = \Sing(u)\cap \overline{B}_{5/4}$ is compact thanks to epsilon regularity \cite{DP, FL21}. So, since $\Hm^{n-2}(\Sigma)=0$, we can repeat the construction of the proof of Proposition~\ref{prop: curvature estimates}: for  any $\e>0$, we obtain a cutoff function $\psi$ vanishing on $\Sing(u)$ so that for any $\varphi \in C^{\infty}_c(B_2)$ with support contained in $B_{5/4}$, we have 
\begin{equation}
    \label{eqn: capacity thing2}
|\na (\psi\varphi)(x)|^2 \leq 2|\na \varphi(x)|^2+ 2|\varphi|^2 \sum_{k} r_k^{-2} \chi_{B_{r_k}(z_k)}, 
\end{equation}
and $\psi\varphi\in C^1_c(B_{3/2}\setminus \Sing(u))$, and $\psi \varphi \to \varphi$ a.e. as $\e \to 0$. This penultimate point means that  \eqref{eqn: step 2 final} holds for $\zeta=\psi \varphi,$ so
$$
\int_{U}(\psi \varphi )^2|D^2u|^2 |\na u|^{p-2+(p-2)_+}  \leq  C \int_{U}|\na \varphi|^2\,dx +C \| \varphi\|_{L^\infty}^2 ( \sum_k r_k^{n-2} +1 )
$$
Take a sequence $\e_k \to 0$ and the corresponding $\psi = \psi_{\e_k}$ and $\zeta_k = \psi_{\e_k}\varphi$, applying
 Fatou's lemma to pass to the limit on the left-hand side.
Choosing a cutoff $\varphi$ with $\varphi\equiv 1$ in $B_{1}$ and $|\na \varphi|\leq C(n)$, we conclude.

{\it Step 4: Proof of the Kato-type inequality.} It remains to show \eqref{eqn: kato}. 
At a point $x \in U \setminus\{|\na u|=0\}$, choose an orthonormal basis such that $\nu = -\na u/|\na u| = e_1$. This way, we have $\sum|\na u_k \cdot \nu|^2 = \sum (\partial_k |\na u|)^2 = |\na |\na u||^2$ and $|\na |\na u|\cdot \nu|^2 = u_{11}^2$. 
 Hence,
\begin{equation}\label{eqn: kato pf}
\begin{split}
\sum_k\langle A \na u_k ,  \na u_k\rangle - \langle A\na |\na u|, \na |\na u|\rangle
&= |\na u|^{p-2} \Big\{\sum_{j,k\geq 1} u_{jk}^2  -\sum_{k\geq 1} u_{1k}^2
   +(p-2)\Big(\sum_{k\geq 1} u_{1k}^2- u_{11}^2 \Big)\Big\}\\
&= |\na u|^{p-2} \Big\{\sum_{j,k>1} u_{jk}^2  +(p-1)\sum_{k>1} u_{1k}^2\Big\}.
\end{split}
\end{equation}
The interior equation in
\eqref{eqn: p EL} in this basis becomes $g=\divv(|\na u|^{p-2} \na u)= |\na u|^{p-2}(\Delta u +(p-2) u_{11})$, or, after rearranging,
\[
(p-1)\, u_{11} = -\sum_{k>1} u_{kk} + g\,|\na u|^{2-p}.
\]
Squaring this expression and applying Cauchy--Schwarz inequality
 then yields
\[
u_{11}^2 \leq \mfrac{2(n-1)}{(p-1)^2} \sum_{j,k>1} u_{jk}^2 + \mfrac{2}{(p-1)^2}\, g^2 |\na u|^{4-2p}.
\]
In turn, we find that
\begin{align}\label{eqn: 1016}
|D^2 u|^2 & = \sum_{j,k>1} u_{jk}^2  +2\sum_{k>1} u_{1k}^2 + u_{11}^2\leq  \Big(1 + \mfrac{2(n-1)}{(p-1)^2}\Big)\sum_{j,k>1} u_{jk}^2  +2\sum_{k>1} u_{1k}^2
  + \mfrac{2}{(p-1)^2}\, g^2 |\na u|^{4-2p}.
\end{align}
As above, we let $c_p:=\frac{(p-1)^2}{2n-1}$. Since $1<p\leq 2$ and $n\geq 2$, we have
$c_p(1 + \frac{2(n-1)}{(p-1)^2}) = \frac{(p-1)^2+2(n-1)}{2n-1}\leq 1$. Moreover, 
$2c_p\leq p-1$ and $\frac{2c_p}{(p-1)^2} = \frac{2}{2n-1}\leq 1$. Therefore, multiplying \eqref{eqn: 1016}  by $c_p|\na u|^{p-2}$ gives us 
\[
c_p|\na u|^{p-2}|D^2 u|^2 \leq |\na u|^{p-2} \Big\{\sum_{j,k>1} u_{jk}^2  +(p-1)\sum_{k>1} u_{1k}^2\Big\}
 + g^2 |\na u|^{2-p}.
\]
Together with \eqref{eqn: kato pf}, this gives \eqref{eqn: kato}.
\end{proof}

\subsection{Further Applications}\label{rmk: meta thm app}

The framework above suggests extensions beyond the standard $p$-Laplacian. Indeed, there are several classes of integrands for which a singular set bound like Theorem~\ref{thm: main} is an open problem, and for which we believe the meta theorem above may be applicable.
Proving such a theorem via this method in these settings would require establisxhing Ingredients~\eqref{p1}-\eqref{p3} uniformly along a compact, connected family containing a quadratic integrand, together with an appropriate good-ball lemma.

A first class consists of anisotropic Orlicz energies
$$\Phi(\xi)=G(F(\xi)), \qquad 0<m\leq\frac{tG''(t)}{G'(t)}\leq M<\infty \quad(t>0),$$
where $F$ is a sufficiently smooth norm away from the origin, $D^2(F^2/2)$ is uniformly elliptic there, and $G(0)=G'(0)=0$, with $G'>0$ and sufficient regularity on $(0,\infty)$. This includes smooth anisotropic $p$-energies $F(\xi)^p$, as well as the nonhomogeneous examples $|\xi|^p+a|\xi|^q$, with $a\geq0$ and $1<p\leq q<\infty$. Suitable logarithmic modifications of power growth provide further candidates. The geometric estimates of \cite{PV05} and the radial Orlicz free-boundary theory of \cite{MW08} provide relevant starting points. Uniform estimates for the penalized problems and the required weighted curvature bounds would still have to be established.

A related direction concerns spatially dependent energies, for example $\int_\Omega (|\nabla u|^{p(x)}/p(x) +\lambda(x)\chi_{\{u>0\}})\,dx$, with sufficiently regular $p(x)$ bounded away from $1$ and infinity, and positive $\lambda(x)$ bounded away from zero. Blow-ups freeze these coefficients, hence the
Hausdorff dimension estimate transfers the corresponding
constant-exponent singular-set estimates, provided
compactness to minimizers of the frozen functionals and
uniform propagation of singularities hold.
This approach would not require a new curvature estimate
for the variable-coefficient equation; see \cite{FL21}
for the variable-exponent improvement-of-flatness theory.

{The orthotropic, or pseudo-$p$-Laplacian, associated with $\Phi(\xi)=\sum_{i=1}^n|\xi_i|^p$, presents a different difficulty. For $p\ne2$, its Hessian is degenerate or singular on coordinate hyperplanes. These hyperplanes contain nonzero gradients satisfying the
Bernoulli condition
$D\Phi(\xi)\cdot\xi-\Phi(\xi) =(p-1)\sum_{i=1}^n|\xi_i|^p=1$; for instance, $\xi=(p-1)^{-1/p}e_1$. Thus, a lower bound for $|\nabla u|$ alone does not provide the ellipticity required in Ingredient~\eqref{p3}. An extension to this setting would require curvature and regularity estimates adapted to the coordinatewise degeneracy, rather than a direct adaptation of the argument above.}

\appendix

\section{First and second variation}\label{app: variations}
Let $f \in C^\infty(\Omega)$ be a nonnegative function and $\mu \in [0,1]$. We compute the first and second variations of the functional $
\Jfm$ defined in \eqref{def: J f mu}
at a minimizer of the functional. To make it apparent where derivatives fall, we will use the notation
\[
A(p)=D\Phi (p) \qquad \text{ and } \qquad Q(p)=D^2\Phi(p).
\]
Taking outer variations $u_\varepsilon=u+\varepsilon\varphi$ for $\varphi\in C_c^\infty(U)$
shows that a minimizer  $u$  of $\Jfm$ satisfies
\begin{equation}\label{anislap}
-\operatorname{div}(A(\nabla u))+2\mu (\mu u-f)=0 \quad \text{in } \{ u>0\} \cap \Omega\,.
\end{equation}

Throughout the section, we assume $u\in W^{1,2}(\Omega)$ is continuous and nonnegative. Letting $U=\{u>0\}\cap \Omega$, we assume that $\G = \partial U\cap \Omega$ is of class $C^2$ and that $u$ solves \eqref{anislap}.  We let $\nu$ denote the outer unit normal to $U$. 

Let $\{\FL_t\}_{t \in (-\e,\e)}$ be a smooth one-parameter family of diffeomorphisms on $\R^n$ with \(\{\FL_t\neq Id\} \subset \subset \Omega\) for all $t \in (-\e,\e)$ and \(\FL_0=Id\). Let $X = \frac{d}{dt}|_{t=0} \FL_t$ denote the initial velocity of this flow, and let \(U_t:=\FL_t(U)\) and $\G_t = \FL_t(\G)$ so that $\partial U_t \cap \Omega = \G_t$. Let $u_t$ be the unique weak solution to
\[
-\divv (A(\nabla u_t))+2\mu (\mu u_t-f)=0 \qquad \text{ in } U_t
\]
 with \(u_t=0\) on $\G$ and $u_t=u$ on $\partial \Omega$.
\begin{lemma}\label{lem: first var} In the setup above, 
\begin{equation}
    \label{eqn: first var}
\mfrac{d}{dt}\Big|_{t=0}\,  \Jfm(u_t)
=
\int_\Gamma \bigl(1 +{\Phi(\nabla u)- A(\na u )\cdot \na u }\bigr)\, X\cdot \nu .
\end{equation}
In particular, if $u$ is a minimizer of $\Jfm$, then ${ A(\na u )\cdot \na u -\Phi(\na u) = 1}$ on $\G.$
\end{lemma}
Before showing Lemma~\ref{lem: first var}, we collect a few facts. To write the first variation as an integral over \(\Gamma\) rather than \(U\), we will use an
anisotropic version of the Rellich identity: for \(u\in C^2(U)\) and for a \(C^1\) vector field  \(X\),
\begin{equation}\label{Rellich}
\operatorname{div}X\,\Phi(\nabla u)
-
\, DX^T\nabla u \cdot A(\na u)
=
\divv \bigl(\Phi(\nabla u) X-(X\cdot \nabla u)A( \na u) \bigr)
+(X\cdot \nabla u)\divv (A (\na u)).
\end{equation}
The proof of \eqref{Rellich} is a direct computation:
\begin{equation*}
\begin{split}
\divv \big(\Phi&(\nabla u)X-(X\cdot \nabla u)A(\na u) \big)\\
&=
\Phi(\nabla u)\divv X
+ D^2 u(A(\na u),  X)
-  D^2u (X, A(\na u))
- DX^T\na u \cdot A(\na u)
-(X\cdot \nabla u)\operatorname{div}(A(\na u)).
\end{split}
\end{equation*}
We also observe that 
\begin{equation}
    \label{eqn: inverse flow D}
D\FL_t = Id + tDX + o(t) \qquad \text{ and } \qquad (D\FL_t)^{-1}=Id-tDX+o(t),
\end{equation}
and recall that the Jacobian $J\FL_t=| \det D\FL_t|$ has
\begin{equation}
    \label{eqn: Jacobian prime}
\mfrac{d}{dt}\Big|_{t=0}J\FL_t=\divv X,
\end{equation}
see, e.g. \cite[Lemma 17.4]{MaggiBook}.  With these tools in hand, we can prove the lemma:
\begin{proof}[Proof of Lemma~\ref{lem: first var}]
We compute derivatives of the three terms in the energy separately. First, by the area formula, we have $
\int_\Omega \chi_{\{u_t>0\}}\,dx=
\int_U  J\Phi_t(y)\,dy,$
and so \eqref{eqn: Jacobian prime} gives 
\begin{equation}\label{firstterm}
\mfrac{d}{dt}\big|_{t=0}\int_\Omega \chi_{\{u_t>0\}}\,dx =\int_U \divv X dy=\int_\Gamma X\cdot \nu.
\end{equation}
Now, let $\tilde u_t(x)=u_t\circ \FL_t$ be the pullback of $u_t$, so that the positivity set of $\tilde{u}_t$ is $U$. By the chain rule, we have
\begin{equation}\label{tilde}
\nabla u_t|_{\FL_t(x)}
=
D\FL_t^{-T}(x)\,\nabla \tilde u_t(x) \qquad \text{ and} \qquad \mfrac{d}{dt}\Big|_{t=0} \tilde{u}_t = \dot{u} + \na u \cdot X\,
\end{equation}
where we let $\dot{u} = \frac{d}{dt}|_{t=0} u_t$.
Therefore, applying the area formula followed by the first identity in \eqref{tilde}, we write the first term in the energy as
\[
\int_{U_t} \Phi(\na u_t(y))\,dy
=\int_{U_t} \Phi\bigl(\nabla u_t|_{\FL_t\circ \Psi_t(y)}\bigr) \,dy=
\int_U \Phi\bigl(\nabla u_t|_{\FL_t(x)}\bigr) J\FL_t(x)\,dx=
\int \Phi\bigl(D\FL_t^{-T}(x)\nabla \tilde u_t(x)\bigr)\, J\FL_t(x)\,dx .
\]
From the second term in \eqref{eqn: inverse flow D}, we have $
\frac{d}{dt}|_{t=0}\, D\FL_t^{-T}=-DX^T$. 
Therefore, letting \(w=\left.\frac{d}{dt}\right|_{t=0}\tilde u_t\), we have
\begin{equation}\label{dertilde}
\mfrac{d}{dt}\Big|_{t=0} D\FL_t^{-T}(x)\nabla \tilde u_t(x)
=
- DX^T\nabla u + \nabla w\,.
\end{equation}
Hence, using this and \eqref{eqn: Jacobian prime} followed by the Rellich identity \eqref{Rellich}, we have
\begin{equation}
    \label{eqn: intermed first var 1}
    \begin{split}
\mfrac{d}{dt}\Big|_{t=0}\, \int_{U_t} \Phi(\na u_t) \,dx
&=
\int_U \divv X\, \Phi(\nabla u) -  A(\na u)\cdot
 DX^T\nabla u +  A(\na u)\cdot \nabla w \\
& = \int_U \divv \big(\Phi(\nabla u) X-(X\cdot \nabla u)A( \na u) \big)
+(X\cdot \nabla u)\divv (A (\na u)) + A(\na u)\cdot \nabla w\\
& = \int_{\G} \big(\Phi(\nabla u) X-(X\cdot \nabla u)A( \na u) \big)\cdot \nu
+ \int_U(X\cdot \nabla u - w)\divv (A (\na u))\,.
    \end{split}
\end{equation}
In the final identity, we applied the divergence theorem and integrated by parts in the final bulk term. Note that the integration by parts produces no boundary term, since  $\tilde u_t(x)=u_t\circ \Phi_t(x)=0$ on \(\Gamma\) for all $t$, so differentiating shows \(w(x)=0\) for all \(x\in \Gamma\). 

We may simplify the boundary term in \eqref{eqn: intermed first var 1}. 
{Since $u$ is constant on $\G$, we have $\nabla u=\partial_\nu u\,\nu$, thus 
\begin{equation}\label{eqn: euler consequence}
    \begin{split}
(\nabla u \cdot X)(A(\na u)\cdot \nu)=
(\partial_\nu u)(X\cdot \nu)(A(\na u) \cdot \nu)
&= (X\cdot \nu)(A(\na u )\cdot \na u) 
    \end{split}
\end{equation}}
Combining this with \eqref{eqn: intermed first var 1} and using $\divv (A(\na u)) = 2 \mu(\mu u-f)$ and the second identity in \eqref{tilde}, we have  
\begin{align}\label{eqn: first var term 1}
\mfrac{d}{dt}\Big|_{t=0}\, \int_{U_t} \Phi(\na u_t) \,dx
&=
 \int_{\G} {(\Phi(\nabla u)- A(\na u )\cdot \na u )}\, X\cdot \nu
-2 \mu \int_U (\mu u-f)\mfrac{d}{dt}\Big|_{t=0} u_t\,.
\end{align}
Finally, we now differentiate the last term in the energy:
\begin{align}\label{eqn: first var term 3}
    \mfrac{d}{dt}\Big|_{t=0} \int_{\Omega}(\mu u_t-f)^2\,dx = 2 \mu \int_{\Omega} (\mu u_t - f)  \mfrac{d}{dt}\Big|_{t=0} u_t(x) \,dx = 2 \mu \int_{U} (\mu u_t - f)  \mfrac{d}{dt}\Big|_{t=0} u_t(x) \,dx 
\end{align}
In the last identity we use that $\frac{d}{dt}|_{t=0} {u}_t = 0$ on $\Omega\setminus U$. Combining \eqref{firstterm}, \eqref{eqn: first var term 1}, and \eqref{eqn: first var term 3} completes the proof.
\end{proof}

Now we compute the second variation of $\Jfm$ at a critical point.
\begin{lemma}\label{lem: second var}
Let $u$, $u_t$ be as in the setup above, and assume further that ${ A(\na u )\cdot \na u -\Phi(\na u) = 1}$ on $\G.$ Then  
\begin{equation}
    \label{eqn: second deriv}
\mfrac{d^2}{dt^2}\Big|_{t=0}\Jfm(u_t)=\int_U Q(\nabla u)[\nabla \dot{u},\nabla \dot{u}]
+2\mu^2 \dot{u}^2 \,dx 
-\int_\Gamma  \big\langle Q(\na u)D^2u[\nu ], \nu\big\rangle  \frac{1 +\Phi(\na u)}{ (A(\na u )\cdot \nu)} \, (X\cdot \nu)^2.
\end{equation}
where as above $\dot{u} = \frac{d}{dt}|_{t=0} u_t$.
If $u$ is a minimizer of $\Jfm$, then for all $\psi \in C^\infty(\Omega)$,
\begin{equation}\label{eqn: appendix stability}
\int_U Q(\nabla u)[\nabla \psi,\nabla \psi]+2\mu^2\psi^2 \,dx \geq
\int_\Gamma\big\langle D^2\Phi(\na u)D^2u[\nu ], \nu \big\rangle \frac{ (D \Phi(\na u )\cdot \nu)}{1 +\Phi(\na u)}\psi^2 \, \mathcal{H}^{n-1}.
\end{equation}
\end{lemma}
\begin{proof} Let $\nu_t$ be the outer unit normal to $U_t$,  $
G_t(x)= {1 +\Phi(\nabla u)- A(\na u )\cdot \na u }$
and $V_t=X(x)\cdot \nu_t(x)$. We can then express \eqref{eqn: first var} as 
\begin{equation}\label{linearize}
\mfrac{d}{dt}\, \Jfm(u_t)
=
\int_{\Gamma_t}G_t(y) V_t(y )\,  d\mathcal H^{n-1}(y) = 
\int_\Gamma G_t(\FL_t(x))\cdot V_t(\FL_t(x))\, J^\Gamma \FL_t(x)\, d\mathcal H^{n-1}
\end{equation} 
for all $t\in (-\e, \e)$. In the second identity we use the area formula on hypersurfaces (see \cite[Theorem 11.3]{MaggiBook}).
Keeping in mind that \(G_0(\FL_0(x))=G_0(x)=0\) on \(\Gamma\), the chain rule gives
\begin{equation}\label{secvar}
\mfrac{d^2}{dt^2}\Big|_{t=0}\Jfm(u_t)
 =\int_\Gamma \mfrac{d}{dt}\Big|_{t=0} G_t(\FL_t(x))\, V_0 .
\end{equation}
To compute $\frac{d}{dt}|_{t=0} G_t(\FL_t)$, recall from \eqref{tilde} and \eqref{dertilde} that
$\frac{d}{dt}|_{t=0}\nabla u_t|_{\FL_t(x)}=- DX^T\nabla u + \nabla w$. Just as before, here we let \(\tilde u_t(x)=u_t\circ \FL_t\) and \(w=\frac{d}{dt}|_{t=0}\tilde u_t\).
So, applying the chain rule, we find {
\begin{equation}\label{11}
\begin{split}
\mfrac{d}{dt}\Big|_{t=0}G_t(\FL_t(x)) &=
\mfrac{d}{dt}\Big|_{t=0}\Phi(\nabla u_t|_{\FL_t(x)}) - \big\langle A(\na u) , \mfrac{d}{dt}\Big|_{t=0}\nabla u_t|_{\FL_t(x)} \big\rangle - \Big\langle Q(\na u)\Big[ \mfrac{d}{dt}\Big|_{t=0}\nabla u_t|_{\FL_t(x)}\Big], \na u\Big\rangle\\
&
=- \Big\langle Q(\na u)\Big[ \mfrac{d}{dt}\Big|_{t=0}\nabla u_t|_{\FL_t(x)}\Big], \na u\Big\rangle = - \big\langle Q(\na u)[ -DX^T\nabla u+\nabla w], \na u\big\rangle .
\end{split}
\end{equation}}
Differentiating both sides of the second identity of \eqref{tilde} with respect to $x$, we find
\[
\nabla w=\nabla \dot{u}+\nabla(\nabla u\cdot X)
=\nabla \dot{u}(x)+D^2uX(x)+DX^T\nabla u.
\]
Thus, since \(X=V_0 \nu\) and \(\nabla u(x)=(\partial_\nu u)\nu\) for \(x\in \Gamma\), 
\eqref{11} becomes
\[
\mfrac{d}{dt}\Big|_{t=0}G_t(\FL_t(x))
= - \big\langle Q(\na u)[ \nabla \dot{u}+D^2u[X]], \na u\big\rangle 
=- \big\langle Q(\na u)[ \nabla \dot{u}], \na u\big\rangle 
- \big\langle Q(\na u)D^2u[\nu ], \na u\big\rangle  V_0
\]
for \(x\in \Gamma\).
Thus by \eqref{secvar},
\begin{equation}\label{secvar2}
\mfrac{d^2}{dt^2}\Big|_{t=0} \Jfm(u_t)
=-\int_\Gamma  \langle Q(\na u)[ \nabla \dot{u}], \na u\big\rangle  V_0
-\int_\Gamma  \big\langle Q(\na u)D^2u[\nu ], \na u\big\rangle  V_0^2.
\end{equation}

Next, we use the fact that $\dot{u}$ solves the equation
\begin{equation}\label{secvar3}
\begin{cases}
-\operatorname{div}(Q(\nabla u)\nabla \dot{u})+2\mu^2 \dot{u}=0 & \text{on } U,\\
\dot{u}=-\partial_\nu u\,V_0 & \text{on } \Gamma,\\
\dot{u}=0 & \text{on } \partial\Omega\cap \partial U .
\end{cases}
\end{equation}
Indeed, the interior equation holds by differentiating the equation $-\divv(A(\nabla u_t))+2\mu (\mu u_t-f)=0$ so
\[
0=\mfrac{d}{dt}\Big|_{t=0}\bigl(-\divv(A(\nabla u_t))+2\mu (\mu u_t-f)\bigr)
=
-\operatorname{div}\bigl(Q(\nabla u)\nabla \dot{u}\bigr)+2\mu^2 \dot{u}
\]
with respect to $t$.
To see the boundary condition on \(\Gamma\) in \eqref{secvar3}, we again use that \(w=0\) on \(\Gamma\)
and apply the second identity in \eqref{tilde}. Multiplying the equation by $\dot{u}$ and integrating by parts in  $U$, we find
\begin{equation}\label{secondterm}
    \begin{split}
   0 & =\int_U Q(\nabla u)[\nabla \dot{u},\nabla \dot{u}]+2\mu^2 (\dot{u})^2 \,dx- \int_\G \dot{u} Q(\nabla u)[\nabla \dot{u},\nu]\,d\Hm^{n-1}\\
        & =\int_U Q(\nabla u)[\nabla \dot{u},\nabla \dot{u}]+2\mu^2 (\dot{u})^2 \,dx+ \int_\G V_0 \, Q(\nabla u)[\nabla \dot{u},\na u]\,d\Hm^{n-1}   \,.   
    \end{split}
\end{equation}
So putting \eqref{secondterm}  into \eqref{secvar2} yields
$$
\mfrac{d^2}{dt^2}\Big|_{t=0}\Jfm(u_t)
=\int_U Q(\nabla u)[\nabla \dot{u},\nabla \dot{u}]
+2\mu^2 \dot{u}^2 \,dx 
-\int_\Gamma  \big\langle Q(\na u)D^2u[\nu ], \na u\big\rangle  V_0^2.
$$
This shows \eqref{eqn: second deriv}. Next, if $u$ is a minimizer, we have $\frac{d^2}{dt^2}|_{t=0}\Jfm(u_t) \geq 0$. Since \(\dot{u}\) has the least energy
of all extensions of \(-\partial_\nu u\,V_0\) into
\(U\), if we let \(\psi\in C_c^\infty(\Omega)\) be
any extension of \(-\partial_\nu u\,V_0\in C_c^\infty(\Gamma)\), then
the non-negativity of $\frac{d^2}{dt^2}|_{t=0}J(u_t)$ guarantees that
\begin{equation}\label{temp}
\int_U Q(\nabla u)[\nabla \psi,\nabla \psi]+2\mu^2\psi^2 \,dx 
\ge
\int_\Gamma 
 \big\langle Q(\na u)D^2u[\nu ], \na u\big\rangle \frac{\psi^2}{(\partial_\nu u)^2}
=
\int_\Gamma \big\langle Q(\na u)D^2u[\nu ], \nu \big\rangle \frac{ (A(\na u )\cdot \nu)}{1 +\Phi(\na u)}\psi^2
\end{equation}
where we used that $\partial_\nu u (A(\na u )\cdot \nu) =  A(\na u )\cdot \na u =1 +\Phi(\na u)$.
\end{proof}

\begin{remark}\label{rmk: 2 hom}
    \rm{ 
    In the computations above, we assumed that $\Phi$ was smooth, which we are free to do in our proof via a mollification argument. 
    However, if $u$ is a minimizer of $\J(\cdot, \Omega)$ with free boundary of class $C^{2,\alpha}$, where $\Phi \in \II$ is an even and two-homogeneous integrand (hence not twice differentiable at the origin), the computations above go through and the variations take a cleaner form.

Indeed, by Lemma~\ref{lem: first var}, which does not require $C^2$ regularity of $\Phi$, and using Euler's identity $\na \Phi (p) \cdot p = 2\Phi(p)$, the Euler Lagrange equation for $u$ becomes
    \begin{equation}\label{eqn: EL 2hom}
    \begin{cases}
        \divv(D\Phi(\na u))= 0 & \text{ in  } \{ u>0\} \cap \Omega\\
        \Phi(\na u) = 1 & \text{ on } \partial \{u>0\} \cap \Omega.
    \end{cases}
 \end{equation}
   Making further use of Euler's identity, the stability inequality \eqref{eqn: appendix stability} takes the form 
 \begin{equation}\label{eqn: 2hom stability 1}
    \int_{\partial \{u>0\}} D^2u[D \Phi (\nabla u),\nu]\,\varphi^2 \,\Phi(\nu) \, \Hm^{n-1} \leq \int_{\{u>0\}} D^2\Phi(\na u)[\na \varphi, \na \varphi]  \,dx \qquad \text{ for all } \varphi\in C^\infty_c(\Omega).
\end{equation}
Here the integrand on the right-hand side is defined a.e, since the critical set $\mathcal{C}=\{ |\na u|=0\}\cap \{u>0\}$ has Lebesgue measure zero, see e.g. \cite[Prop. 3.2]{DMMN}.
 The proof of Lemma~\ref{lem: second var} goes through in this setting because the free boundary condition $\Phi(\na u)=1$ together with the $C^{2,\alpha}$ regularity of the free boundary allow us to localize the variation to a neighborhood of the free boundary, excluding degeneracy of the gradient. 
 }
     
\end{remark}

\begin{remark}\label{rmk: aniso MC} \rm{ Continuing to assume that $\Phi \in \II$ is even and $2$-homogeneous and $u$ is a minimizer of $\J(\cdot, \Omega)$ with free boundary of class $C^{2,\alpha}$, we can further rewrite the stability inequality \eqref{eqn: 2hom stability 1} to take an analogous form to the usual form of the stability inequality for the classical Alt-Caffarelli energy; see for instance \cite[Eqn (2.4)]{JS}.

Indeed, let $F = \sqrt{\Phi}$, and let $H_\Sigma^F(x)$ denote the anisotropic $F$-mean curvature of $\G$ arising as the first variation of the surface energy $\int_{\partial U} F(\nu_{U}) \, \partial \Hm^{n-1}$, oriented with respect to the inner unit normal. Then, making use of the level set identity for the anisotropic mean curvature \cite[Eqn (3.11)]{DMMN} and Euler's identity once again, the left-hand side of  \eqref{eqn: 2hom stability 1} can be rewritten so that \eqref{eqn: 2hom stability 1} takes the form
\begin{equation}\label{eqn: 2hom stability 2}
    \int_{\partial \{u>0\}} H_\G^F \varphi^2 F(\nu) \,d \Hm^{n-1} \leq \int_{\{u>0\}} D^2(F^2/2)(\na u)[\na \varphi, \na \varphi]  \,dx \qquad \text{ for all } \varphi\in C^\infty_c(\Omega)\,.
\end{equation} 
} 
\end{remark}

\section{Proof of the convergence lemma}\label{sec: cptness proof}

In this Appendix, we prove the compactness properties of minimizers stated in Lemma \ref{lem:compactness}. For the classical Alt-Caffarelli energy, this result is rather standard. Here the main differences are the varying integrands and the lower-order perturbations. We sketch the main ingredients of the proof.
 
\begin{proof}
To simplify our notation, we set \(F:=\sup_k\|f_k\|_{L^\infty(B_1)}<\infty\), 
\(\mathcal{E}_k:=\mathcal{E}^{\Phi_k}_{f_k,\mu_k}\),
\(\mathcal{E}:=\Jfm\) and 
\[
    \omega_k(M):=
    \sup_{|p|\leq M}|\Phi_k(p)-\Phi(p)|
    \longrightarrow0
    \qquad\text{for each fixed }M<\infty.
\]
As a result, for each $s\in(0,1)$ and  \(w_k\in W^{1,\infty}(B_s)\) with
\(\|\na w_k\|_{L^\infty(B_s)}\leq M\), we have
\begin{equation}\label{eq:cv-bounded-grad}
    \left|
    \int_{B_s}\bigl(\Phi_k(\na w_k)-\Phi(\na w_k)\bigr)\,dx
    \right|
    \leq |B_s|\,\omega_k(M)\longrightarrow0.
\end{equation}

For \(s\in(0,1)\), locally uniform convergence of the sequence $\{u_k\}$ and  Lemma~\ref{lem:interior-lip}
give
\begin{equation}\label{eq:cv-lip}
    \sup_k\bigl(
        \|u_k\|_{L^\infty(B_s)}
        +\operatorname{Lip}(u_k;B_s)
    \bigr)\leq L, \qquad 
    \|u\|_{L^\infty(B_s)}
        +\operatorname{Lip}(u;B_s)\leq L.
\end{equation}

Moreover, uniform convergence of $\{f_k\}$ gives a modulus of continuity $\omega_f$ such that  for \(y\in B_s\) and \(D\in(0,1)\),
\begin{equation}\label{eq:cv-fmod}
    |f_k((1-D)y)-f_k(y)|
    \leq2\|f_k-f\|_{L^\infty(B_1)}+\omega_f(D)
    =o_k(1)+\omega_f(D) .
\end{equation}
 
\medskip
\noindent\emph{Step 1: The dilation competitor.}
In this step we show that  for all \(0<\rho<s<1\), if  \(v-u=0\) on \(\partial B_\rho\), then
\begin{equation}\label{eq:cv-dilation}
    \limsup_{k\to\infty}\,\mathcal{E}_k(u_k,B_s)
    \leq
    \mathcal{E}(v,B_s) .
\end{equation}

Fix $
    0<D<\min\{\frac12,1-\frac\rho s\}$.
 Let
\(\ell:[(1-D)s,s]\to[0,1]\) be the affine function with \(\ell((1-D)s)=1\),
\(\ell(s)=0\), so \(|\ell'|=1/(Ds)\). Writing \(\pi(x):=s\,x/|x|\) for the
radial projection onto \(\partial B_s\), define
\[
    w_k(x):=
    \begin{cases}
        v\bigl(\tfrac{x}{1-D}\bigr), & |x|\leq(1-D)s,\\[2pt]
        \ell(|x|)\,u(\pi(x))+\bigl(1-\ell(|x|)\bigr)u_k(\pi(x)),
        & (1-D)s<|x|<s,\\[2pt]
        u_k(x), & |x|\geq s .
    \end{cases}
\]
Since \(w_k\) is an admissible
competitor for \(u_k\), we have
\begin{equation}\label{eq:cv-min-wk}
    \mathcal{E}_k(u_k,B_s)\leq\mathcal{E}_k(w_k,B_s) .
\end{equation}

We estimate \(\mathcal{E}_k(w_k,B_s)\). Set \(t=1-D\).
For every \(p\in\mathbb R^n\), Hessian bound on $\Phi_k$ gives
$$
    |\Phi_k(t^{-1}p)-\Phi_k(p)|
    \leq
    \int_1^{t^{-1}}|D\Phi_k(\tau p)|\,|p|\,d\tau\leq
    \frac{\La}{2}t^{-2}|p|^2
    \leq C D|p|^2.
$$
Therefore, with \eqref{eq:cv-fmod}, we find
$$
    \int_{B_{ts}}\Phi_k(\na w_k)\,dx
    =t^n\int_{B_s}\Phi_k(t^{-1}\na v)\,dy\leq
    \int_{B_s}\Phi(\na v)\,dy
    +CD\int_{B_s}|\na v|^2\,dy+o_k(1),
    $$
    $$
    |\{w_k>0\}\cap B_{ts}|
    =t^n|\{v>0\}\cap B_s|
    \leq|\{v>0\}\cap B_s|,
    $$

and 
$$
    \int_{B_{ts}}|\mu_kw_k-f_k|^2\,dx
    =t^n\int_{B_s}|\mu_kv(y)-f_k(ty)|^2\,dy\leq
    \int_{B_s}|\mu v-f|^2\,dy
    +o_k(1)+C_v\omega_f(D).
$$

In the annulus
\(A_D:=B_s\setminus B_{(1-D)s}\), whose volume is at most \(C(n)D\), we
compute
\[
    \na w_k
    =\ell'(|x|)\,\frac{x}{|x|}\,
    \bigl(u(\pi(x))-u_k(\pi(x))\bigr)
    +\ell\,\na\bigl(u\circ\pi\bigr)
    +(1-\ell)\,\na\bigl(u_k\circ\pi\bigr).
\]
 By \eqref{eq:cv-lip}, we have
\[
    |\na w_k|
    \leq\frac{\eta_k}{Ds}+CL,
    \qquad
    \eta_k:=\|u-u_k\|_{L^\infty(\overline{B_s})}\longrightarrow0 .
\]
Hence, bounding \(\Phi_k(p)\leq C(\La)|p|^2\), the characteristic function
by \(1\), and the fidelity term by \(C(1+F^2)\), we have $ \mathcal{E}_k(w_k,A_D)
    \leq C(\frac{\eta_k^2}{D}+D).$Combining these estimates with \eqref{eq:cv-min-wk} gives
$$
    \mathcal{E}_k(u_k,B_s)
    \leq
    \mathcal{E}(v,B_s)
    +C_v\bigl(D+\omega_f(D)\bigr)
    +C\frac{\eta_k^2}{D}
    +o_k(1).
$$
Taking first $\limsup_{k\to\infty}$ at fixed $D$, and then
letting $D\downarrow0$, we obtain \eqref{eq:cv-dilation}.
 
\medskip
\noindent\emph{Step 2: Lower semicontinuity.}
For every \(s\in(0,1)\), we show that
\begin{equation}\label{eq:cv-lsc}
    \mathcal{E}(u,B_s)\leq\liminf_{k\to\infty}\mathcal{E}_k(u_k,B_s) .
\end{equation}

Indeed, \eqref{eq:cv-lip} and \eqref{eq:cv-bounded-grad} give $|
    \int_{B_s}\bigl(\Phi_k(\na u_k)-\Phi(\na u_k)\bigr)\,dx|
    \leq |B_s|\omega_k(L)\longrightarrow0.$
Convexity and weak $H^1(B_s)$ convergence then imply
$$
    \liminf_k\int_{B_s}\Phi_k(\na u_k)\,dx
    \geq\int_{B_s}\Phi(\na u)\,dx.
$$
For the measure term, 
Fatou's lemma gives
\(\liminf_k|\{u_k>0\}\cap B_s|\geq|\{u>0\}\cap B_s|\). 
Finally, the fidelity term converges, since \(u_k\to u\) in \(L^2(B_s)\), \(f_k\to f\)
uniformly, and \(\mu_k\to\mu\).
 
\medskip
\noindent\emph{Step 3: proof of (2).}
Let \(0<\rho<s<1\) and let \(v-u\in H^1_0(B_\rho)\). Combining
\eqref{eq:cv-lsc} and \eqref{eq:cv-dilation}, we have
\[
    \mathcal{E}(u,B_s)
    \leq\liminf_{k\to\infty}\mathcal{E}_k(u_k,B_s)
    \leq\limsup_{k\to\infty}\mathcal{E}_k(u_k,B_s)
    \leq\mathcal{E}(v,B_s) .
\]
Since \(v=u\) on \(B_s\setminus B_\rho\), subtracting the
common finite contribution on that annulus gives
\begin{equation}\label{eq:cv-localmin}
    \mathcal{E}(u,B_\rho)\leq\mathcal{E}(v,B_\rho)
    \qquad
    \text{whenever }v-u\in H^1_0(B_\rho),\quad \rho<1.
\end{equation}
Thus \(u\) is a local minimizer in \(B_1\).

\medskip
\noindent\emph{Step 4: Proof of (1).}
Fix \(s\in(0,1)\). Taking \(v=u\) in \eqref{eq:cv-dilation} and combining
with \eqref{eq:cv-lsc},
\[
    \mathcal{E}_k(u_k,B_s)\longrightarrow\mathcal{E}(u,B_s) .
\]
In particular, using \eqref{eq:cv-bounded-grad} and \eqref{eq:cv-lip},
\(\int_{B_s}\Phi(\na u_k)\to\int_{B_s}\Phi(\na u)\). The uniform
convexity of \(\Phi\) gives
\[
    \int_{B_s}|\na u_k-\na u|^2dx
    \leq
    C\int_{B_s}\Bigl[\Phi(\na u_k)-\Phi(\na u)
    -D\Phi(\na u)\cdot(\na u_k-\na u)\Bigr]dx
    \longrightarrow0.
\]

\medskip
\noindent\emph{Step 5: proof of (3) and (4).}
Now that the limit $u$ is an energy minimizer, it satisfies the Lipschitz bound as well as non-degeneracy properties. The convergence of the positive sets as well as the free boundaries follow from standard arguments; see, e.g., \cite{V}.

\medskip
\noindent\emph{Step 6: proof of (5).}
For notational simplicity we prove the claim with \(r=1\). We may assume that the left-hand side of (5) is
finite and, passing to a subsequence realizing the \(\liminf\), that
\[
    \sup_k\|D^2u_k\|_{L^p(\{u_k>0\}\cap B_1)}\leq C_0<\infty .
\]
Set \(\Omega^+:=B_1\cap\{u>0\}\).
For \(0<\e<1\), let
\[
    V_\e:=
    \left\{
        x\in\Omega^+:
        \operatorname{dist}
        (x,\mathbb R^n\setminus\Omega^+)>\e
    \right\}.
\]
Then \(\overline{V_\e}\Subset\Omega^+\), and
\(V_\e\uparrow\Omega^+\) as \(\e\downarrow0\).
If \(V_\e=\emptyset\), the desired inequality on \(V_\e\)
is trivial.
Otherwise \(m_\e:=\min_{\overline{V_\e}}u>0\), and locally
uniform convergence gives \(u_k\geq m_\e/2\) on \(V_\e\)
for all sufficiently large \(k\).
In particular, \(V_\e\subset\{u_k>0\}\) and
\(\|D^2u_k\|_{L^p(V_\e)}\leq C_0\). Let \(q=p/(p-1)\). For
\(\varphi\in C^\infty_c(V_\e)\) and any \(i,j\), the weak \(H^1\)
convergence and integration by parts give
\[
    \int_{V_\e}\partial_iu\,\partial_j\varphi\,dx
    =\lim_{k\to\infty}\int_{V_\e}\partial_iu_k\,\partial_j\varphi\,dx
    =-\lim_{k\to\infty}\int_{V_\e}\partial_{ij}u_k\,\varphi\,dx ,
\]
and the last integrals are bounded by \(C_0\|\varphi\|_{L^q(V_\e)}\). By
duality, \(u\) admits weak second derivatives
\(\partial_{ij}u\in L^p(V_\e)\). Moreover, since \((\partial_{ij}u_k)_k\) is
bounded in \(L^p(V_\e)\) and its distributional limit is
\(\partial_{ij}u\), we have \(D^2u_k\rightharpoonup D^2u\) componentwise,
weakly in \(L^p(V_\e)\). By weak lower semicontinuity of the convex
functional \(w\mapsto\int|w|^p\),
\[
    \int_{V_\e}|D^2u|^p\,dx
    \leq\liminf_{k\to\infty}\int_{V_\e}|D^2u_k|^p\,dx
    \leq\liminf_{k\to\infty}\int_{\{u_k>0\}\cap B_1}|D^2u_k|^p\,dx .
\]
Since \(V_\e\uparrow\{u>0\}\cap B_1\) as \(\e\downarrow0\), the monotone
convergence theorem concludes the proof of (5).
\end{proof}

\section{Proof of $\e$-Regularity}\label{sec: e reg proof}
We sketch the proof of Theorem~\ref{thm: eps reg}, emphasizing the differences from \cite{DFS15}. 
The first step is the following partial Harnack inequality; c.f. \cite[Theorem 4.1]{DFS15} \cite[Theorem 3.1]{DeSilva} or \cite[Theorem 7.7]{V}.
\begin{lemma}[Partial Harnack] \label{lem: partial harnack}
   There exist $\bar{\e} =\bar{\e}(n,\La, \alpha)>0$ and $c=c(n,\La)\in (0,1)$ such that the following holds. Suppose $u$ is a viscosity solution to \eqref{eqn: e reg sec EL} in $B_1$ for some $g \in C(B_1)\cap L^\infty(B_1)$ and suppose $0 \in \overline{\{u>0\}}$. If there exist  $e \in \R^n$ with $\cG(e)=1$ and $a_0<b_0$ with $|a_0-b_0|\leq \bar\e$
   such that
   \begin{equation}\label{eqn: linear trapped 1}
   (x \cdot e+a_0 )_+ \leq u(x) \leq 
      (x \cdot e+b_0)_+ \qquad \text{ in }B_1
      \end{equation}
      and
      \begin{equation}\label{eqn: small RHS partial BH}
          \|g\|_{L^\infty(B_1)} \leq |a_0-b_0|^2,
      \end{equation}
   then there are $a_1, b_1$ with $a_0 \leq a_1 <b_1\leq b_0$ and $|b_1-a_1|\leq (1-c)|a_0-b_0|$ such that 
   \[
    (x \cdot e+a_1 )_+ \leq u(x) \leq 
    (x \cdot e+b_1 )_+ \qquad \text{ in }B_{1/64\Lambda}.
   \]
\end{lemma}
\begin{proof} The proof of this lemma only uses that $\mathcal{M}^-_{1/\La,\La}(D^2 u ) \leq g \leq \mathcal{M}^+_{1/\La,\La}(D^2 u )$ in the viscosity sense in $\{u>0\}$ and that $\cG(\na u)=1$ on $\partial \{u>0\}$ in the viscosity sense. The proof essentially differs from \cite[Theorem 4.1]{DFS15} only in the construction of the barrier $w$ below, since here the free boundary condition does not depend on $|\na u|$ only.

Let $\bar{\e} \leq 1/(24\La^2)$ and $\TT ={|e|}/{(24 \Lambda})$. Note that $|e| \in [ 1/\Lambda,\Lambda]$ (as $\cG(e)=1$) and thus $\bar{\e} \leq \TT$.
   Moreover, since $0\in\overline{\{u>0\}}$, we know $b_0 \geq 0$ and thus $a_0 \geq -\TT.$ We consider two cases: $a_0\geq \TT$ and $|a_0|\leq \TT.$

{\it Case 1:} First suppose $a_0 \geq \TT$. Then $u$ satisfies $\mathcal{M}^-_{1/\La,\La}(D^2 u ) \leq g \leq \mathcal{M}^+_{1/\La,\La}(D^2 u )$ in the viscosity sense in $B_{1/(24\La)} \subset B_1 \cap \{x \cdot e > -\TT\}$.  We let
\begin{equation}
    \label{eqn: ell def}
\ell(x) = x \cdot e+a_0 \qquad \text{ and }\qquad \e = b_0-a_0,
\end{equation} 
so that the assumption \eqref{eqn: linear trapped 1} implies $\ell(x) \leq u(x) \leq \ell(x)+\e$ in $B_{1/(24\La)}$. If $u(0) \geq \ell(0) + \e/2,$ let $v=u-\ell$.
In $B_{1/(24\La)}$, the function $v$ is nonnegative and again satisfies $\mathcal{M}^-(D^2 v, 1/\La,\La) \leq g \leq \mathcal{M}^+(D^2 v, 1/\La,\La) $ in the viscosity sense.
So, by Harnack inequality (see \cite[Theorem 4.3]{CaffCabre}) there is a constant $C =C(n,\Lambda)>0$ such that 
\[
\frac{\e}{2} \leq v(0) \leq C\Big(\inf_{B_{1/(48\La)}}v + \mfrac{1}{(24\La)^2}\|g\|_{L^\infty}\Big) \leq C\Big(\inf_{B_{1/(48\La)}} v + \mfrac{1}{(24\La)^2}\e^2\Big) 
\]
Choosing $\bar{\e}=\bar{\e}(n,\Lambda)$ small enough,  $C/{(24\La)^2} \e^2\leq \e/4$, so absorbing this and rearranging terms, we get 
\[
\ell(x) + \mfrac{\e}{4C}  \leq u(x) \qquad \text{ in }B_{1/(48\La)},
\]
which yields the conclusion with $a_1 = a_0 + \e/(4C)$ and $b_1=b_0$ in this case.
If instead $u(0) \leq \ell(0) + \e/2$, we argue analogously with the function $w = \ell +\e-u$ to find $
u(x) \leq \ell(x) + (1-\mfrac{1}{4C})\e$ in $B_{{1}/{(48 \La)}}.$
This yields the conclusion with $b_1 = b_0 -\e/(4C)$ and $a_1=a_0.$

{\it Case 2:} Now suppose $|a_0| \leq \TT.$ Let $\bar{x} =  \frac{e}{3\Lambda|e|}$. For the linear function $\ell$ and number $\e$ defined in \eqref{eqn: ell def} above, first consider the case when $u(\bar{x}) \geq \ell(\bar{x}) +\e/2.$ In this case, we claim there exists $c=c(n,\La)>0$ such that
\begin{align}
    \left(x \cdot e+ a_0 + c\e \right)_+ \leq u(x)& \qquad \text{ for all $x$ in } B_{1/(64\Lambda)}\,.
\end{align}
 This will prove the claim with $a_1= a_0+ c\e$ and $b_1 =b_0.$ To this end, first note that $x\cdot e>\TT$ for all $B_{1/(4\Lambda)}(\bar{x})$, so this ball is contained in $\{u>0\}$.  Applying Harnack inequality to the function $v= u-\ell$ just as in Case 1, we find 
\[
\frac{\e}{2} \leq v (\bar{x}) \leq  C\Big(\inf_{B_{1/(8\Lambda)}(\bar{x})} v + \mfrac{1}{16\Lambda^2}\|g\|_{L^\infty}\Big) \leq C \inf_{B_{1/(8\Lambda)}(\bar{x})} v + \frac{\e}{4}
\]
provided $\bar{\e}=\bar{\e}(n,\Lambda)$ is chosen small enough. Rearranging, this means
\begin{equation}\label{eqn: aa}
   u \ge \ell + c_*\e \qquad \text{ in } B_{1/(8\Lambda)}(\bar{x}) 
\end{equation}
where $c_* = 1/(4C)$. Now, fix the matrix $M_e = (D^2\Phi(e))^{-1}.$ Then letting $q(x) = M_e[x-\bar{x},x-\bar{x}],$ and $\gamma =64n\Lambda^7$, consider the barrier $w$ defined piecewise by
\[
w(x) = \begin{cases}
    1 & \text{ if } q(x) < 1/(64\Lambda^3)\\
    0 & \text{ if }  q > 1/(4\Lambda) \\
    \bar{c}(e^{-\gamma q(x) } - e^{-\gamma/(4\Lambda)}) & \text{ if  }1/(64\Lambda^3) < q <1/(4\La)\,.
\end{cases}
\]
Here $\bar{c}= \bar{c}(n,\Lambda)$ is chosen to make $w$ continuous. Let $A(\bar{x})$ denote the annular region $A(\bar{x}) :=\{1/(64\Lambda^3 )<q(x) < 1/(4\Lambda)\}$ in which $w$ is nonconstant. Using that $|x-\bar{x}|^2/\Lambda \leq q(x)\leq \Lambda |x-\bar{x}|^2$, we see that $ \{ q < 1/(4\Lambda) \}\subset B_{1}$ and $\{ q <1/(64\La^3)\} \subset B_{1/(8\Lambda)}(\bar{x})$. Moreover,
we have
\begin{align}\label{eqn: grad and hessian}
    \na w(x) = -2\bar{c}\gamma e^{-\gamma q(x) } M_e(x-\bar{x}) \qquad \text{ and } \qquad 
    D^2 w(x) =2\bar{c}\gamma e^{-\gamma q(x) } \big( 2\gamma M_e(x-\bar{x})\otimes M_e(x-\bar{x})  - M_e\big)
\end{align} 
In particular, for 
$x \in A(\bar{x}) \cap \{x \cdot e \leq \TT\}$, we have
\begin{align}\label{eqn: pos partial}
  \na w(x) \cdot D \cG(e) =- 2\bar{c}\gamma e^{-\gamma q(x) }  (x-\bar{x})\cdot e \geq c(n,\Lambda) (-\tau + |e|/(3\Lambda)) \geq c(n,\La).
\end{align}
Here we have used that $D\cG(p) = D^2\Phi(p)[p]$ for any $p$ for which $D^2\Phi(p)$ exists. 
Now, fix $t \in (0,1)$ and let 
$$v_t(x) = \ell(x) + c_*\e w(x) -tc_*\e.$$
Noting that $D^2 v_t = c_*\e D^2 w,$ we have 
\begin{equation}\label{eqn: subsoln}
\begin{split}
\mathcal{M}_{1/\La, \La}^-(D^2v_t(x)) 
&\geq 
c_*\e 
\mathcal{M}_{ 1/\La, \La}^-(D^2 w(x))\geq 2c_*\e \bar{c}\gamma e^{-\gamma q(x) }   \Big(\mfrac{2\gamma q(x)}{\Lambda^2} - n\Lambda^2\Big) \geq c(n,\Lambda)\e
\end{split}
\end{equation}
for $x \in A(\overline{x})$, thanks to our choice of $\gamma$.
Choosing $\bar{\e}$ small enough, \eqref{eqn: subsoln} and \eqref{eqn: small RHS partial BH} guarantee $\mathcal{M}_{1/\La, \La}^-(D^2v_t(x))  >g$ in $A(\bar{x})$.
We claim that $u>v_t$ in $A(\bar{x})$. Suppose not.
Noting that $u >v_t$ on $B_{1/(8\La)}(\bar{x})$ by \eqref{eqn: aa} and on $B_1\setminus \{q < 1/(4\La)\}$ by construction, this means that $u-v_t$ has a local minimum at some $x$ in $A(\bar{x})$, so a vertical translation of $v_t$ touches $u$ from below at this $x$. We note that $x$ cannot lie in $A(\bar{x}) \cap \{u>0\}$ since  $\mathcal{M}_{1/\La, \La}^-(D^2u(x))\leq g$  and  $\mathcal{M}_{1/\La, \La}^-(D^2v_t(x))  >g(x)$ in the viscosity sense  in this set. Next, $x$ cannot lie in $A(\bar{x})\cap \GU$, since
\[
\cG(\na v_t) =  \cG(e) +  D\cG (e) \cdot (\na v_t-e) + o(\e) = 1+ c_* \e \na w(x)\cdot D^2 \Phi(e)[e] + o(\e) >1
\]
for $\e$ small enough,
thanks to \eqref{eqn: pos partial}, again violating $u$ being a viscosity solution. 
Finally, we cannot have $x \in A(\bar{x}) \cap \{u=0\}$, since at such a point we would have $\na (u-v_t)(x)=-e-c_*\e \na w(x) =0$, contradicting \eqref{eqn: pos partial} again.

So, $u >v_t$ in $A(\overline{x})$ for all $t \in (0,1)$. Hence, taking $t\to0$, we have $u \geq \ell + c_* \e w$ in $A(\overline{x})$. In particular, since $w$ is bounded below in $B_{\sfrac{1}{(64\La)}}$, this establishes the claim in this case. If instead $u(\bar{x} )\leq \ell(\bar{x}) + \e/2,$ an analogous proof shows $(x\cdot e + b_0 - c\e)_+ \geq u(x)$
for all $x$ in $B_{\sfrac{1}{(64\Lambda)}}$.
This completes the proof.
\end{proof}
By iterating Lemma~\ref{lem: partial harnack}, we find the following:
\begin{corollary}\label{cor: iof convergence}
    Let $\Phi_k \in \II$ with $\Phi_k \to \Phi$, $g_k \in C(B_1)\cap L^\infty(B_1)$, and $\e_k\to 0$. For each $k$, suppose $u_k$ satisfies \eqref{eqn: e reg sec EL} with $g_k$ in place of $g$ and with $\Lo =\Lo_{\Phi_k}$, and for $e_k$ with $\cG_k(e_k)=1$ we have
    \[
    (x\cdot e_k - \e_k)_+ \leq u_k(x) \leq (x\cdot e_k + \e_k)_+  \qquad \text{ in }B_1
    \]
    and
    \begin{equation}\label{eqn: small g hyp contra arg}
    \| g_k\|_{L^\infty(B_1)} \leq \e_k^2.
        \end{equation}
    Then let $\tilde{u}_k (x) := (u_k(x) - x\cdot e_k)/\e_k$ and $\text{graph}(\tilde{u}_k):= \{ (x,\tilde{u}_k(x)) : x \in \overline{\Omega}_{u_k} \cap B_{1/2} \}\subset \R^{n+1}$. Then
    there exist $e \in \R^n$ with {$\cG(e)=1$} and  a H\"{o}lder continuous function $\tilde{u}:B_{1/2} \cap \{e \cdot x \geq 0\} \to \R$ such that,
    up to a subsequence, $e_k \to e$,  $\tilde{u}_k \to \tilde{u}$ locally uniformly on $B_{1/2}\cap \{e\cdot x >0\}$ and $\text{graph}(\tilde{u}_k) \to \text{graph}(\tilde{u})$ in the Hausdorff distance. Here $\text{graph}(\tilde{u}) = \{(x, \tilde{u}(x) : x \in B_{1/2} \cap \{ x\cdot e\geq 0\}\}$.
\end{corollary}
The proof is the same as those of \cite[Lemmas 7.13 and 7.14]{V} and is thus omitted.
Following \cite{DeSilva}, the next step toward proving Lemma~\ref{lem: IOF} is the following linearization; c.f. \cite[Lemma 7.15]{V}.
\begin{lemma}\label{lem: lin}
  With the same assumptions and notation as Corollary~\ref{cor: iof convergence}, $\tilde{u}$ has $|\tilde{u}|\leq 1$ in $B_{1/2}\cap\{x\cdot e\geq0 \}$ and solves 
    \begin{equation}
        \label{eqn: linearized eqn}
    D^2 \Phi(e):D^2 \tilde{u} = 0 \text{ in } \{x\cdot e>0\} \cap B_{1/2}, \qquad D^2\Phi(e)[e]  \cdot \na \tilde{u} =0 \text{ on }\{x\cdot e =0\} \cap B_{1/2}.
    \end{equation}
\end{lemma}
\begin{proof}

First we show $\tilde{u}$ solves the equation in the interior. Suppose $\phi \in C^\infty(B_1)$ touches $\tilde{u}$ from below at $x_0 \in \{x\cdot e>0\}\cap B_{1/2}$. We can find $x_k \to x_0$ with $x_k \in \{u_k>0\}\cap B_{1/2}$ and $c_k \to 0$ such that $\phi+c_k$ touches $\tilde{u}_k$ from below locally at $x_k$. At such a point, $Q_k(x) =x\cdot e_k + \e_k(\phi(x)+c_k)$ touches $u_k$ from below at $x_k$. Since $\na Q_k(x_k) \in \An$ for $k$ large, the assumption that $u_k$ is a viscosity solution guarantees that 
\[
D^2 \Phi_k(e_k + \e_k \na \phi(x_k)): D^2\phi(x_k) \leq g_k(x_k)/\e_k.
\]
The right hand side tends to zero thanks to \eqref{eqn: small g hyp contra arg}, while the left-hand side converges to $D^2\Phi(e) : D^2\phi(x_0)$ thanks to the uniform convergence of $D^2\Phi_k$ on $\An$ and the smoothness of $\phi$. Thus $D^2 \Phi(e):D^2\phi(x_0) \leq 0.$ An analogous argument holds for a smooth function $\phi$ touching $\tilde{u}$ from above at $x_0 \in \{ x\cdot e>0\}$. Thus $\tilde{u}$ satisfies $D^2\Phi(e):D^2 \tilde{u} =0$ in $\{ x\cdot e>0\}$ in the viscosity sense; this is a constant coefficient linear equation, so standard regularity theory shows that $\tilde{u}$ is analytic in $\{ x \cdot e>0\}\cap B_{1/2}$ and solves the equation classically. 

Next, suppose $P$ is a quadratic polynomial that touches $\tilde{u}$ locally from below at  $x_0 \in \{x\cdot e =0\}\cap B_{1/2}$. Without loss of generality, we may assume $x_0=0$, and by replacing $P(x)$ with $P(x)-\eta|x|^2-\eta( e \cdot x)+ C(\eta)(x\cdot e)^2$ for $\eta$ small, we may assume that $P$ locally touches $\tilde{u}$ strictly from below at $0$ and that $\mathcal{M}^-(D^2P, 1/\La, \La) \geq 1$ in a small neighborhood of the origin. So, as before, we can find $x_k\in \overline{\{u_k>0\}}$ with $x_k \to 0$ and $c_k\to 0$ such that $P +c_k$ touches $\tilde{u}_k$ from below at $x_k$. 

In terms of $u_k$, this means the polynomial $Q_k = \e_k (P+c_k) + e_k\cdot x $ touches $u_k$ from below at $x_k$ and satisfies $\mathcal{M}^-(D^2Q_k, 1/\La, \La) \geq \e_k$. 
In view of \eqref{eqn: small g hyp contra arg} and the fact that $\Lo_{\Phi_k} u_k =g_k$ in $\{u_k>0\}$ in the viscosity sense, we cannot have $x_k \in \{u_k>0\}$ and hence $x_k \in \G_{u_k}$ for $k$ large enough. At such a free-boundary contact point, $Q_k^+$ touches $u_k$
from below. Hence, as $u_k$ satisfies the free boundary condition in the viscosity sense and using the uniform $C^1$ modulus of continuity, we have
$$1 \geq \cG_k(\na Q_k(u_k)) = \cG_k(e_k) + \e_kD^2\Phi_k(e_k)[e_k] \cdot \na P(x_k) +o(\e_k)= 1 + \e_k D^2\Phi_k(e_k)[e_k]  \cdot \na P(x_k) +o(\e_k).$$
Here we again used that $D\cG(p) = D^2\Phi(p)[p]$ whenever $D^2\Phi(p)$ exists (and in particular for any $p \in \An.$
Therefore, $D^2\Phi_k(e_k)[e_k]  \cdot \na P(x_k) \leq 0$, and so passing to the limit, $D^2\Phi_k(e_k)[e_k]  \cdot \na P(x_0) \leq0$
 We argue analogously if $P$ touches $\tilde{u}$ from above at $0$ to find $D^2\Phi_k(e_k)[e_k]  \cdot \na P(x_0) \geq 0.$ Thus the oblique boundary condition in \eqref{eqn: linearized eqn} in the viscosity sense, and by boundary regularity for viscosity solutions with oblique boundary conditions, see e.g. \cite[Theorem 1.2]{LZoblique}, $\tilde{u}$ is regular up to the boundary and this holds classically.
\end{proof}

From here, another standard arguement yields the following improvement of flatness statement; c.f \cite[Lemma 4.1]{DeSilva}, \cite[Theorem 7.4]{V}. 
\begin{lemma}\label{lem: IOF}
There exist ${\e}_0 = {\e}_0(\La, n) \in (0,1)$, $C = C(\La ,n)>0$, $\sigma(\La, n)\in (0,1)$, and  $r_0 =r_0 (\La, n) \in (0,\sigma^2/2)$, and such that the following holds.
Given $g \in C(B_1)\cap L^\infty(B_1)$, and $e \in \R^n$ with $\cG(e)=1$,
suppose $u$ is a viscosity solution to \eqref{eqn: e reg sec EL} in $B_1$ with   $0 \in \GU$ such that
\begin{align}\label{eqn: IOF flat}
 &   \left(x \cdot e -\e\right)_+ \leq u(x) \leq \left(x \cdot e +\e \right)_+ \qquad \text{in } B_1
\end{align}
and 
\begin{align}\label{eqn: IOF small RHS}
      \|g\|_{L^\infty(B_1)} \leq \e^{2}
\end{align}
 for some $\e \leq \e_0$. Then there is some $e' \in \R^n$ with $\cG(e')=1$ and $|e-e'|<C\e$ such that $u_{r_0}(x) = u(r_0x)/r_0$ satisfies
\begin{equation}\label{eqn: improvement}
(x \cdot e' -\sigma \e)_+ \leq u_{r_0}(x) \leq \left(x \cdot e' +\sigma \e \right)_+ \qquad \text{in } B_{1}\,.
\end{equation}
\end{lemma}
 Finally,  Theorem~\ref{thm: eps reg} follows from Lemma~\ref{lem: IOF} via a standard argument, as in \cite{DeSilva, V}.

\bibliographystyle{acm}
\bibliography{references.bib}

\end{document}